\documentclass[11pt]{article}

\usepackage{amssymb}
\usepackage{amscd}
\usepackage{amsmath}
\usepackage{amsfonts}
\usepackage{amsthm}
\usepackage{xcolor}
\usepackage{enumerate}
\usepackage{enumerate}

\usepackage[pdftex]{hyperref}

\theoremstyle{plain}

\newtheorem{theo}{Theorem}[section] 
\newtheorem{prop}[theo]{Proposition}
\newtheorem{lemme}[theo]{Lemma}
\newtheorem{cor}[theo]{Corollary}
\newtheorem{defin}[theo]{Definition}

\theoremstyle{definition}

\newtheorem{rem}[theo]{Remark}
\newtheorem{exe}[theo]{Example}

\newcommand{\R}{{\mathbb{R}}}
\newcommand{\N}{{\mathbb{N}}}
\newcommand{\C}{{\mathbb{C}}}
\newcommand{\Z}{{\mathbb{Z}}}

\newcommand{\be}{{\beta}}
\newcommand{\al}{{\alpha}}
\newcommand{\la}{{\lambda}}

\newcommand{\si}{{\sigma}}

\newcommand{\ga}{{\gamma}}

\newcommand{\om}{{\omega}}
\newcommand{\Om}{\Omega}

\newcommand{\La}{{\Lambda}}
\newcommand{\Si}{{\Sigma}}
\newcommand{\ep}{\epsilon}

\newcommand{\Ci}{{\mathcal{C}}^{\infty}} 

\newcommand{\op}{\operatorname}
\newcommand{\con}{\overline}
\newcommand{\bigo}{\mathcal{O}}

\newcommand{\biginf}{\mathcal{O}_{\infty}}

\newcommand{\D}{\mathcal{D}}
\newcommand{\symb}{\mathcal{S}}

\newcommand{\lag}{\mathcal{L}} 

\newcommand{\Hilb}{\mathcal{H}}

\newcommand{\wt}{\widetilde}

\newcommand{\pol}{\mathcal{P}}

\newcommand{\lan}{\mathcal{L}}

\newcommand{\kpi}{ \Bigl( \frac{k}{2 \pi} \Bigr)^n}

\newcommand{\EE}{\mathbf{E}}

\newcommand{\Uu}{\rho}
\newcommand{\Vvv}{\tilde \rho}
\newcommand{\Ppp}{p}

\newcommand{\Fam}{\mathcal{F}}

\newcommand{\clr}{\color{black}}
\newcommand{\clb}{\color{black}}

\newcommand{\cB}{\color{black}}
\newcommand{\cA}{\color{black}}

\begin{document}

\title{Landau levels on a compact manifold}
\author{L. Charles}

\maketitle

\begin{abstract} We consider a magnetic Laplacian on a compact manifold, with
  a constant non-degenerate magnetic field. In the large field limit, it is
  known that the eigenvalues are grouped in clusters, the corresponding sums of
  eigenspaces  being called the Landau levels. The first level has been
  studied in-depth as a natural generalization of the K\"ahler quantization. The
  current paper is devoted to the higher levels: we compute their
  dimensions as Riemann-Roch numbers, study the associated Toeplitz algebras
  and prove that each level is isomorphic with a quantization twisted
  by a convenient auxiliary bundle.   
\end{abstract}

\section{Introduction}

The main theme of this paper and its companion  \cite{oim_copain} is the Landau levels of compact
manifolds. For a physicist, the Landau quantization refers to a charged
particle confined to two dimensions and exposed to a magnetic field. It has 
discrete energy levels connected by ladder operators. Besides the planar geometry
considered by Landau \cite{La30}, the case of Riemann surfaces has been investigated in
the context of the quantum Hall effect, see \cite{IeLi94} and references
therein.

On a mathematical point of view, a natural generalisation is the Bochner
Laplacian acting on the sections of a Hermitian line bundle $L$ on a compact
manifold. This Laplacian is defined from two data: a Riemannian metric of the
base and a
connection of the line bundle. The idea underlying this work is that when the curvature of the connection is non-degenerate and large with respect to the metric, the
spectrum of the Laplacian exhibits a structure similar to the Landau
quantization.

More specifically, let us assume that the curvature is related to the Riemannian metric
by a complex structure, and consider the spectrum of the Laplacian of $L^{k}$
in the large $k$ limit. 
In this setting,  Faure-Tsujii
\cite{FaTs15} have shown
that the eigenvalues are grouped in clusters, each of them
representing a generalised Landau level. The first level was previously
identified by Guillemin-Uribe {\cite{GuUr88}} and studied further by
Borthwick-Uribe \cite{BoUr96} as a generalization of K\"ahler quantization. In
particular, its dimension is given  by a Riemann-Roch number and it comes with
an algebra of Toeplitz operators quantizing the classical Poisson algebra. 

Our goal in this paper is to extend these results to the higher Landau levels.
Our main results are:
\begin{enumerate}
\item  the dimension of the $m$-th Landau level is the
Riemann-Roch number of $L^k \otimes F_m$ when $k$ is sufficiently large, where
$F_m$ is the symmetric $m$-th power of the
complex tangent bundle of the base.
\item there is an algebra of Berezin-Toeplitz operators associated to the
  $m$-th Landau level, the symbol of these operators being sections of the
  endomorphism bundle $\op{End} F_m$.
  \item the $m$-th Landau level is isomorphic
with the first Landau level twisted by $F_m$ through a ladder operator, these
isomorphisms are compatible with the Berezin-Toeplitz operators.
\end{enumerate}

The main ingredient to establish these results is an asymptotic expansion of
the Schwartz kernel of the spectral projector of each level. For the first
level, when the complex structure is integrable, the K\"ahler case, this kernel is the Szeg\"o kernel. Its
asymptotic is well-understood since the seminal work by Boutet de
Monvel and Sj\"ostrand \cite{BoSj} and has been
used in numerous papers starting from \cite{BoGu},  \cite{BoMeSc}, \cite{Ze}. In the non-K\"ahler case,
the asymptotics of the first level projector kernel has been obtained by Borthwick-Uribe
\cite{BoUr07} and Ma-Marinescu \cite{MaMa08}. For the higher Landau levels, this asymptotic
expansion will be proved in our second paper \cite{oim_copain}. 

In the current paper,
we will rely on this asymptotic expansion or more generally we will show that
the previous results hold for higher Landau levels defined as the image of any projector whose
Schwartz kernel has the convenient asymptotics. Here the inspiration is the
generalised Toeplitz structure of Boutet de Monvel-Guillemin \cite{BoGu} and our previous work
\cite{oim}, the idea being that the only important feature of the Landau levels is
this asymptotic expansion.

The main tool we will use for the proofs is a particular class of operators
containing the Landau level projectors, the associated Toeplitz operators and
also the generalised ladder operators. The operators in this class are controlled at first order by their symbols, which
are defined as sections of a bundle of non-commutative algebras. Each of these algebras is generated by the spectral projectors and ladder
operators of a Landau Hamiltonian. By this mechanism, the basic properties of
the Landau quantization are transferred to the Bochner Laplacian.

To finish this general introduction, let us mention the two contemporaneous papers
\cite{Yuri_1}, \cite{Yuri_2} by Yuri Kordyukov on the same subject, \cB
 which contain some results on Berezin-Toeplitz operators common with
 ours, cf. Remark \ref{sec:comp_Yuri} for a comparison. However, the
computation of the dimension of the $m$-th Landau level, the ladder operators
and the general operator algebras we use, do not appear in any other work. \clb  
Let us
mention as well that in a related but different context, belonging to homogeneous
microlocal analysis instead of semi-classical analysis, Boutet de
Monvel-Guillemin \cite[Chapter 15]{BoGu}  and Epstein-Melrose  \cite[Chapter 6]{EM} have considered generalised Szegö
projections at higher level with associated Toeplitz algebras, which are similar to our constructions.

\subsection{Magnetic Laplacian} 

\subsubsection*{Constant magnetic intensity} 
Consider a Riemannian manifold $(M,g)$
with a Hermitian line bundle $L $ equipped with a connection $\nabla$.
Associated to these data is a Laplacian $\frac{1}{2} \nabla^* \nabla$ acting on $\Ci ( M,
L)$, which from the physical point of view is a  Schr\"odinger operator with a
magnetic field $\Om = i \op{curv} ( \nabla) \in \Om^2 (M, \R)$.

We will assume
that $\Om$ is non-degenerate at each point and has a constant magnetic
intensity with respect to $g$ in
the following sense. In the case
where $M$ is a surface, the magnetic intensity is the positive function
defined by $ |\Om| = B \op{vol}_g$, where $\op{vol}_g $ is the Riemannian
volume, and we merely assume that $B$ is constant. In higher dimension, $\Om$ being
non-degenerate, the dimension of $M$ is even, say $2n$. At any $p \in M$,
there exists a skew-symmetric endomorphism  $j_B (p)$ of $(T_pM, g_p)$ such that $\Om_p (X,Y) = g_p
( j_B(p) X, Y)$.  The eigenvalues of $j_B(p)$ are $ \pm i B_{\ell} (p)$ with $0 <
B_1(p) \leqslant \ldots \leqslant B_n (p)$. We assume that these eigenvalues
are all equal, 
$B_1 = \ldots =  B_{n}$, and do not depend on $p$. Equivalently $j_B(p) = B j(p)$ with $B$ a positive constant and $j$ an almost
complex structure of $M$ compatible with $g$, cf. Proposition 2.5.6. of \cite{McSa}.

So we have that $\Om = B \om$ where $B >0$ is constant and $\om$ is a
symplectic form of $M$ defined by $\om ( X,Y) = g (jX, Y)$. We will consider the large
$B$ limit. To do this, we will replace $L$ by $L^k$, $k \in \N$, so that the
curvature of $\nabla^{L^k}$ is $k B \om$, and let $k$ tend to infinity.
We will also normalise the metric so that $B=1$, and our magnetic intensity is
simply $k$.








Alternatively, we can introduce our data as follows. Consider a compact symplectic
manifold $(M^{2n}, \om)$ with a compatible almost complex structure $j$ and a
Hermitian line bundle $L \rightarrow M$ with a connection $\nabla$ having
curvature $\frac{1}{i} \om$.  Such a bundle is called a {\em prequantum}
bundle in the Kostant-Souriau theory, where it is used to define the geometric
quantization of $M$.
For any positive integer $k$, we consider the Laplacian
\begin{gather}
\Delta_k = \tfrac{1}{2}( \nabla^{L^k}) ^* \nabla^{L^k} : \Ci ( M ,
L^k)\rightarrow \Ci ( M , L^k)
\end{gather}
with $\nabla^{L^k} : \Ci ( M, L^k) \rightarrow \Om^1
(M, L^k)$ the covariant derivative induced by $\nabla$, and the Riemannian
metric  $g(X,Y) = \om
(jX, Y)$ independent of $k$.


\subsubsection*{Earlier results}

It is known that the spectrum $\si ( \Delta_k)$ of $\Delta_k$ is
partitioned into clusters around each point of $k ( \frac{n}{2} + \N)$ in the
large $k$ limit. More precisely, for any $m \in \N$, define the interval $I_m$
$$ I_0 =  [0, \tfrac{n}{2} + \tfrac{1}{2}], \qquad I_m =  (\tfrac{n}{2}  +m) + [-\tfrac{1}{2} , \tfrac{1}{2} [ \text{ if
} m \geqslant 1 ,$$
so that we have a partition $[0,\infty [ = \bigcup
_{m\in \N} I_m$. Then we set 
\begin{gather} \label{eq:def_Landau}
\Si_{m,k} :=  \bigl( k^{-1} \si (\Delta_k) \bigr)  \cap I_m 
, \qquad \Hilb_{m,k} := \bigoplus_{\la \in \Si_{m,k}} \ker ( k^{-1}
\Delta_k - \la). 
\end{gather}
It was proved by Faure-Tsuji \cite{FaTs15} that
\begin{gather} \label{eq:Si_m:Faure_Tsujii}
\Si_{m,k} \subset
\Bigl( \tfrac{n}{2}+ m + C_m k^{-\frac{1}{4}} [-1,1] \Bigr) 
\end{gather}
and by Demailly \cite{Dem} that
\begin{gather} \label{eq:asymptot_Demailly}
\op{dim} \Hilb_{m,k} = \Bigl( \frac{k}{2 \pi} \Bigr)^n { {m+n-1} \choose
  {n-1} }   \op{vol} (M) + \op{o} (k^n) 
\end{gather} 
For a surface $(n=1)$ with a constant Gauss curvature $S$, more precise results
have been obtained by Iengo-Li \cite{IeLi94}: if $k + m S >0$, then
\begin{gather}  \label{eq:surface_eignvalue_dim}
\begin{split} 
  \Si_{m,k} = \bigl\{  \tfrac{1}{2} + m   + k^{-1} S \tfrac{m(m+1)}{2} \bigr\} , \\
\op{dim} \Hilb_{m,k} =  \tfrac{k}{2 \pi}   \op{vol} (M) + (\tfrac{1}{2} + m )
\chi (M).
\end{split} 
\end{gather}
So in this case, when $k$ is sufficiently large, the $m$-th eigenvalue is
degenerate with multiplicity equal to  $\op{dim} \Hilb_{m,k}$.

The first cluster has been further studied. 
In the K\"ahler case,
that is when the complex structure
$j$ is integrable, $L$ has itself a natural holomorphic structure such that
$\con{\partial}_L = \nabla^{0,1}$ and by Kodaira identities, we have when $k$
is sufficiently large that
\begin{gather}
\Si_{0,k} = \{ \tfrac{n}{2} \} , \qquad \Hilb_{0,k} = H^0 ( M, L^k).
\end{gather}
and the dimension of $\Hilb_{0,k}$ is given by the Riemann-Roch-Hirzebruch
Theorem
\begin{gather} \label{eq:RRH}
\dim \Hilb_{0,k} = \int_M \exp \Bigl( \frac{ k \om }{2\pi
} \Bigr) \; \op{Todd} M 
\end{gather}
Here $\op{Todd} M$ is the Todd class of
$(M,j)$.

More generally, when $j$ is not necessarily integrable, it was proved by Guillemin-Uribe \cite{GuUr88} that
\begin{gather} \label{eq:si_0-k_Guillemin_Uribe}
\Si_{0,k} \subset \Bigl( \tfrac{n}{2}
+ C_0 k^{-1} [-1,1] \Bigr)
\end{gather}
and by Borthwick-Uribe \cite{BoUr96} that the dimension of
$\Hilb_{0,k}$ is given by \eqref{eq:RRH} when $k$ is sufficiently large. 

 \subsection{Main results}  

In the sequel $m \in \N$ is a fixed non negative integer and all the results
hold in the large $k$ limit, with estimates, bounds  depending on $m$.

\subsubsection*{Dimension} 
 Our first result is the computation of the dimension of $\Hilb_{m,k}$, as the
Riemann-Roch number of $L^k \otimes \D_m (TM)$ where $\D_m (TM)$ is the $m$-th
symmetric power of $(T^{0,1} M)^*$. Here, $T^{0,1}M = \ker (j + i )$ and  $j$
is the almost complex structure introduced previously. The reason why we prefer
to work with $(T^{0,1}M)^*$ instead of the isomorphic bundle $T^{1,0}M$ should
be clear later.

\begin{theo} \label{theo:dim_landau}
  If $k$ is sufficiently large, then
  $$ \op{dim} \Hilb_{m,k} = \int_M \exp \Bigl( \frac{ k \om }{2\pi
  } \Bigr) \; \op{ch} ( \D_m (TM) ) ) \op{Todd} M $$
with $\op{ch}$ the Chern character and $\op{Todd} M$ the Todd class of
$(M,j)$. 
\end{theo}

As far as we know, Theorem \ref{theo:dim_landau} is a new result, except in the cases already
mentionned ($n=1$ with constant curvature or $m=0$). 

\begin{rem} \begin{itemize} 
\item[-]   When $n=1$, $\D_m ( TM)$ is isomorphic with $K^{-m}$, $K$ being the canonical
bundle, and it is easy to see that   we recover the second equation
\eqref{eq:surface_eignvalue_dim}. However, even for $n=1$, Theorem
\ref{theo:dim_landau} goes further since we don't assume
that the Gauss curvature is constant. In this generality, it is not likely
that  $\Si_{m,k}$  consists of a single degenerate eigenvalue, but the
dimension of the $m$-th cluster is given by the same formula.
\item[-]  For a general dimension $n$, $\D_m (TM)$ has rank ${ {m+n-1} \choose
  {n-1} }$ and we recover the asymptotic \eqref{eq:asymptot_Demailly}.
\end{itemize}
\end{rem}


\subsubsection*{Symbol spaces} In the sequel, we will use $\D_m (TM)$ as a bosonic space, with
associated creation an annihilation operators defined as follows. For any $x
\in M$, let us view $ \D_{m} (
T_xM)$ as the space of homogeneous polynomials maps $T_x^{0,1} M  \rightarrow \C$ with
degree $m$. Set
\begin{gather} \label{eq:dtm_intro}
  \D (T_xM) := \bigoplus_{m \in \N} \D_m (T_xM)
\end{gather}
and let $\pi_m(x)$ be
the corresponding projector of $\D (T_x M)$ onto $\D_m (T_x M)$. For any $ Y
\in T_xM \otimes \C$, let $\rho (Y)$ be the endomorphism of $\D (T_xM)$ defined as follows. Write $Y = U + \con V$ with $U, V \in T_x ^{1,0}
M$. Then 
\begin{gather} \label{eq:rho}
\rho ( Y)=
\rho ( U) + \rho ( \con V) \; \text{ with } \;  \begin{cases}  \rho (U) =  \text{
    multiplication by }  i 
\om ( U, \cdot) \\ \rho ( \con {V}) = \text{ derivation with respect to } \con
V  
\end{cases} 
\end{gather}
More concretely, let $(U_i)$ be a basis of $T^{1,0}_xM$ such that $\frac{1}{i} \om ( U_i ,
\con {U}_j) = \delta_{ij}$. Let $(z_i)$ be the basis of $(T^{1,0}_xM)^*$ dual
to $(U_i)$, so $z_i = \frac{1}{i} \om ( \cdot, \con{U}_i)$ and $ \D
(T_x M) = \C [\con{z}_1 , \ldots, \con{z}_n]$. Then for any polynomial
$P$ of the variable $\con{z}_1$, $\ldots$, $\con{z}_n$
\begin{gather}
\rho (
U_i) P = -  \con{z}_i P, \qquad  \rho ( \con{U}_i)  P = \frac{\partial P}{\partial
  \con{z}_i}  .
\end{gather}
So  $-\rho ( U_i)$ and $\rho (\con{U}_i)$  are respectively the creation and
annihilation operators. 

\subsubsection*{Berezin-Toeplitz operators} 


  

Our second result is about Berezin-Toeplitz operators. By \cite{BoUr96}, the
spaces $\Hilb_{0,k}$ can be considered as quantizations of $M$, replacing the
standard K\"ahler quantization $H^0 ( M,L^k)$, for symplectic manifold not
necessarily having an integrable complex structure. An important feature is that there is a natural
way to pass from classical to quantum Hamiltonians, provided by the
Berezin-Toeplitz quantization. In the semi-classical limit, defined here as the
large $k$ limit,  the product and commutator of quantum observables
correspond to the product and Poisson bracket of classical observables, up to
some error terms.
More precisely, let $\Pi_{m,k}$ be the orthogonal projector of $\Ci ( M , L^k)$
onto $\Hilb_{m,k}$ and for any $f \in \Ci (M)$, let $T_{m,k} (f)$ be the endomorphism of
$\Hilb_{m,k}$ defined by 
\begin{gather} \label{eq:toeplitz}
  T_{m,k} (f) \psi  = \Pi_{m,k} ( f \psi)
\qquad \forall \; \psi \in \Hilb_{m,k} .
\end{gather}
For the first Landau level, it is known \cite{BoMeSc}, \cite{oim_op}, \cite{MaMa}
that for any $N$ 
\begin{gather} \label{eq:t_0_k_produit_tout_ordre}
T_{0,k} ( f) T_{0,k} (g) = \sum_{\ell = 0 } ^N  k^{-\ell} T_{0,k} ( B_\ell ( f,g)) +
\bigo (k^{-(N+1)})
\end{gather}
for some bidifferential operators $B_{\ell} : \Ci ( M) \times \Ci ( M)
\rightarrow \Ci ( M)$, where
\begin{gather}  \label{eq:b_0_b_1}
B_0 ( f,g) = fg, \qquad B_1(f,g) = - \tfrac{1}{2} 
g(X,Y) + \tfrac{1}{2i} \om ( X, Y),
\end{gather}
$X$ and $Y$ being the Hamiltonian vector
fields of $f$ and $g$ respectively.

For the generalisation to higher Landau levels, we will use in addition to the
$T_{m,k} ( f)$'s the following operators: let $p \in \N$ and $X_1$, \ldots,
$X_{2p}$ be vector fields of $M$. Define $ T_{m,k} (X_1, \ldots , X_{2p}):
\Hilb_{m,k} \rightarrow \Hilb_{m,k}$ by 
\begin{gather} \label{eq:toep_der} 
T_{m,k} (X_1, \ldots , X_{2p}) (\Psi) = k^{ -p} \Pi_{m,k} ( \nabla^{L^k}_{X_1} \ldots
\nabla_{X_{2p}}^{L^k}  \Psi)   
\end{gather}
Let $\mathcal{T}_m$ be the vector space of families $ (P_k : \Hilb_{m,k} \rightarrow
\Hilb_{m,k}, \; k \in \N)$ spanned by the $(T_{m,k} (f), k \in \N)$'s and $(T_{m,k} ( X_1, \ldots ,
X_{2p}), \; k \in \N) 's$. Here the functions $f$ or vector fields $X_1,
\ldots X_{2p}$ do not depend on $k$.

\cB 
Define the semiclassical completion
$\mathcal{T}_m ^{\op{sc}}$ as the vector space of families $(P_k : \Hilb_{m,k} \rightarrow \Hilb_{m,k} , \; k \in
\N)$ such that for any $N$,
$$ P_k = \sum_{\ell= 0 } ^N k^{-\ell} P_{\ell,k} + \bigo ( k ^{-N-1}) $$
where the coefficients $(P_{\ell, k })_k$, $\ell \in \N$ all belong to
$\mathcal{T}_m$, and the $\bigo $ is for the operator norm. 
This expansion is meaningful because as we will see later, for any $\ell$, the
operator norm $\| P_{\ell, k}\|$ is bounded independently of $k$. 

\clb
\begin{theo} \label{theo:Toeplitz_Landau}
 For any $m \in \N$, we have:
  \begin{enumerate}
  \item  $ \mathcal{T}_m^{\op{sc}}$ is closed under product.
  \item \cA There exists a unique linear
  map $ \mathcal{\tau} : \mathcal{T}_m^{\op{sc}} \rightarrow \Ci ( M ,
  \op{End} ( \D_m (TM))) $  given on the generators \clb
  \eqref{eq:toeplitz} and \eqref{eq:toep_der} by  
  \begin{xalignat}{2} \label{eq:symbolmap}
    \begin{split} 
    \tau ( T_{m,k} ( f) )_x   & = f(x) \op{id}_{\D_m
      (T_x M)}, \\ \tau ( T_{m,k} ( X_1, \ldots , X_{2p}) )_x &  = \pi_m(x) \rho(X_1(x))
      \ldots \rho (X_{2p}(x))   
    \end{split}
  \end{xalignat}
$\tau$ is onto, its kernel consists of $k^{-1} \mathcal{T}_{m}^{\op{sc}}$.
\item   For any $P,Q \in \mathcal{T}_{m}^{\op{sc}}$
  \begin{xalignat}{2} \label{eq:prod}  
    & \tau ( PQ) = \tau ( P) \tau (Q) , \\ \label{eq:norm}
    & \| P_k \| = \sup \{ \| \tau (
    P)_x \| , \; x \in M \} + \op{o} (1),  \\ \label{eq:trace}
    &
 P_k(x,x) = \Bigl( \frac{k}{2\pi} \Bigr)^n ( \op{tr} (\tau (P)_x) + \bigo
 (k^{-1})). 
\end{xalignat}
\item For any $f$, $g$ in $\Ci ( M)$, we have
\begin{gather} \label{eq:produit}
  T_{m,k} (f) T_{m,k} (g) = T_{m,k} ( fg) + k^{-1} T_{m,k} (X,Y)  + \bigo (
  k^{-2}) 
\end{gather}
where $X$, $Y$ are the Hamiltonian vector fields of $f$ and $g$. In
  particular,
\begin{gather} \label{eq:com_poisson} 
  ik [ T_{m,k} (f), T_{m,k} (g) ]
  = T_{m,k} ( \{ f,g \}) + \bigo ( k^{-1})
\end{gather}
with $\{ \cdot, \cdot \}$ the Poisson Bracket of $(M, \om)$. 
\end{enumerate}
\end{theo}

We call $\tau$ the symbol map. The symbol of the generators \eqref{eq:symbolmap} is defined in
terms of the endormorphisms \eqref{eq:rho}. The product of symbols in the right-hand
side of \eqref{eq:prod} is the pointwise composition. In the norm estimate
\eqref{eq:norm}, the norm of $\tau   ( P)_x$ is defined in terms of the
hermitian structure of $\D_m (T_xM)$.  In \eqref{eq:trace}, $P_k
(x,x)$ is the value of the Schwartz kernel of $P_k$ at $(x,x)$. Integrating
\eqref{eq:trace}, we obtain the following estimate of the trace of $P_k$,
\begin{gather}
\op{tr}
P_k  = \Bigl( \frac{k}{2\pi} \Bigr)^n   \int_M \op{tr} (\tau (P)_x) \mu_M(x)
+ \bigo ( k^{n-1})  
\end{gather}
where $\mu_M = \om^n /n!$.
Since $P_k =T_{m,k} (1)$ is the identity
of $\Hilb_{m,k}$, we recover the estimate \eqref{eq:asymptot_Demailly}. 

\begin{rem} $ $ \label{rem:Toeplitz}
\begin{enumerate}
  \item  In the surface case, $n=1$, $\D_m(TM)$ is a line bundle, so any
    endomorphism of $\D_m(T_xM)$ is scalar and by Assertion 2  of
Theorem \ref{theo:Toeplitz_Landau}, $\mathcal{T}_{m}^{\op{sc}}$ consists of the
families
\begin{gather} \label{eq:-p_k-=}
( P_k = T_{m,k} ( f( \cdot, k)) + \bigo ( k^{-\infty}) , \; k \in \N)
\end{gather}
where the
multiplicator $f( \cdot, k)$ depends on $k$ in such a way that it admits an expansion $f( \cdot, k ) = f_0
+ k^{-1} f_1 + \ldots$.  
\cB It holds as well that the $T_{m,k} (f)$ satisfy
\eqref{eq:t_0_k_produit_tout_ordre} for some bidifferential operators
$B_\ell^m$ depending on $m$. Indeed, by \cite[Section 5.4]{oim}, this property is equivalent to the
locality of the product: for any two functions $f$, $g$ with disjoint
supports, $T_{m,k} (f) T_{m,k} ( g) = \bigo (k^{-\infty})$. This latter
property follows from the fact that the Schwartz kernel of $\Pi_{m,k}$ is in
$\bigo ( k^{-\infty})$ outside the diagonal. \clr

\eqref{eq:produit} writes in this case
\begin{gather} \label{eq:bm_1-f-g}
B^m_1 (f,g) = - ( \tfrac{1}{2} +m ) g(X,Y) + \tfrac{1}{2i} \om (X,Y)
\end{gather}
where $X,Y$ are the Hamiltonian vector fields of $f$ and $g$.

\item  \cA
  If $n \geqslant 2$ and $m\geqslant 1$, the Toeplitz algebra
  $\mathcal{T}_m^{\op{sc}}$ is strictly larger than the subspace $
  \mathcal{T}_m^{\op{f}}$ consisting of the Toeplitz operators
  \eqref{eq:-p_k-=} with scalar multiplicators, because the
  symbol map $\tau $ is onto. Nevertheless, we may ask if  $
  \mathcal{T}_m^{\op{f}}$ is closed under product. This is not the case.
  Indeed, if follows from \eqref{eq:produit}, that if $f \in \Ci (M)$ is not
  locally constant, then there exists no function $h$ such that $T_{m,k} (f)^2 - T_{m,k} (f^2) = k^{-1} T_{m,k} (h) +
  \bigo ( k^{-2})$. So interestingly, by developping the theory of
  Berezin-Toeplitz operator for higher Landau level, we are naturally led to
  use matrix valued symbols. 
\item The Toeplitz operators defined as in \eqref{eq:toep_der}, but with an odd
  number of vector fields, can be incorporated in the theory. They are odd
  Toeplitz operators, cf. Remark \ref{rem:odd_toeplitz}, and they belong to
  $k^{-\frac{1}{2} }  \mathcal{T}_m^{\op{sc}}$. We purposely have avoided any
  square root of $k$ in Theorem \ref{theo:Toeplitz_Landau}.   
\end{enumerate}  \clb
\end{rem}
\begin{rem} $ $ \label{sec:comp_Yuri}
  \begin{enumerate} 
\item  The main estimates for the Toeplitz operators associated to functions, that is $T_{m,k} (f) T_{m,k} (g) = T_{m,k} ( fg) + \bigo ( k^{-1})$ and
 $ik [T_{m,k} (f), T_{m,k} (g) ]= T_{m,k} ( \{ f, g \}) + \bigo ( k^{-1})$
 have been proved independently by Kordyukov \cite{Yuri_2} by using the
 techniques of \cite{MaMa}.
\item \cA
  In \cite{Yuri_2}, the non degenerate magnetic fields $j_B$ with
 constant eigenvalues $B_1$, \ldots, $B_n$ are considered as well. The
 corresponding magnetic Laplacian has clusters centered at the points of
 $k\Si$ where 
 $\Si = \{ \sum B_i ( \frac{1}{2} + \al (i) ) / \al \in \N^n \}$. 
 In the companion paper \cite{oim_copain}, we prove that the number of eigenvalues in
 the cluster at $\La \in \Si$ is given by a Riemann-Roch number as in Theorem
 \ref{theo:dim_landau} where $\mathcal{D}_m (TM)$ is replaced
 by a bundle $F_{\La}$ \footnote{$F_{\Lambda}$ is defined as follows: for any $p \in M$, introduce a
 basis $( \partial_j, \con{\partial}_j)$ of $T_pM \otimes \C$ such that $j_B
 \partial_j = i B_j \partial_j$ and denote by $(z_j, \con{z}_j)$ the dual
 basis of $T^*_pM \otimes \C$.
 Then the fiber of $F_{\Lambda}$ at $p$ is spanned by the $\con{z}_1^{\al(1)}
 \ldots \con{z}_n^{\al (n)}$ where $\al \in \mathcal{K}_{\Lambda}$.  }
with rank the cardinal of
 $$\mathcal{K}_{\La} = \{ \al \in \N^n, \sum B_i
 ( \tfrac{1}{2} + \al (i) ) = \Lambda \}. $$ 
When $|\mathcal{K}_{\Lambda} | =1$, it is proved in \cite{Yuri_2} that the
space of Toeplitz operators in the $\Lambda$-cluster associated to functions of $M$ is an algebra.
This particular case is very similar to the first Landau level ($m=0$) because the endomorphisms of $K_{\Lambda}$ are scalar, the symbol
composition law is commutative and the Toeplitz  operator \eqref{eq:toep_der}
associated to vector fields are not necessary. Observe as well that in dimension $n \geqslant 2$, the condition that
$|\mathcal{K}_{\Lambda}| =1$ for $\Lambda \neq \frac{1}{2} \sum B_i$ implies
that the $B_i$ are not all equal, so that the tangent bundle of $M$, endowed with
 the unique up to isotopy complex structure compatible with the symplectic
 structure, splits into a non trivial sum of complex subbundles. This seems
 rather restrictive and is not satisfied for instance  by the complex projective spaces
 $\mathbb{P}^n$ when $n$ is even, cf.  \cite{split}.

\item 
We can study as well with our techniques the Toeplitz operators associated to
the $\Lambda$-cluster and this will be partly done in Section
\ref{sec:projectors-Toeplitz} where we consider a vector bundle $F$
generalising $\mathcal{D}_m (TM)$ or $F_{\Lambda}$. Assuming that the numbers
$|\al| = \al (1) + \ldots + \al (n)$ have all the same parity when $\al$ runs
over $\mathcal{K}_{\La}$, we have the same result as Theorem
\ref{theo:Toeplitz_Landau}. When this parity condition is not satisfied,
the square roots of $k$ seem to be unavoidable. 

We have chosen to work out the case where all the $B_i$ are equal because
first any symplectic manifold can be endowed with such a magnetic field, which is merely
a compatible almost complex structure, and second the higher
rank vector
bundle $\mathcal{D}_m (TM)$ makes it rather different from the already much
studied case of first
Landau levels.
\end{enumerate} \clb
\end{rem}

\subsubsection*{Ladder operators} 

The last result we would like to emphasize in this introduction is the construction of some ladder
operators for the spaces $\Hilb_{m,k}$. In the surface case with
constant Gauss curvature, $\Hilb_{m,k}$ is naturally isomorphic with the
space of holomorphic sections of $L^k \otimes K^{-m}$ where $K$ is the
canonical bundle \cite{Te06}, the isomorphism being
the {\it ladder} operator  $\con{\partial}_{L^k \otimes K^{-m+1}} \circ \ldots
\circ \con{\partial}_{L^k}$, cf. the appendix \ref{sec:appendix}. 
Here we will show that the family $(\Hilb_{m,k}, \; k \in \N)$ is isomorphic to a quantization of $M$
twisted by the vector bundle $\D_m (TM) $.

Recall that for any Hermitian vector
bundle $F \rightarrow M $, we can define a family of
finite dimensional subspaces $\Hilb_{F,k} \subset \Ci ( M , L^k \otimes F)$, $k
\in \N$, having the following properties:
\begin{enumerate} 
\item $\op{dim} \Hilb_{F,k} = \int_M \op{ch} ( L^k \otimes F) \; \op{Todd} M$,
  when
  $k$ is sufficiently large.
\item the space $\mathcal{T}_F^{\op{sc}}$, consisting of families of $(T_k \in \op{End}
  (\Hilb_{F,k}) , \; k \in \N)$ having an expansion of the form 
$$  T_k = \sum_{\ell = 0 }^N k^{-\ell} T_{F,k} ( f_{\ell} ) + \bigo ( k^{-N+1}) , \qquad
\forall N \in \N $$
for a sequence $(f_\ell)$ of $\Ci ( M , \op{End} F)$, is closed under product.
Here,  $T_{F,k} (f_{\ell}) (\psi) = \Pi_{F,k} ( f_{\ell} \psi)$ for any $\psi
\in \Hilb_{F,k} $ where $\Pi_{F,k}$ is the orthogonal projector of $\Ci ( M, L^k
\otimes F)$ onto $\Hilb_{F,k}$. 
\item At first
order, the product is given by the pointwise product, that is  $ T_{F,k} ( f)
T_{F,k} (g) = T_{F,k} ( fg) + \bigo (k^{-1})$ for any $f,g \in \Ci (M,
\op{End} F)$. 

\end{enumerate} 
The algebra $\mathcal{T}_F^{\op{sc}}$ is called the Toeplitz algebra. In the K\"ahler
case, that is when $(M, \om)$ is
K\"ahler, $L$ holomorphic with $\nabla$ the Chern connection, and $F$
holomorphic as well, the space $\Hilb_{F,k}$ can be defined as the space  $
H^0 (L^k \otimes F)$ of
holomorphic sections. In the non-K\"ahler case, various
constructions have been developed \cite{BoUr96},
\cite{MaMa}: Spin-c quantization, first Landau level of a Laplacian acting on $\Ci ( M , L^k
\otimes F)$, or more generally any image of a projector of $\Ci ( M , L^k
\otimes F)$ having a specific Schwartz kernel \cite{oim}. Let us call such a
family $(\Hilb_{F,k}, \; k \in \N)$ a {\em quantization of $(M,L)$ twisted by
$F$}.  These twisted quantizations have sometimes better properties than the
non twisted one (corresponding to $F=\C$),
typically when $F$ is a half-form bundle \cite{oim_eq}. The general case where the rank of
$F$ is $\geqslant 2$ may be viewed as a free generalization
without any application,
but interestingly, this is exactly what we need. 

Assume $F$ is equipped with a connection $\nabla^F : \Ci ( M , F) \rightarrow
\Om^1 (M, F)$. Let $G = (T^{0,1}M)^*$ and $D_{F,k} : \Ci ( M , L^k \otimes F) \rightarrow \Ci ( M , L^k
\otimes F \otimes G)  $ be the $(0,1)$-part of the connection $\nabla^{F \otimes
  L^k}$ induced by $\nabla^F$ and $\nabla^{L^k}$. Endow $G$
with a connection and define  the differential operators
\begin{gather} \label{eq:def_w_k}
\begin{split}
  W_k :  \Ci ( M , L^k ) \rightarrow \Ci (
  M, L^k \otimes \D_m (TM) )\\
  W_k = 
  R_m D_{G^{\otimes (m-1)},k} \circ  D_{G^{\otimes
    (m-2)},k} \circ \ldots \circ D_{G, k} \circ D_{\C,k}
\end{split}
\end{gather}
where $R_m$ is the projection from $G^{\otimes m }$ onto $\D_m (TM) =
\op{Sym}^m G$. 

\begin{theo} \label{theo:ladder} For any quantization $( \Hilb_{F,k}, \; k \in \N)$ of $(M,L)$
  twisted by $F= \D_m ( TM)$, the linear maps
  $$V_k = \tfrac{1}{m!} k^{-\frac{m}{2}} \Pi_{F,k} W_k
  : \Hilb_{m,k} \rightarrow \Hilb_{F,k}, \qquad k \in \N $$
  satisfy: 
\begin{enumerate}
\item $V_kV_k^* = \op{id}_{\Hilb_{F,k}} + \bigo ( k^{-1})$ and $V_k^* V_k
  =\op{id}_{\Hilb_{m,k}} + \bigo ( k^{-1})$. In particular, $V_k$ is an isomorphism when $k$ is
  sufficiently large. 
    \item the conjugation by $V=(V_k)$ is an isomorphism between the Toeplitz algebra $ \mathcal{T}^{\op{sc}}_m$ and
      $\mathcal{T}_F^{\op{sc}}$ modulo $\bigo ( k^{-\infty})$. 
      In particular, for any $(P_k)  \in
      \mathcal{T}^{\op{sc}}_m$, $(V_k P_k  V_k^*)_k$ belongs to
      $\mathcal{T}_F^{\op{sc}}$ and if $f \in \Ci ( M , \op{End} F)$ is the
      symbol $ \tau(P_k)$, then $V_k P_k  V_k^* = T_{F,k} (f) + \bigo (
      k^{-1})$. 
    \end{enumerate}
  \end{theo}

  The first assertion of Theorem \ref{theo:ladder} tells us that $V_k$ is almost
  unitary. This can be improved by setting $U_k := A_k V_k$ with $A_k$ the
  endomorphism of $\Hilb_{F,k}$ equal to $(V_kV_k^*)^{-1/2} |_{\Hilb_{m,k}}$
  when $k$ is sufficiently large and to $0$ for the first values of $k$. 
  Then $U_k U_k^* = \op{id}_{\Hilb_{F,k}}$ and $U_k^* U_k =
  \op{id}_{\Hilb_{m,k}}$ when $k$ is sufficiently large. Furthermore the
  second assertion of
  \ref{theo:ladder} holds with $(U_k)$ instead of $(V_k)$.

\cB  The inspiration for \eqref{eq:def_w_k} comes from the case of surface ($n
  =1$) with constant Gauss curvature. In this case, choosing for $D_{F,k}$ the
  $
\con\partial$ operator, it holds that $W_k (
  \Hilb_{m,k} ) \subset  \Hilb_{F,k}$, cf. the discussion after Theorem
  \ref{theo:appendix}, so the projector $\Pi_{F,k}$ in the definition of $V_k$
  is not necessary. We expect that something similar happens in higher
  dimension under the convenient assumptions.  \clb

\subsection{Generalised Landau level and Schwartz kernel expansion} 

\subsubsection*{Generalised Landau level}

In the previous results, the $m$-th Landau level $\Hilb_{m,k}$, $k\in \N$ can be replaced by
any family $(\Hilb_{m,k} \subset \Ci ( M , L^k ) , \; k \in \N ) $ of finite
dimensional subspaces, such that
the Schwartz kernel of the orthogonal projector $\Pi_{m,k}$ of $\Ci ( M, L^k)$
onto $\Hilb_{m,k}$ has a specific behavior in the large $k$ limit. We require
first that 
\begin{gather} \label{eq:first_order}
\Pi_{m,k} (x,y) = \Bigl( \frac{k}{ 2\pi} \Bigr)^n E^k (x,y) Q^{(n-1)}_m ( k
\delta (x,y) ) + \bigo (k^{n-1})
\end{gather}
where
\begin{itemize}
  \item[-]  $E$ is a section of $L \boxtimes \con{L}$ such that $|E(x,y)| < 1$
    when $x \neq y$, and its second order Taylor expansion along the diagonal has a
specific form, cf. Equation \eqref{eq:hypotheseE}. \cB In particular,  in a coordinate chart at $x$,  $\ln |E(x+
\xi, x)|= - \frac{1}{4} |\xi|_x^2 + \bigo ( |\xi|^3)$ where $|\cdot |_x$ is
the norm of $T_xM$. \clb  \footnote{\cB In the whole paper, when working in a
  coordinate chart $(U, \chi)$ of $M$, we write $x+ \xi$ for $ \chi^{-1} (
  \chi (x) + T_x \chi (\xi) )$ where $x \in U$ and $\xi \in T_xM$ sufficiently
  close to the origin. \clb}  
\item[-] 
$\delta \in \Ci (M^2)$ is any function vanishing to second order along the diagonal
and satisfying $\delta (x+ \xi , x) =  |\xi|_x^2 + \bigo (|\xi|^3)$, $\xi \in T_xM$.
$Q_m^{(p)}  $ is the generalised Laguerre polynomial $Q^{(p)}_m (x) = \frac{x^{-p}}{m!} \bigl (
\frac{d}{dx} -1 \bigr) ^m  x ^{m+p}$.  
\end{itemize}
\cA The Schwartz kernel of the projector onto the $m$-th Landau level of the magnetic
Laplacian of $\C^n$ is given by an expression similar to \eqref{eq:first_order},
with the convenient section $E$ and the same Laguerre polynomials,
cf. \eqref{eq:noyau_niveau_m} and \eqref{eq:2pin-e_cn-u}. \clb

In addition to \eqref{eq:first_order}, we require a full expansion of the form
\begin{gather} \label{eq:full}
\Pi_{m,k} (x,y) =   \Bigl( \frac{k}{ 2\pi} \Bigr)^n E^k (x,y) \sum
_{\ell \in \Z}  k^{-\ell} a_{\ell} (x,y) + \bigo (k^{-\infty})
\end{gather}
with coefficients $a_{\ell} \in \Ci ( M^2)$ such that for $\ell<0$, $a_{\ell}$
vanishes to order $m(\ell) \geqslant -2 \ell$ along the diagonal and $m(\ell) + 2\ell \rightarrow
\infty $ as $\ell  \rightarrow - \infty$. The meaning of this expansion is
not obvious because the negative $\ell$'s give positive powers of $k$.
Actually, the condition satisfied by $|E|$ implies that $|E^k (x,y) b(x,y) |
= \bigo ( k^{-m/2})$ when $b$ vanishes to order $m$ along the diagonal.
So the $\ell$-th summand in \eqref{eq:full} is in $\bigo ( k^{n - \frac{1}{2} ( m ( \ell)
  + 2\ell)})$, and the expansion is meaningful because of the conditions
satisfied by $m (\ell)$.

We will prove that for any family $(\Hilb_{m,k})$ whose associated projector
$\Pi_{m,k}$ satisfies  \eqref{eq:first_order} and \eqref{eq:full}, Theorems
\ref{theo:dim_landau},  \ref{theo:Toeplitz_Landau} and \ref{theo:ladder} hold.
On the other hand, in the second part of this work \cite{oim_copain}, cf. also
\cite{Yuri_1}, it is proved that the Schwartz kernel of the orthogonal projector $\Pi_{m,k}$ onto the Landau levels
    $\Hilb_{m,k}$ defined in \eqref{eq:def_Landau} from the Laplacian $\Delta_k$, satisfies
    \eqref{eq:first_order} and \eqref{eq:full}.
The assumption that the magnetic field is constant
with respect to the metric can be relaxed. It is actually possible to define some Landau levels and describe the asymptotic expansion of
the associated projector as soon as a particular gap condition is satisfied, cf \cite{oim_copain}.

\subsubsection*{The class $\lag (A,B)$}

To establish our results, we will introduce a specific class of operators,
containing the projector $\Pi_{m,k}$, the Berezin Toeplitz operators $T_{m,k}
(f)$ and $T_{m,k} ( X_1, \ldots, X_{2p})$ and also the projector $\Pi_{F,k}$ of any twisted quantization, the
corresponding Toeplitz operators $T_{F,k} (g)$,   the isomorphisms
$V_k$ of Theorem \ref{theo:ladder} and their unitarizations $(U_k)$. 
This operator class has a natural filtration, with associated symbol spaces,
which allows to prove most of the results by successive approximations as often
in microlocal analysis. Interestingly, in the symbolic calculus appear the
eigenprojectors of \cB the Landau Laplacian of $\C^n$, providing another link between the
usual Landau levels and our geometric Landau levels. \clb

Introduce two auxiliary Hermitian vector bundles $A$, $B$ over $ M$. Then
$\lag (A,B)$ consists of families   $(P_k : \Ci( M , L^k \otimes A ) \rightarrow \Ci(M
, L^k \otimes B), \; k \in \N)$ of operators having a smooth Schwartz kernel satisfying
\begin{gather} \label{eq:exp_lag}
P_k (x,y) = \Bigl( \frac{k}{2\pi} \Bigr)^n E^{k}(x,y) \sum_{\ell \in \Z }
k^{-\frac{\ell}{2}} b_{\ell}(x,y) + \bigo ( k^{- \infty})
\end{gather}
where $E$ is defined as in \eqref{eq:first_order}; the coefficients $b_{\ell} $ are in $\Ci
( M^2, B \boxtimes \con{A})$; for $\ell <0$, $b_{\ell}$  vanishes to order $m (\ell)
\geqslant -\ell $ along the diagonal; $m( \ell ) + \ell \rightarrow \infty $ as $\ell
\rightarrow - \infty$, and  the meaning of this expansion is the same as in \eqref{eq:full}.
We have a decomposition into even/odd elements: 
$(P_k) \in \lag^{+} (A,B)$ (resp. $\lag^{-} (A,B)$) if the expansion \eqref{eq:exp_lag} holds with a sum over the $\ell$'s
even (resp. odd).

The main property is  that this class of operators is
closed under composition:
\begin{gather} \label{eq:comp}
  \lag^{\ep'} ( B, C) \cdot \lag^{\ep} (A,B) \subset \lag^{\ep \ep'} (A,C),
  \qquad \ep , \ep' \in \{ \pm 1\}.
\end{gather}
In particular, $\lag^+ (A) := \lag^{+} (A,A)$ is an algebra. 
We also have a filtration
$ \lag_q(A,B) := \lag (A,B) \cap
\bigo  (k^{-q/2})$, $q \in \N$ and the corresponding graduation is described
by symbol maps $\si_q$
$$ 0 \rightarrow  \lag_{q+1} (A,B) \rightarrow \lag_q (A,B) \xrightarrow{\si_q} \Ci (M
, \symb (M) \otimes \op{Hom} (A,B) ) \rightarrow 0 $$
Here, $\symb (M)$ is an infinite rank vector bundle over $M$, each fiber
$\symb_x(M)$  is a subalgebra of the algebra of  endomorphisms of the space
$\D (T_xM)$ defined in \eqref{eq:dtm_intro}. This is compatible with the composition \eqref{eq:comp} in
the sense that $\lag_p (B,C) \cdot \lag_q
(A,B) \subset \lag_{p+q} (A,C)$ and the corresponding product of symbols is the
pointwise product of $\symb(M)$ tensored by $\op{Hom} (B,C) \otimes \op{Hom}
(A,B) \rightarrow \op{Hom} (A,C)$.

The projector $(\Pi_{m,k})_k$ is an idempotent of $\lag^+ (\C)$ with symbol
$\si_0 ( \Pi_{m} )$ equal at $x$ to the projector $\pi_m (x)$ onto the $m$-th
summand in \eqref{eq:dtm_intro}. The Toeplitz algebra introduced previously is
\begin{gather} \label{eq:toep_land}
\mathcal{T}_m^{\op{sc}} = \{ P
\in \lag^{+} (\C) / \; \Pi_m P \Pi_m = P \}.
\end{gather}
The isomorphism $V$ of
Theorem \ref{theo:ladder} belongs to $\lag ( \C, F)$ and has the same parity
as $m$. 

Interestingly,
$\symb_x(M)$ has a representation as operators of $L^2 (\C^n)$, and in
this  representation, $ \pi_m(x) = \si_0 ( \Pi_m)(x)$ is the projector onto the $m$-th Landau
level of a magnetic Laplacian of $\C^n$.

\subsubsection*{Outline of the paper} 
In Section \ref{sec:class-lag-pres}, we introduce the class $\lag (A,B)$, and
state its main properties. In Section \ref{sec:projectors-Toeplitz}, we prove
variations of the theorems stated before, where the $\Hilb_{m,k}$ are subspaces of $\Ci ( M , L^k
\otimes A)$  such that the corresponding family
$(\Pi_{m,k})$ of orthogonal projectors belongs to $\lag^+ (A)$ with a
convenient symbol.
Sections \ref{sec:landau-levels-cn}, \ref{sec:schw-kern-oper} and
\ref{sec:derivatives} are devoted to the proof of the properties of $\lag
(A,B)$. The proofs of  Theorems  \ref{theo:dim_landau}, \ref{theo:Toeplitz_Landau} and
\ref{theo:ladder} is given in the last subsection \ref{sec:ledernier}. In the appendix \ref{sec:appendix}, we prove formulas  \eqref{eq:surface_eignvalue_dim} on constant
curvature surface.

\subsubsection*{Acknowledgment} I would like to thank Yuri Kordyukov for his collaboration at an early stage of this
work. I thank as well the anonymous referees for their careful reading, their
comments and their suggestions.

\section{The class \texorpdfstring{$\lag (A,B)$}{L(A,B)}} \label{sec:class-lag-pres}

  We start the discussion with the algebra in which the symbol of operators of
  $\lag (A, B)$ takes their values. The class $\lag (A,B)$ is defined in
  Subsection \ref{sec:operators} \cA and its main properties are stated, the proof
  are postponed to Section \ref{sec:schw-kern-oper}. \clb

  \subsection{Symbol spaces} \label{sec:symbol-spaces}

\subsubsection*{The algebra $\symb ( \C^n)$}
Let $n$ be a positive integer and denote by $z_1, \ldots, z_n$ the linear coordinates of
$\C^n$. Let $\D (\C^n)= \C [ \con{z}_1, \ldots, \con{z}_n]$ be the space of
antiholomorphic polynomial maps from $\C^n$ to $\C$.
Introduce the scalar product 
\begin{gather} \label{eq:scal_product_bargm}
 \langle f, g \rangle = (2\pi)^{-n}\int_{\C^n} e^{-|z|^2} f (z)\, \con{g (z)} \; d \mu_n
(z) , \qquad f,\; g \in \D ( \C^n)
\end{gather}
where $|z|^2 = \sum_{i=1}^n |z_i|^2$ and $ \mu_n $ is the measure $\prod_{i=1}^n dz_i
d\con{z}_i$. The family $((\al !)^{-\frac{1}{2}} \con{z}^\al , \, \al \in
\N^n)$ is an orthonormal basis of $\D ( \C^n)$. We will also need the decomposition into
even and odd functions
\begin{gather} \label{eq:paire_impaire}
\D ( \C^n) = \D^+ (\C^n) \oplus \D ^{-} ( \C^n) 
\end{gather}
where $ \D^+ (\C^n) $ is spanned by the $\con{z}^\al$ with $|\al| = \sum \al(i)$
even and $\D ^{-} ( \C^n) $ by the $\con{z}^\al$ with $|\al|$ odd. 

Let $\symb (\C^n)$ be the space of endomorphisms \footnote{\cA In the whole paper,
the endomorphisms are vector space endomorphisms \clb} $s $ of $\D(
\C^n)$ such that $s ( \con{z}^\al) = 0$ except for a finite number of $\al \in
\N^{n}$. We claim that $\symb ( \C^n)$ is closed under product and
taking adjoint. To see that, simply observe that $\symb ( \C^n)$ is the space of
endomorphisms having a matrix in the basis $(\con{z}^\al)$ whose almost all
entries are equal to zero. Notice as well that the family $(\Uu_{\al, \be}, \;
\al, \be \in \N^n)$ of $\symb (\C^n)$
defined by
\begin{gather} \label{eq:def_Ualphabeta}
  \begin{split} 
\Uu_{\al \be } \bigl( (\be !)^{-\frac{1}{2}}  \con{z}^\be \bigr) = (\al !)^{-\frac{1}{2}}
\con{z}^{\al}, \\ \Uu_{\al \be} ( \con{z}^{\ga} ) = 0 , \quad \forall \ga \in \N^n
\setminus \{ \be \}
\end{split} 
\end{gather}
is a vector space basis of $\symb ( \C^{n})$. And we have
\begin{gather} \label{eq:rel_U}
\Uu_{\al \be} \circ \Uu_{\tilde \al \tilde \be } = \delta_{\be \tilde{\al}}
\Uu_{\al \tilde \be} , \qquad \Uu_{\al \be}^* =
\Uu_{\be \al} 
\end{gather}
for all $\al$, $\be$, $\tilde \al$, $\tilde \be$ in $\N^n$.

Each element $s \in \symb ( \C^n)$
can be written in a block matrix $s = \begin{pmatrix} s_{++} & s_{-+}
  \\ s _{+-} & s_{--} 
\end{pmatrix}
$ in the decomposition \eqref{eq:paire_impaire}, which leads to a decomposition into even/odd endomorphisms
\begin{gather} \label{eq:symb+_-}
\symb ( \C^n) =   \symb^{+} ( \C^n) \oplus  \symb^{-} ( \C^n) 
\end{gather}
where $s \in \symb^{+} ( \C^n)$ iff $s_{-+} = s_{+-} =0$, and  $s \in
\symb^{-} (\C^n)$ iff $s_{++} = s_{--} = 0
$. Observe that $\Uu_{\al\be}$ has the same parity as $|\al | + | \be|$.
Furthermore
\begin{gather} \symb^{\ep} ( \C^n)\, \cdot \, \symb^{\ep'} ( \C^n) \subset
  \symb^{\ep \ep'} ( \C^n)
\end{gather}
for any $\ep, \ep' \in \{ 1, -1 \}$.

\subsubsection*{Extension to vector bundles} 
In the previous definitions, we can replace $\C^n$ by any $n$-dimensional
Hermitian vector space $\EE$. We denote by $\D (\EE)$ the space of
antiholomorphic polynomial maps $\EE \rightarrow \C$. Choosing an orthonormal
basis $(e_i)$ of $\EE$, we can identify $\EE$ with $\C^n$ and then define the scalar product of
$\D (\EE)$ by the formula \eqref{eq:scal_product_bargm}. Since the
weight $|z|^2$ and the measure $d \mu_n$ are invariant by unitary change of
coordinates, the resulting scalar product of $\D (\EE)$ is independent of
$(e_i)$. Similarly, we define the subspace $\symb (\EE)$ of the space of
endomorphisms of $\D (\EE)$ and associated to the basis $(e_i)$ of $\EE$, we have a
basis $(\Uu_{\al,\be} ,\; \al, \be \in \N^n )$ of $\symb (\EE)$. The decompositions
into even/odd elements are defined and denoted as for $\C^n$ by
\begin{gather}
  \D (\EE) = \D^+ (\EE) \oplus \D^{-} (\EE), \qquad \symb (\EE) = \symb ^{+} (\EE) \oplus
  \symb^{-}  (\EE). 
\end{gather}

We can extend all these constructions to vector bundles. Let $\EE \rightarrow M$ be a
Hermitian vector bundle with rank $n$. Define the infinite-dimensional vector
bundles $\D (\EE)$ and $\symb (\EE)$ over $M$ with fibers $\D (\EE)_x  = \D (\EE_x)$ and
$\symb (\EE)_x = \symb (\EE_x)$. Later, we will choose for $\EE$ the complex tangent bundle
of an almost-complex manifold, and we will construct operator whose symbols are
smooth sections of $\mathcal{S} (\EE) \otimes A$, where $A$ is an auxiliary
vector bundle.  Since the bundle $\symb(\EE)$ has infinite rank, let us make 
precise the definition of its smooth sections: a section $s \in \Ci ( M ,
\symb (\EE) \otimes A)$ is a family $(s(x) \in \symb (\EE_x) \otimes A_x, \; x \in M)$ such that for any
orthonormal frame $(e_i)$ of $\EE$ and $(a_j)$ of $A$ over the same open set $U$ of $M$, if $(\Uu_{\al,
  \be}(x))$ is the basis of $\symb (\EE_x)$ associated to the basis $(e_i(x))$,
then
$$ s (x) = \sum \la_{\al,\be,j} (x) \, \Uu_{\al,\be} (x) \otimes a_j(x), \qquad
x \in U $$
where the $\la_{\al, \be, j} $ are smooth functions on $U$, almost all equal
to zero.

\subsection{Operators} \label{sec:operators}

\subsubsection*{Schwartz kernel}

Consider a compact symplectic manifold $(M, \om)$ with a compatible
almost-complex structure $j$ and a prequantum bundle $L \rightarrow M$. 
Assume we have two auxiliary Hermitian vector bundles $A$ and $B$ over
$M$. The dimension of $M$ is $2n$ with $n \in \N$.

We will define a space $\lag ( A, B)$ consisting of families
of operators
\begin{gather} \label{eq:P_k_family}
  \bigl( P_k : \Ci (M, L^k \otimes A ) \rightarrow \Ci ( M , L^k
  \otimes B ), \; k \in \N \bigr) 
\end{gather}
having smooth Schwartz kernels satisfying some conditions. Let us first recall
some standard definitions and notations.

\cB We denote by $\con{A}$ the conjugate
 bundle of $A$ and by $A \boxtimes B$ the external tensor product of $A$ and
 $B$. \clb
The Schwartz kernel of $P_k$ is the section $K_k$ of $(L^k \otimes B) \boxtimes
(\con{L}^k \otimes \con{A})$ such that
$$ (P_k f )(x) = \int_M K_k (x,y)\cdot f(y) \; \mu_M ( y), \qquad \forall \;
f \in \Ci ( M , L^k \otimes A) $$
where the $\cdot$ stands for the scalar product $(\con{L}_y^k \otimes
\con{A}_y ) \times (L_y^k \otimes A_y) \rightarrow \C$, and $\mu_M  =
\om^n/n!$. We will denote the
operator and its Schwartz kernel by the same letter, hoping it is not too
confusing. 

Since $L$, $A$ and $B$
are Hermitian bundles, the bundle $(L^k \otimes B) \boxtimes (\con{L}^k
\otimes \con{A})$ has a natural metric, so the pointwise norm $|P_k (x,y)|$ is
well-defined. For any $N \in \N$, we will say that
$(P_k)$ is in $\bigo ( k^{-N})$ on an open set $U$ of $M^2$ if $|P_k (x,y)| =
\bigo ( k^{-N})$ for $(x,y) \in U$ with a $\bigo$ uniform on any compact subsets
of $U$. We say that $(P_k)$ is in $\bigo ( k^{-\infty})$ on $U$ if $(P_k)$ is in
$\bigo ( k^{-N})$ on $U$ for any $N$.

We will also use the uniform norm $\| P_k \| = \sup \|P_k (f) \| / \| f\|$
with respect to the usual $L^2$ norms of section: $\|f\|^2 = \int_M |f(x)|^2
\; d\mu_M (x)$.

\subsubsection*{Definition of $\lag (A,B)$}

\cA 

By definition, a family $(P_k)$ as in \eqref{eq:P_k_family} belongs to $\lag ( A , B)$
if each $P_k$ has smooth Schwartz kernel satisfying for any $N$ 
\begin{gather} \label{eq:exp_Lag_bis}
P_k (x,y) = \Bigl( \frac{k}{2\pi} \Bigr)^n E^{k}(x,y) \sum_{\substack{\ell \in
      \Z, \\ \ell + m ( \ell ) \leqslant N}} 
k^{-\frac{\ell}{2}} b_{\ell}(x,y) + \bigo ( k^{n - \frac{N+1}{2}})
\end{gather}
where 
\begin{enumerate} 
\item  $E$ is a section $L \boxtimes \con{L}$ such that $|E (x,y)| < 1$ for
  any $x\neq y$, and for any $y \in M$, the section $E_y (x) = E(x,y)$ of $L
  \boxtimes \con{L}_y$ satisfies $E_y (y) = u
  \otimes \con{u}$ for any $u \in L_y$ with $|u|=1$, 
 $(\nabla E_y )(y) =0$ and  
$$  ( \nabla_\xi \nabla_\eta E_y ) (y)  = -(\tfrac{i}{2} \om ( \xi, \eta) +
  \tfrac{1}{2} \om ( \xi , j \eta)) E_y (y), \quad \forall \; \xi ,\eta
  \in T_yM . $$
\item $m : \Z \rightarrow \N \cup \{ \infty \}$ is such that $\{ \ell; \; \ell
  + m ( \ell ) \leqslant N \}$ is finite for any $N$ and $\ell + m ( \ell )
  \geqslant 0$ for any $\ell$. Moreover, for any $\ell$, $b_{\ell} $ is a
  section of $B \boxtimes \con{A}$ vanishing to order $m(
  \ell)$ along the diagonal. 
\end{enumerate} 
As already mentionned in the introduction, the analytic meaning of the
expansion \eqref{eq:exp_Lag_bis} is not obvious, nevertheless we postpone the
explanations to Section \ref{sec:global-expansion}, cf. Lemma
\ref{lem:estimate} and Lemma \ref{lem:estim_expansion}. The existence of $E$ will
be proved in Section \ref{sec:section-ee}, it is not unique, \clb \cB any section $E$
satisfying the stated conditions can be used for the expansion
\eqref{eq:exp_Lag_bis}, but the coefficients $b_{\ell}$ depend on the choice
of $E$, cf. Lemma \ref{lem:changement_E}. \clr \cA

By rescaling the coordinates transverse to the diagonal by a factor
$k^{\frac{1}{2}}$, we can write the expansion \eqref{eq:exp_Lag_bis} in the
following alternative way. To simplify the statement, we assume first that that
$A$ and $B$ are the trivial line bundle $\C_M := M \times \C$ and let $\lag (
\C) := \lag (\C_M, \C_M)$. 

\begin{prop} \label{prop:un_de_plus}
  Let $(P_k)$ be an operator family $(P_k)$ of the form \eqref{eq:P_k_family}
  with smooth Schwartz kernels. Then $(P_k)$ belong to $\lag (\C)$ if and
  only if the  Schwartz kernel family is in $\bigo ( k^{-\infty})$ on $M^2
  \setminus \op{diag} M$ and for any coordinate chart $U \subset M$ and unitary frame $t : U \rightarrow L$, for
any $N \in \N$, we have over $U^2$ that
\begin{gather} \label{eq:expansion_kernel}
P_k ( x + \xi , x) = \Bigl( \frac{k}{2\pi} \Bigr)^n e^{- k\varphi (x, \xi) } 
\sum_{p =0 }^{N}  k^{- \frac{p}{2}} a_{p}  \bigl( x,
 k^{\frac{1}{2}} \xi \bigr)  + \bigo \bigl( k^{n - \frac{N+1}{2}} \bigr)
\end{gather}
where we have identified $L_{x+\xi}^k \otimes \con{L}_{x}^k \simeq \C$ by using
$t$ and 
\begin{itemize}
\item[-]  $\varphi( x, \xi ) = - i \bigl (\sum_{i=1}^{2n} \al_i (x)\xi_i + \frac{1}{2}\sum_{i,j=1}^{2n}
  (\partial_{x_i}\al_j)(x) \xi_i \xi_j \bigr) + \frac{1}{4} |\xi|_x^2 $, with
  $\al =
  \sum \al_i dx_i \in \Om^1 (U, \R)$ the connection one-form defined by $\nabla t = \frac{1}{i} \al \otimes t$. 
\item[-] $a_{p} (x,\xi) \in \C$ depends polynomially on $\xi$, meaning that for some
  $d(p) \in \N$,  $a_p (x,\xi) =
  \sum_{|\al |\leqslant d(p)} a_{p,\al} (x) \xi^\al$ with smooth coefficients
  $a_{p, \al}$.
\end{itemize}
\end{prop}

The proof is postponed to Section \ref{sec:global-expansion}. Since the real part of $\varphi
(x,\xi)$ is $\frac{1}{4} |\xi|_x^2$, we have for any $p$
$$ e^{-k \varphi (x, \xi)} a_p \bigl( x, k^\frac{1}{2} \xi \bigr)   = \bigo
(1).$$
So the $p$-th summand in
\eqref{eq:expansion_kernel} is in $\bigo (k^{n -\frac{p}{2}} \bigr)$ and the
expansion is meaningful. 
In the case where $A$ and $B$ are general vector bundles, we introduce frames
of $A$ and $B$ on $U$, so that the  Schwartz kernel $P_k$ on $U^2$ becomes a
$\C^r$-valued functions with $r =(\op{rank} A)( \op{rank} B)$, and we have the
same characterization  with $\C^r$-valued coefficients $a_p$.
\clb


The
advantage of the expansion \eqref{eq:expansion_kernel} is that its analytical
meaning is more transparent, the
drawback is that it depends on local choices (coordinates, frames, rescaling
$k^{\frac{1}{2}} \xi$) whereas the
expansion \eqref{eq:exp_Lag_bis} is global. 

\subsubsection*{Properties of $\lag (A,B)$}

$\lag (A , B)$ has a natural filtration defined as follows. For any $q \in
\N$,  $\lag _q ( A , B)$ is the subspace of $\lag (A,B)$ consisting of the
operators such that the local expansions \eqref{eq:expansion_kernel} hold
with a sum starting at $p = q$, that is the coefficients $a_0$, \ldots
$a_{q-1}$ are zero.  

\begin{prop} \label{prop:lag} $ $
  \begin{enumerate}
  \item    $\lag _q ( A, B) = k^{-\frac{q}{2}}
    \lag (A, B)$ and if $q \geqslant q'$, then $\lag_q( A,B) \subset \lag_{q'} (A,B)$.
  \item  For  any $(P_k)$ in $\lag (A, B)$, 
  \begin{xalignat*}{2}
 &  (P_k) \in \lag_q(A,B)  \Leftrightarrow  \| P_k \| = \bigo (k^{-\frac{q}{2}}) \\
   & \Leftrightarrow  \text{the Schwartz kernel family of $(P_k)$ belongs to }
     \bigo ( k^{n-\frac{q}{2}}) 
 \end{xalignat*}
\item  $\lag_{\infty} (A,B) :=  \bigcap_q \lag_q( A,B)$ consists of the families
  \eqref{eq:P_k_family} with a smooth Schwartz kernel in $ \bigo ( k^{-\infty})$.
\item for any sequence $(P_q)_{q \in \N}$ of $\mathcal{L} (A,B)$ such that $P_q \in \lag_q(A, B)$ for any
$q$, there exists $P \in \lag (A,B)$ satisfying $P = \sum_{p=0}^q P_p$
modulo $\lag_{q+1} (A,B)$ for any $q$.
\end{enumerate}
\end{prop}

We will now describe the quotients $\lag_q ( A,B)/ \lag_{q+1} (A,B)$ by using
the material  introduced in Section \ref{sec:symbol-spaces}. Since $M$ has an almost
complex structure $j$ compatible with $\om$, the tangent bundle $TM$ is a complex vector bundle with
a Hermitian metric, which defines our bundle $\symb (M) := \symb (TM)$. 

\begin{theo} \label{theo:lag}
  For any $q \in \N$, there exists a linear map
  $$\si_q: \lag_q (A,B)  \rightarrow \Ci ( M, \symb (M) \otimes \op{Hom}
  (A,B))$$  which is onto and has  kernel $\lag_{q+1} (A,B)$. Furthermore, the
  following holds for any $ P \in \lag_q ( A,B)$: 
  \begin{enumerate} 
\item $\si _q ( P) = \si_0 ( k^{\frac{q}{2}} P)$.
\item For any  $f \in  \Ci ( M , \op{Hom} (B,
  C))$, $( f \circ P_k )$ belongs to $\lag_q ( A,C)$ and $\si_q( f \circ P)
  = f \circ \si_q (P)$. For any $g \in \Ci ( M , \op{Hom} ( C, A))$, $( P_k
  \circ g)$ belongs to $\lag_q ( C,B)$ and $\si_q ( P \circ g) = \si_q ( P)
  \circ g$.
  \item  $P^* $ belongs to $\lag_q
    ( B,A)$ and $\si (P^*) = \si (P)^*$.
    \item For any  $P' \in \lag_{q'}(B,C)$,
      $P' \circ P$ belongs to $\lag_{q'+ q }( A,C)$ and
      $$ \si_{q'+ q} ( P' \circ P) = \si_{q'} (P') \circ \si_q (P).$$
    \item The Schwartz kernel of $P_k$ on the diagonal satisfies
      $$ P_k (x,x) = \frac{k^{n-\frac{q}{2}}}{ (2\pi)^n} \Bigl[ \op{tr} ( \si_q( P)(x))
      + \bigo ( k^{-\frac{1}{2}}) \Bigr] $$
      where  $\op{tr}$ is the map $\symb (T_xM) \otimes \op{Hom}
      (A_x,B_x) \rightarrow (L_x \otimes \con{L}_x)^k \otimes B_x \otimes
      \con{A}_x \simeq \op{Hom}
      (A_x,B_x)$ sending $s \otimes f$ to $(\op{tr} s) f$.
    \item \cB The operator norm of $P_k$ satisfies
      $$ \| P_k \| = k^{-\frac{q}{2}}  \bigl( \sup_{x \in M} \|\si_q ( P) (x) \| +
      \op{o} (1) \bigr)$$
      where $\|\si_q ( P) (x) \|$ is the operator norm for the norm of $\D (
      T_xM)$ corresponding to the scalar product
      \eqref{eq:scal_product_bargm}. \clb
  \end{enumerate}
\end{theo}

Let us explain how is defined the symbol map $\si_0$ for $A= B = \C_M$.
Consider $P \in \lag (A,B)$ and the local expansion
\eqref{eq:expansion_kernel}. We view $(x,\xi)$ as a tangent vector of $M$,
that is  $
\xi \in T_xM$, so we consider $a_0 (x, \cdot)$ as a
polynomial of $T_xM$. Then   it is not obvious but nevertheless true that this polynomial does not depend on
the choice of the coordinate chart $U$ and the unitary frame $t$. To compare, the
coefficients  $a_p$ in \eqref{eq:expansion_kernel} with $p \geqslant 1$ do
depend on the choice of the coordinates and the frame of $L$.

To pass from $a_0(x,\cdot)$ to the symbol of $P$ at $x$, we first choose a
unitary frame $(e_i)$ of $T_xM$. So $T_xM \simeq \C^n$ by sending $\xi = \sum
z_i e_i(x) $ to $z(\xi)= (z_i)$. We also have a basis
 $\Uu_{\al, \be}(x)$ of $\symb (T_xM) $ defined in
\eqref{eq:def_Ualphabeta}. Then
$$ \si_0(P) (x) = \sum f_{\al, \be} (x) \Uu_{\al,\be} (x) \Leftrightarrow
a_0 (x,\xi) =  \sum f_{\al, \be} (x) \Ppp_{\al, \be} (z(\xi))$$
where we use the polynomials
$ \Ppp_{\al, \be} (z) =  \bigl( \frac{1}{\al! \be !} \bigr) ^{1/2} \bigl(
\partial_z - \con{z})^{\al} z^\be .$ These polynomials form a basis of $\C [z,
\con{z}]$, cf. proof of Proposition \ref{prop:tilde_symb_and_op}.

\begin{exe} \label{exemple:projecteur}
Choose a connection of $A$ and let $\Delta_k$ be the Laplacian
\begin{gather}  \label{eq:laplacien_fibre_auxiliaire}   
\Delta_k = \tfrac{1}{2} (\nabla^{L^k
  \otimes A})^* \nabla^{L^k \otimes A} : \Ci ( L^k \otimes A) \rightarrow \Ci
( L^k \otimes A).  
\end{gather}
For any $m \in \N$, let $\Pi_{m,k}$ be the spectral projector
\begin{gather} \label{eq:projecteur_fibre_auxiliaire}  
\Pi_{m,k} := 1_{[m - \frac{1}{2}, m +
  \frac{1}{2}]} ( k^{-1} \Delta_k).
\end{gather}
By \cite{oim_copain}, the family $(\Pi_{m,k})$ belongs to $\lag(A,A  )$, its
  $\si_0$-symbol at $x$ is $\pi_m(x) \otimes \op{id}_{A_x}$ where $\pi_m(x)$ is the projector of $\D (T_xM)$ onto the
  subspace $\D_m(T_xM)$ of homogeneous degree $m$ polynomials.
  
  Since  $
  \pi_m (x) = \sum_{|\al|=m} \rho_{\al, \al}(x)$, if the
  auxiliary bundle $A$ is trivial,  the corresponding
  function $a_0$ is  
\begin{gather} \label{eq:a_0laguerre}
  a_0(x, \xi) = \sum_{|\al| =m} p_{\al,\al}(z(\xi)) = Q_m^{(n-1)} (
  |z(\xi)|^2)
\end{gather}
where $Q_m^{(p)} $ is the Laguerre polynomial  $Q^{(p)}_m (x) = \frac{x^{-p}}{m!} \bigl (
\frac{d}{dx} -1 \bigr) ^m  x ^{m+p}$. The second equality in
\eqref{eq:a_0laguerre} follows from $p_{m,m} ( z) = Q_m^{(0)} (|z|^2)$ 
and the identity 
$$Q_m^{(n-1)} (x_1+ \ldots + x_n) =
\sum_{|\al| = m } Q_{\al(1)}^{(0)} (x_1) \ldots  Q_{\al(n)}^{(0)} (x_n).$$
Actually we won't use the expression in terms of Laguerre polynomials, what
really matters is the fact that $\si_0 ( \Pi_m) (x)$ is the orthogonal
projector onto $\D _m(T_xM)$. \qed
\end{exe}

The definition of the symbol map $\si_0$ is motivated by Theorem
\ref{theo:lag} and its efficiency in the proofs of Section
\ref{sec:projectors-Toeplitz}. But this definition does not explain why it is
natural to associate to $P \in \lag (A,B)$ an endomorphism of $\D
(T_xM)$. A first explanation is provided by the following construction of peaked
sections. A deeper reason will be provided later in Section
\ref{sec:landau-levels-cn}. 

We still assume that $A = B = \C_M$ to simplify the exposition.
 Let $x \in M$ be a base point, with a coordinate chart $U$  at
$x$ and a unitary frame $t: U \rightarrow L$. Let $\psi \in \Ci_0(U)$ be equal
to $1$ on a neighborhood of $x$.
  To any $f \in \D ( T_xM )$, we associate a family $\Phi_k^f \in \Ci ( M , L^k)$ defined by
$$ \Phi_k^f (x+\xi ) = \Bigl( \frac{k}{2\pi} \Bigr)^{\frac{n}{2}} e^{-k
  \varphi(x,\xi)} \, f(k^{\frac{1}{2}} \xi) \, \psi (x+\xi) \, t
^k(x+\xi) , \qquad k \in \N $$
where $\varphi$ is the same function as in \eqref{eq:expansion_kernel}. \cB Let
$\| f \| = \sqrt{(f,f)}$ be the norm associated to the scalar product
\eqref{eq:scal_product_bargm}. \clr

\begin{prop} \label{prop:peaked-sections}
  For any $f \in \D (T_xM)$,
\begin{gather} \label{eq:peaked_1}
  \| \Phi^f_k \| = \| f \| + \bigo ( k^{-1/2}),
\end{gather}
and for any $P \in \lag ( \C_M, \C_M)$,
\begin{gather} \label{eq:peaked_2}
  P_k \Phi^f_k = \Phi^g_k + \bigo ( k^{-1/2})
\end{gather}
where $g = \si_0(P)(x) \cdot f$.
\end{prop}
By \eqref{eq:peaked_1} the map sending $f$ into $(\Phi_k^f)$ is injective, so  \eqref{eq:peaked_2}  characterizes the symbol $\si_0 (
P) (x)$. 
For a more general result with auxiliary bundles $A$, $B$ and the estimates of the
scalar product of peaked sections, cf. Proposition \ref{prop:peaked-sections_++}.

\cA We say that an element  $P$ in $\lag (A,B)$ is even (resp. odd)   if 
  the expansion \eqref{eq:exp_Lag_bis} holds with
 $b_{\ell } = 0$ for any odd $\ell \in \Z$ (resp. even).
 \begin{lemme}\label{lem:parite}
    For any $P$ in $\lag (A,B)$, $(P_k)$ is even (resp. odd) if and only if
 in the local
expansions \eqref{eq:expansion_kernel}, every polynomial $a_p(x,\cdot)$ has
the same (resp. the opposite) parity as $p$.
\end{lemme}
\clb

Denote by $\lag^+ ( A, B)$ and $\lag^- ( A,
B)$ the subspaces of even and odd elements respectively. 

\begin{theo} \label{theo:parity}
  We have  \begin{enumerate}
\item $ \lag ( A, B) =  \lag^+ ( A, B) +  \lag^- ( A, B)$, $\lag^+ (A,B)
  \cap \lag ^{-} (A,B) = \lag_{\infty} (A,B)$. 
\item $\lag ^{\ep} (A,B) \cdot \lag^{\ep'} (B,C) \subset \lag ^{\ep \ep'}
      (A, C) $ for any choice of signs $\ep$, $\ep'$. 
\item  $\si_q( \lag_q(A,B) \cap \lag^{\ep} (A,B) ) = \Ci ( M,
      \symb^{\ep (-1)^q}(M) \otimes \op{Hom} (A,B))$. 
 \end{enumerate}
  \end{theo}

  The proofs of Proposition \ref{prop:lag}, Theorem \ref{theo:lag} and Theorem
  \ref{theo:parity} are postponed to Section \ref{sec:proofs-results}. The
  proof of Lemma \ref{lem:parite} is at the end of Section \ref{sec:global-expansion}.

  \subsection{Comparison with earlier works}

  The expansions \eqref{eq:exp_lag}, \eqref{eq:expansion_kernel} or similar versions appeared in
  the literature \cite{ShZe}, \cite{oim_op}, \cite{MaMa} to describe Bergman kernels of ample line bundles and their
symplectic generalizations  as well as the associated Toeplitz operators. In a
more general context, the Boutet de Monvel-Guillemin
theory \cite{BoGu} is built on two classes of operators:  Hermite
operators and Fourier integral operators respectively. The spaces $\lag^{+}
(A,B)$ may be viewed as an intermediate choice in the semi-classical setting.

  In \cite{oim}, we considered a subalgebra of  $\lag^{+} ( A,A)$,
  denoted by $\mathcal{A}(A)$, consisting of operators having an expansions \eqref{eq:expansion_kernel} in which
  each $a_p(x,\cdot)$ has degree $\leqslant \frac{3}{2} p$. For our
  applications in this paper, it is necessary to consider the larger spaces $\lag (A,B)$,
  because our generalized projectors $\Pi_m$ and unitary equivalences do not
  belong to $\mathcal{A} (A)$. More precisely, only the projector
  corresponding to the first Landau level belongs to $\mathcal{A} (A)$.

  Theorem \ref{theo:lag} is a generalization of similar results for
  $\mathcal{A}(A)$ established in \cite{oim}, and surprisingly the proofs are  somehow easier in this new generality. 
  However, a crucial difference with \cite{oim} relies in the 
  symbols. Roughly, the symbols of the elements of $\mathcal{A}(A)$ were
  defined directly as the polynomials $a_0 (x,\cdot)$.  This  had the
advantage that it is easier to pass from the Schwartz kernel of the
operator to the symbol. The drawback of
this definition is that the product for these symbols is given by the
mysterious formula
\begin{gather} \label{eq:old_prod}
(u \star v) ( x, z, \con{z})  =\bigl[ \exp ( \Box  ) ( u (x, - \zeta, \con{z} - \con {\zeta} ) v (x, z+ \zeta , \con{\zeta} )
 )\bigr]_{\zeta = \con{\zeta} =0 }   
\end{gather}
where $\Box = \sum \partial^2 / \partial \zeta^i \partial \con{\zeta}^i$. This
was actually tractable for what we did in \cite{oim} because our main interest
was the projector $\Pi$ for the first Landau level,  with symbol $\si_0(\Pi)=
\Uu_{00}$ and associated function symbol $u(x,\xi) =1$. But for our new projectors whose symbols are typically
a sum of the $\Uu_{\al\al}$'s, it is
essential to
work with the symbol in $\symb (M)$.  For instance, it is even not obvious
how to recover the relations \eqref{eq:rel_U} from the product 
\eqref{eq:old_prod}.

\section{Projectors of \texorpdfstring{$\lag (A)$}{L(A)} and Toeplitz
  operators} \label{sec:projectors-Toeplitz} 

In this section, we consider an auxiliary Hermitian vector bundle $A$ with
arbitrary rank. We
denote by $\lag (A): = \lag (A,A)$ the associated algebra and by $\lag^+(A)$
the subalgebra consisting of even elements. The symbols of operators of $\lag
(A)$ are sections of $\symb (M) \otimes \op{End} A$. We will view
$\symb (M)_x \otimes \op{End} A_x$ as a subspace of $\op{End} ( \D (T_xM)
\otimes A_x)$.

Let $F$ be  a subbundle of $\D_{\leqslant m}(TM)
\otimes A$ for some $m \in \N$, where $\D_{\leqslant m} (T_xM)$ is
the subspace of $\D(T_xM)$ of polynomial with degree $\leqslant m$. We assume
that $F$ has a definite parity, that is $F \subset \D ^\ep (TM) \otimes A$ with
$\ep \in \{\pm 1\}$. Associated
to $F$ is the section $\pi$ of $\symb (M) \otimes \op{End} A$ such that $\pi(x)$ is the
orthogonal projector of $\D(T_xM) \otimes A_x$ onto $F_x$ at each point $x \in M$. The
content of the following subsections is:
\begin{itemize}
  \item[-] 
\ref{sec:constr-proj}: we construct  a selfadjoint projector
$\Pi \in \lag ^+(A)$ with symbol $\pi$.
\item[-] \ref{sec:toeplitz-algebra}:  we study the
Toeplitz algebra 
$ \mathcal{T} = \{ \Pi P \Pi , \; P \in \lag^{+} (A) \} $
\item[-]  \ref{sec:unitary-equivalence}: we prove $(
  \op{Im} \Pi_k)$ is isomorphic with any quantization of $(L,M)$ twisted by
  $F$, and deduce that the dimension of $ \op{Im} \Pi_k$ is the Riemann-Roch number of $L^k \otimes F$ when $k$ is sufficiently
large.
\end{itemize}

A possible choice for $F$ is $F = \D_m (TM)\otimes A $ where $m\in \N$ and $\D_m
(T_xM)$ is the subspace of $\D (T_xM)$ consisting of homogeneous polynomials
with degree $m$. As explained in example \ref{exemple:projecteur}, the
projector $\Pi_{m,k}$ onto the $m$-th Landau level $\Hilb_{m,k}$  belongs to $\lag^+ (A)$
and has symbol the projector onto $F$, so it can be used as the projector $\Pi$. Theorems \ref{theo:dim_landau}, \ref{theo:Toeplitz_Landau} and
\ref{theo:ladder} will mainly follow from the results in Sections
\ref{sec:toeplitz-algebra} and  \ref{sec:unitary-equivalence}.


By \cite{oim_copain}, the spectral projectors of Laplacians with a
magnetic field, not necessarily constant but still satisfying some convenient
assumptions, give other instances of projectors in $\lag ^+(A)$. 

Another choice for $F$ is $F = \D_0 (TM) \otimes A$ where $\D_0 (T_xM) =\C$ is the
subspace of $\D ( T_xM)$ of constant polynomials. The corresponding quantum
space  and Toeplitz algebra is the quantization of $(M,L)$ twisted by $A$.

A last example is the Spin-c Dirac quantization twisted by an auxiliary bundle
$B$. In this case, $A = S \otimes B$ where $S$ is the
spinor bundle \clr $\bigoplus \bigwedge^k (T^*M)^{0,1} $ and $F = \D_0(TM)
\otimes \bigwedge^0 (T^*M)^{0,1} \otimes B$. \clb
This example will be used to compute the dimension of our quantum spaces
from the Atiyah-Singer Theorem.

\subsection{Construction of the projector}\label{sec:constr-proj}
Let $\chi : \R \rightarrow \R$ be defined by $\chi (x) =1$ if $x \geqslant
\frac{1}{2}$ and $\chi (x) =0$ otherwise. If $P$ is a bounded self-adjoint
operator of a Hilbert space $\mathcal{H}$, then using the functional calculus
for Borel bounded functions, we define a new bounded operator $\chi (P)$ of
$\mathcal{H}$, cf. as instance \cite[Theorem VII.2]{ReSi}. Since $\chi$ is
real valued and $\chi^2 = \chi$, $\chi (P)$ is a self-adjoint projector. 

\begin{theo} \label{theo:constr-proj} 
Let $P \in \lag ( A)$ be self-adjoint and having symbol
  $\si_0 (P) = \pi$. Then $\chi(P)$ belongs to $\lag (A)$ and
  $\si_0(\chi (P)) = \pi$. If furthermore $P \in \lag ^+ (A)$, then $\chi (P)$
  is in $\lag ^+ (A)$. 
\end{theo}

\cB An operator $P$ satisfying the assumptions exists by the surjectivity of
$\si_0$, cf. Theorem \ref{theo:lag}. \clb Theorem  \ref{theo:constr-proj} holds without the assumption that $F$ has a
definite parity. When $F$  does have a definite parity, $\pi$ is even, so
we can choose  $P \in \lag^{+}(A)$ with symbol $\pi$. 

\begin{proof}
To prove that the $\chi (P_k)$'s have smooth kernels, we will use the following
basic fact:  let $Q$, $Q'$ be two operators with smooth kernels acting on $\Ci (M,A)$ and $Q''$ be a bounded operator of $L^2(M, A)$. Then $Q Q''Q'$ has a
smooth kernel. This follows from the Schwartz theorem saying that
the operators with smooth kernel are the operators which can be continuously
extended $\mathcal{C}^{-\infty} \rightarrow \Ci $. We will also need 
the following pointwise norm estimates: consider families of operator $Q_k, Q'_k : \Ci
(M,L^k \otimes A) \rightarrow \Ci (M, L^k\otimes A)$ and $Q''_k : L^2 (M, L^k
\otimes A) \rightarrow L^2 (M, L^k \otimes A)$. Then by  \cite[Section
4.3]{oim}, if the Schwartz kernel
families of $(Q_k)$ and $(Q'_k)$ are respectively  in $\bigo (k^{-N})$ and
$\bigo (k^{-N'})$, and the operator norms of $Q_k^{''}$ are in $\bigo (1)$,
then the Schwartz kernel family of $Q_k Q''_k Q_k'$ is in $\bigo (k^{-(N
  +N)'})$. 

Back to our problem, we can write $\chi(P_k) = P_k \tilde{\chi} (P_k) P_k$ with $\tilde{\chi}
(x) = \chi(x) / x^2$. Since $ \tilde{\chi}(P_k)$ is bounded, this shows that $\chi (P_k)$
has a smooth kernel. This also shows that the Schwartz kernel family of $\chi
(P)$ is in
$\bigo (k^{2n})$.  To improve this, observe that $Q = P^2 - P$ is in $\lag
(A)$ and $\si _0 (Q) = \pi^2 - \pi =0$, so $\| Q_k \| = \bigo (k^{-1/2})$, which implies easily that $\frac{1}{2}$ is not in the spectrum of $P_k$ when $k$ is
sufficiently large, cf. \cite[Proposition 4.2]{oim}.

Now for $x \in \R \setminus \{ \frac{1}{2}\}$,
$y = x^2 -x > -1/4 $ and we have
$$ \chi (x) = x + (1-2x) f ( x^2 -x ) \quad \text{with } \quad  f (y) = \tfrac{1}{2}
( 1-  (1+ 4y)^{-1/2} )$$
\cA For any $m \in \N$, write the Taylor expansion of $f$ at $0$ at order $m $ as
follows: \clb $f (y) = \sum_{\ell =0}^m a_\ell
y^\ell + y^{m+1} f_m (y)$ with $f_m \in \mathcal{C}^0(]-\frac{1}{4},
\infty[, \R)$.  Then
\begin{gather} \label{eq:toutestla}
  \chi (P) =  P + \sum_{\ell = 0 }^m a_\ell (1-2P) Q^\ell + (1 - 2 P )
  Q^{m+1} f_m (Q).
\end{gather}
Now $\si_0(Q)=0$ implies that $Q^\ell$ and $PQ^{\ell}$ belong both to
$\lag_{\ell} ( A)$. Furthermore, $\|f_m(Q_k) \| = \bigo (1)$. 
Since $Q^{m+1} f_m (Q) = Q^m f(Q) Q$ and similarly $PQ^{m+1} f_m (Q) = PQ^m
f(Q) Q$, it follows from the preliminary observation  that the Schwartz
kernel family of $(1 - 2 P )
  Q^{m+1} f_m (Q)$ is in $\biginf ( k^{2n -m} )$.  We can now conclude easily
  the proof  from \eqref{eq:toutestla} by choosing at each step sufficiently large value of
  $m$:  first the
  Schwartz kernel family of $\chi (P)$ is in $\bigo (k^{-\infty})$ outside the diagonal and second the local expansions
\eqref{eq:expansion_kernel} hold. 
\end{proof}

\subsection{Toeplitz algebra} \label{sec:toeplitz-algebra}

Choose  a self-adjoint projector $\Pi \in \lag ^+ (A)$ with symbol $\pi$,
which exists by Theorem \ref{theo:constr-proj}. For any $k\in \N$, let
$\Hilb_k  = \op{Im} \Pi_k \subset \Ci ( M , L^k \otimes A)$. Computing
the trace of $\Pi_k$ by integrating its Schwartz kernel over the diagonal, we
deduce from the last assertion of Theorem \ref{theo:lag} that $\Hilb_k$ is finite dimensional
and \cB
\begin{gather} \label{eq:estim_dim}
\op{dim} \Hilb_k \sim \Bigl( \frac{k}{2 \pi} \Bigr)^n (\op{rank} F) \;
\op{vol} (M, \om) 
\end{gather}
As we will see later, when $k$ is sufficiently large, this dimension depends
polynomially on $k$ and is a Riemman-Roch number, cf. Theorem
\ref{theo:dim_general}. \clb

We will now  work with families of operators $(T_k \in \op{End} \Hilb_k$, $k \in
\N)$. Equivalently, we can consider that each $T_k$ acts on the larger space $\Ci ( M , L^k \otimes
A)$ and satisfies $\Pi_k T_k \Pi_k = T_k$. 
Define the space
\begin{gather} \label{eq:Toeplitz_Pi}
\mathcal{T} = \{ T \in \lag ^{+} (A) / \, \Pi T \Pi = T \} .
\end{gather}
For any $q \in \N$, set $\mathcal{T}_q  := k^{-q} \mathcal{T} =
\lag_{2q} (A) \cap \mathcal{T}  $ and $\mathcal{T}_{\infty} = \bigcap_q  \mathcal{T}_q $. Clearly,
$$  \mathcal{T}_\infty
\subset  \mathcal{T}_q
\subset \mathcal{T}_p \subset  \mathcal{T} $$ when $q \geqslant p$. 

\begin{theo} \label{theo:toeplitz} 
  \hspace{1em} 
  \begin{enumerate}
  \item For any $T \in \mathcal{T} $, $T$ belongs to $\mathcal{T}_q
    $ iff $\| T_k \| = \bigo ( k^{-q})$.
     Furthermore, $\mathcal{T}_{\infty}$ consists of the families $(T_k \in \op{End}
  (\Hilb_k), \, k \in \N^*)$ such that $\| T_k \| = \bigo(k^{-N})$ for any $N$.   
 \item 
  $ \mathcal{T} $ is closed under composition and taking adjoint: $ \bigl( \mathcal{T}_q \bigr)^* =  \mathcal{T}_q$ and 
  $\mathcal{T}_q  \cdot \mathcal{T}_p \subset \mathcal{T}_{q+p} $ for any $q$ and $p$.
 \item For any $q$, there exists a linear map $\tau_q : \mathcal{T}_q  \rightarrow
   \Ci ( M , \op{End} F ) $, which is onto, has kernel
   $\mathcal{T}_{q+1} $ and is determined by $\si_{2q} ( T) = \tau_q( T)
   \pi$.  Furthermore, if $P \in \mathcal{T}_q $ and $Q \in
   \mathcal{T}_p $, then
   \begin{xalignat}{2}
     \begin{split} &  \tau_q( P) = \tau_0(k^{q}P) , \qquad \tau_q( P ^*) =
   \tau_q( P)^* , \\ &
   \tau_q ( P ) \tau_p( Q) = \tau_{q+p} ( PQ)  \\
   & \| P_k \| = k^{-q} \bigl( \sup \{ \| \tau_q (
    P)_x \| , \; x \in M \} + \op{o} (1) \bigr),
 \end{split}
\end{xalignat}
and the restriction to the diagonal of the Schwartz kernel of $P_k$
   satisfies
   \begin{gather} \label{eq:trace_diag} 
 P_k (x,x) = \frac{k^{n-q}}{(2\pi)^n} \Bigl[ \op{tr} (\tau_q ( P) (x)) +
   \bigo (k^{-1}) \Bigr].
 \end{gather} 
\end{enumerate}
\end{theo}

Let us give more details on the equation $\si_{2q} ( T) = \tau_q( T)
   \pi$ defining the symbol map $\tau_q$. Recall that $\pi(x)$ is the
   orthogonal projector of $\D (T_xM) \otimes A_x$ onto $F_x$. 

Then for an endomorphism $s$ of $F_x$, we define $s
\pi(x) \in \symb (T_xM) \otimes \op{End} A_x $
as the endomorphism of $\D (T_xM) \otimes A_x$ sending $\psi $ into $s ( \pi(x) \psi)$.

\begin{proof} 1. The first assertion follows from Part 2 of Proposition \ref{prop:lag}. To establish the second assertion, we deduce from the first part of the
  proof of Theorem \ref{theo:constr-proj} that if a family $(P_k \in \op{End}
  (\Hilb_k))$ satisfies $\| P_k \| = \bigo ( k^{-\infty})$, then its Schwartz
  kernel is in $\bigo ( k^{-\infty})$ because $ \Pi_k P_k \Pi_k = P_k$ and
  the Schwartz kernel of $\Pi_k$ is in $\bigo (k^{n})$.

  Parts 2 and 3 follow
  from Theorem \ref{theo:lag} and the fact that $\lag ^{+} (A)$ is a
  subalgebra of $\lag (A)$ by theorem \ref{theo:parity}. To define $\tau_q
  (P)$, simply observe that $\Pi P \Pi = P$ implies that $\pi \si_{2q} (P) \pi
  = \si_{2q}(P)$, so we can write $\si_{2q} (P) = \tau_q(P)  \pi$ with $\tau_q
  (P)$ a section of $ \op{End} F$.
  The map $\tau_q$ is onto
  because for any section $s$ of $\op{End} F$, there exists $P$ in $ \lag_{2q} (A)$ with $\si_{2q} (P ) = s \pi$. Since $\pi (s \pi) \pi = s
  \pi$, we have $\si_{2q} ( \Pi P \Pi) = s \pi$ and clearly $\Pi P \Pi \in
  \mathcal{T}_q$.

  The kernel of $\tau_q$ is $\mathcal{T}_{q+1}$ because
  $\tau_q (P) =0$ implies that $\si_{2q} (P) =0$ so $P \in \lag^{+}_{2q+1} (A)$ so $\si_{2q+1} (P)$
  is odd by Theorem \ref{theo:parity}. But $\Pi P \Pi = P$ implies that $\si_{2q+1} ( P  ) = \pi \si_{2
    q+1} (P) \pi$. This implies that $\si_{2q+1} (P)$ is even because  $F$ has
  a definite parity. Indeed, if for instance $F \subset \D^+(TM) \otimes A$,
  and $f = \pi g \pi$ with $g \in \symb (T_xM) \otimes \op{End} A_x$, then
 in the decomposition $$\D (T_xM) \otimes A_x = (\D^+ (T_xM) \otimes A_x) \oplus ( \D^{-} (T_xM) \otimes
 A_x),$$ $f$ has the form $\begin{pmatrix} f_{++} & 0 \\ 0 & 0 
 \end{pmatrix}$, so $f$ is even.  Consequently, $\si_{2q+1}
  (P) = 0$ so $P \in \lag_{2q+2} (A)$.
The formulas giving the symbol of products, adjoints, the operator norm and the Schwartz kernel on
the diagonal follow directly from Theorem \ref{theo:lag}. \cB Observe that the
$\bigo ( k^{-\frac{1}{2}})$ in Assertion 5 of Theorem \ref{theo:lag} becomes a
  $\bigo ( k^{-1})$ in \eqref{eq:trace_diag} because $P$ being even, the
  restriction of the
  asymptotic expansion of its Schwartz kernel \eqref{eq:exp_Lag_bis} to the
  diagonal only involves integral powers of $k^{-1}$. \clb
\end{proof}

\begin{rem} \label{rem:odd_toeplitz}
We can consider as well {\em odd Toeplitz operators}, that is $T \in \lag^{-} (A)$ such that $T =
\Pi T \Pi$. The space of these operators is $k^{-\frac{1}{2}} \mathcal{T}$. Indeed, if $T $ is such
an operator, then its symbol $ \si_0 ( T) = \pi \si_0 (T) \pi$ is at the same
time even and odd because $F$ has a definite parity, so $\si_0(T) =0$ so $T
\in \lag_1 (A)$, so $k^{\frac{1}{2}} T \in \lag (A)$ and is even, so
$k^{\frac{1}{2}} T \in \mathcal{T}$.  \qed
\end{rem}

For any $f \in \Ci(M)$ and $k \in \N$, define the endomorphism $T_k(f)$ of
$\Hilb_k$ such that
$$ \langle T_k (f) \psi, \psi' \rangle = \langle f \psi , \psi'\rangle ,
\qquad \forall \psi, \psi' \in \Hilb_k .$$
Viewed as an operator of $\Ci ( M, L^k \otimes A)$, $T_k(f)$ is merely $\Pi_k
f \Pi_k$. It follows from part 2 of Theorem \ref{theo:lag} that the family
$(T_k(f))$ belongs to $\mathcal{T}$ and has symbol  $\tau_0 ( T_k (f))= f
\op{id}_{F}$. By part 4 of Theorem \ref{theo:toeplitz}, we deduce
that 
$$ T_k (f) T_k (g) = T_k (fg) + \bigo (k^{-1}).$$
A consequence of Theorem \ref{theo:unitary_equivalence} will be that
$$ [ T_k (f) , T_k (g) ] = i k^{-1} T_k ( \{ f,g \} ) + \bigo ( k^{-2})$$
with $\{ f, g \}$ the Poisson bracket of $f$ and $g$ with respect to $\om$. This equality
does not follow from Theorem \ref{theo:toeplitz}. However, \cB the center of
$\op{End} F_x$ consisting on the scalar multiple of the identity, \clb  the following
characterization of the Toeplitz operators having a scalar symbol follows from Theorem \ref{theo:toeplitz}: for any $P \in
\mathcal{T} $, 
$$ P = T (f) +\bigo (k^{-1}) \text{ for some } f \in \Ci (M) \, \Leftrightarrow \;
\forall Q \in \mathcal{T}, \, [P,Q ] \in \mathcal{T}_{1} .$$
By Jacobi identity, this proves that $[ T_k (f), T_k (g) ] = k^{-1} T_k (h) +
\bigo ( k^{-2})$ for some function $h \in \Ci(M)$. 

\subsection{A unitary equivalence}\label{sec:unitary-equivalence}

Consider an auxiliary vector bundle $B$ with an arbitrary rank. Set $F' = \D_0
(T M) \otimes B $ with $\D_0 (T_xM) \subset \D (T_xM)$ the subspace of constant
polynomials. Then by Theorem \ref{sec:constr-proj}, there  exists a projector $\Pi'$ in $\lag^{+}(B)$ having symbol
$$\si_0(\Pi') = \Uu_{00} \otimes \op{id}_{B} \in \Ci ( M,\symb(M) \otimes \op{End}
B),$$ where  $\Uu_{00} (x) \in \symb (T_xM)$ is the orthogonal projector of
$\D (T_xM)$ onto $\D_0 (T_xM)$, the notation being the same as in  \eqref{eq:def_Ualphabeta}.

Starting from $\Pi'$, we define $\Hilb_k' := \op{Im} \Pi'_k$ and the
corresponding Toeplitz space $\mathcal{T}' := \{ P
\in \lag^{+} (B)/ \; \Pi' P \Pi' = P \}$. Since $\D_0(TM) = \C$, $F' \simeq B$,
so the symbols of the Toeplitz operators of $\mathcal{T}'$ are sections of
$\op{End} B$:
$$ 0 \rightarrow \mathcal{T}'_{q+1} \rightarrow \mathcal{T}'_q
\xrightarrow{\tau'_q} \Ci ( M , \op{End} B) \rightarrow 0 .$$
Our goal now is to establish an equivalence between the $(\Hilb_k, \mathcal{T})$ 
and $(\Hilb_k', \mathcal{T}')$ when the bundle $B$ is  $F$. The
critical point is the existence of a convenient symbol. Recall our assumption that $F \subset
\D ^{\ep} (TM) \otimes A$ with $\ep \in \{ \pm 1 \}$. 

\begin{lemme} \label{lem:unitary-symbole}
  If $B = F$, then  there is a canonical symbol $\rho \in \Ci (M , \symb^{\ep} (M) \otimes \op{Hom} (A, B))$ such that $ \rho^* \rho = \pi $ and $\rho \rho ^* = \Uu_{00}  \otimes \op{id}_{B} .$
\end{lemme}
\begin{proof} On the one hand, $\pi (x) $ is the
  orthogonal projector of $\D (T_xM) \otimes A_x$ onto $F_x$. On the other hand, $\pi'(x) := \Uu_{00} (x) \otimes \op{id}_{B_x}$ is the
  orthogonal projector of $\D (T_xM) \otimes B_x$ onto $\C \otimes B_x$. Since
  $B = F$, the images of $\pi (x)$ and $\pi'(x)$ are isomorphic by the map
  $ \xi (x) : F_x \rightarrow \op{Im} \pi '(x)$ sending $f$ into $1
\otimes f$. We define $\rho (x)$ as the extension of $\xi(x)$   
\begin{gather} \label{eq:def_rho}
  \rho (x) : F_x \oplus  F_x^{\perp} \rightarrow
(\op{Im} \pi'(x) ) \oplus  ( \op{Im} \pi'(x) )^{\perp} 
\end{gather}
having the block decomposition $\begin{pmatrix} \xi (x) & 0 \\ 0 & 0 
\end{pmatrix} $. So $\rho (x)$ is canonically defined.
The equalities  $ \rho(x)^* \rho(x) = \pi (x) $ and $\rho(x) \rho(x) ^* = \pi '(x)$ are 
easily verified  by using that $\xi
(x)$ is unitary. Writing $\rho $ in terms of a local frame of $F$, we see that
$\rho (x)$ depends smoothly on $x$. Finally, $F_x \subset \D^{\ep} (T_xM)
\otimes A_x$ and $\op{Im} \pi'(x) \subset \D^{+} (T_x M) \otimes B_x$, so
$\rho (x) \in \symb^{\ep} (M)_x \otimes \op{Hom} (A_x, B_x)$.   
\end{proof}

\begin{theo} \label{theo:unitary_equivalence}
Assume that $B = F$ and $\rho$ is the symbol defined above. Then there exists $U \in \lag ^{\ep}(A,B)$ with symbol $\si_0(U) = \rho$ and such that
\begin{gather} \label{eq:U}
  U^*_k U_k = \Pi_k, \qquad U_k U_k^* = \Pi'_k 
\end{gather} 
when $k$ is sufficiently large.  Modifying $\Pi_k'$ for a finite 
 number of $k$, we can choose $U$ so that \eqref{eq:U} holds for any $k$.
 In this case, the Toeplitz algebras $ \mathcal{T}$ and $\mathcal{T}'$ are
 isomorphic by the map sending $P$ into $U P U^* $.
 Furthermore, $P \in \mathcal{T}_q$ if and only if $UPU^* \in \mathcal{T}'_q$
 and when this is satisfied
 \begin{gather} \label{eq:symbol_iso}
   \tau'_q (U P U^* ) = \tau_q (P)  .  
 \end{gather}
\end{theo}

\begin{proof}
Choose $W \in \lag^{\ep} (A,B)$ with symbol $\rho$ and set $V := \Pi' W \Pi$. Then $V
\in \lag^{\ep} (A,B)$ with $ \si_0(V) = \pi' \rho \pi = \rho$ and since $\rho^* \rho =
\pi$ and $\rho \rho^* = \pi'$, we have
\begin{gather}  \label{eq:first_Step}
V^*_k V_k = \Pi_k + \bigo (k^{-\frac{1}{2}}), \qquad V_k V^*_k = \Pi_k' + \bigo (
k^{-\frac{1}{2}} ) 
\end{gather}
So $V_k$, viewed as an operator from $\Hilb_k$ to $\Hilb_k'$, is invertible
when $k$ is sufficiently large. 

Observe also that $V^*V$ is a Toeplitz
operator of $\mathcal{T}$ with symbol $\op{id}_{F}$. So $V^*V = \Pi + Q$ with $Q \in
\mathcal{T}_1$. Since $\| Q_k\| = \bigo (k^{-1})$, the spectrum of $Q_k$ is
contained in $[-\frac{1}{2}, \frac{1}{2} ]$ when $k$ is sufficiently large.
Modifying $Q_k$ for a finite number of $k$, we can assume this holds for any
$k$, and when $k$ is sufficiently large, we still have $V_k^*V_k = \Pi_k + Q_k$. 
Let $P_k$ be the endomorphism of $\Ci ( M , L^k \otimes A$) which is zero on
$\Hilb_k^{\perp}$ and equal to $(\op{Id}_{\Hilb_k} + Q_k ) ^{-1/2}$ on
$\Hilb_k$. We claim that $(P_k)$ belongs to $\mathcal{T}$ and has
symbol $\op{id}_{F}$. 

Assuming this temporarily, it follows that $U_k := V_k P_k $ belongs
to $\lag ^{\ep} (A,B)$, has symbol $\rho$ and satisfies when $k$ is sufficiently
large   $U_k^* U_k = P_k V_k^* V_k P_k = \Pi_k$. Since $U_k$, viewed as an
operator from $\Hilb_k$ to $\Hilb'_k$ is invertible, this also implies that
$U_kU_k^* = \Pi'_k$.    

To prove the claim above, we write the Taylor expansion $(1+x)^{-1/2} = 1 +
\sum_{\ell = 0 }^m a_\ell x^\ell + x^{m+1} f_m (x)$ with $f_m$ a continuous
function $[-\frac{1}{2},\frac{1}{2}] \rightarrow \R$. Then 
\begin{gather} \label{eq:racine}
 P_k =  \Pi_k + \sum_{\ell = 0 }^m a_\ell Q_k^{\ell} + Q_k^{m+1} f_m (Q_k)
.
\end{gather}
Then we show that $P$ belongs to $\lag ^+ (A)$ by arguing as in the proof of
Theorem \ref{theo:constr-proj}:  $Q^\ell \in \mathcal{T}_\ell $ and $\| f_m (
Q_k) \| = \bigo (1)$, so the Schwartz kernel family of $Q_k^{m+1} f_m (Q_k) =
Q_k^m f_m (Q_k) Q_m$ is in $\bigo ( k^{2n - m-1 })$. Choosing $m$
sufficiently large at each step,  we then deduce from \eqref{eq:racine} that
the Schwartz kernel family of $P_k $  is $\bigo (
k^{-\infty})$ outside the diagonal and  that the local expansions
\eqref{eq:expansion_kernel} hold.

So we have proved the existence of $U \in \lag^{\ep}(A,B)$ with $\si_0(U) = \rho$
and satisfying \eqref{eq:U} for any $k$ except a finite set. For the missing
$k$'s, we modify $\Pi_k'$ by choosing any subspace
$\Hilb'_k$ of $\Ci ( M , L^k \otimes B)$ having the same dimension as 
$\Hilb_k$, define $\Pi_k'$ as the orthogonal projector onto $\Hilb_k'$ and
$U_k$ as any isometry $\Hilb_k \rightarrow \Hilb'_k$ extended to zero on
$\Hilb_k^{\perp}$. Then $\Pi_k'$ and $U_k$ have a smooth Schwartz kernel, so
the new families $\Pi '$ and $U$ are still in $\lag ^{+} (B)$ and $\lag ^{\ep}
(A,B)$ respectively.

It is now easy to prove the last assertion: if $P \in \lag ^{+} (A)$, then $U
P U^* \in \lag ^{+}(B)$ because $U \in \lag ^{\ep} (A,B)$ and $U^* \in
\lag^{\ep} (B,A)$. If $ \Pi P \Pi = P$, then $\Pi' (U P U^* ) \Pi'
= U P U^*$ by \eqref{eq:U}. So $P \in \mathcal{T}$ implies that $U P U^*  \in
\mathcal{T}'$, which defines an isomorphism from $\mathcal{T}$ into
$\mathcal{T}'$ because we can invert it by sending $Q$ into $U^* Q U$.  Furthermore, $\si_0 ( UPU^* ) = \rho^* \si_0 ( P)
\rho$ which leads to \eqref{eq:symbol_iso}. 
\end{proof}

A first corollary is the computation of the symbols of commutators in terms
of Poisson bracket. Recall the Toeplitz operators $T_k (f) : \Hilb_k
\rightarrow \Hilb_k$ associated to $f \in \Ci ( M)$. Define similarly $T'_k
(f) : \Hilb'_k \rightarrow \Hilb_k'$.

\begin{cor}
$ [ T_k (f) , T_k (g) ] = i k^{-1} T_k ( \{ f,g \} ) + \bigo ( k^{-2})$ for
any $f , g \in \Ci (M)$. 
\end{cor}
Another proof will be provided in Proposition \ref{prop:subsymbolic}.

\begin{proof} This amounts to show that for any two Toeplitz operators $T$, $S$ of
  $\mathcal{T}$ with symbol $\tau_0(T) = f \op{id}_F$, $\tau_0(S) = g
  \op{id}_F$, we have $\tau_1 ( [T,S] ) = i \{ f, g\} \op{id}_F$. By Theorem
  \ref{theo:unitary_equivalence}, this holds for $\mathcal{T}$ if and only if
  this holds for $\mathcal{T}'$. \cB The results for $\mathcal{T}'$ has been
  proved in \cite[Theorem 1.4]{oim}, when the projector is chosen as in
  \cite[Theorem 1.1]{oim}.\clb
\end{proof}

The operators $T'_k(f)$ are defined not only for $f \in \Ci (M)$ but also for
$f \in \Ci ( M, \op{End} B)$. Since $\tau_q'( k^{-q} T_k' (f)) = f$, it follows
that we can define the Toeplitz operators of $\mathcal{T}'$ as the families
$(T_k)$ such that for any $N$,
$$ T_k = \sum_{\ell =0 }^{N} k^{-\ell} T_{k} (f_\ell ) + \bigo ( k^{-(N+1)})$$
for a sequence $(f_\ell)$ of $\Ci ( M, \op{End} B)$. This provides a
definition of $\mathcal{T}'$ without any reference to the algebra $\lag ^{+}
(B)$. Observe also that the coefficients $f_{\ell}$ are uniquely determined
by $T$ and the map $\mathcal{T}' \rightarrow \Ci ( M , \op{End} B) [[\hbar]]$
sending $T$ into $\sum \hbar^\ell f_{\ell}$ is a full symbol map, meaning that
it is onto and its kernel is $\mathcal{T}' \cap \bigo ( k^{-\infty})$. This
full symbol map can also be used to get uniform control of the product of
Toeplitz operators, cf. \cite{oim}. But unfortunately, this does not hold for
$\mathcal{T}$, except in the particular case where $F$ has rank one, so that
$\op{End} F \simeq \C$. This happens in particular for higher Landau level in
dimension $n=1$. 

A second consequence of Theorem \ref{theo:unitary_equivalence} is the
computation of the dimension of our quantum spaces. \cB Here the parity assumption
is not necessary. 

\begin{theo} \label{theo:dim_general}
  Let $\Pi \in \lag (A)$ be a projector whose symbol $\pi = \si_0 ( \Pi)$  has a constant
  rank. Then the dimension of $\Hilb_k= \op{Im}( \Pi_k)$ is
$$  \op{dim} \Hilb_k = \int_M \op{ch} ( L^k \otimes F ) \; \op{Td}
(M)$$
when $k$ is sufficiently large, where $F$ is the subbundle of $\D (TM) \otimes
A$ given by $F_x = \op{Im} \pi (x)$ for any $x \in M$. 
\end{theo} \clb

\begin{proof}
  We introduce a new family $(\Hilb''_k:= \op{Ker} D_k)$ where  $D_k$
  is the spin-c Dirac operator acting on $\Ci ( M , L^k \otimes B \otimes
  S)$ with \clr $S :=\bigwedge (T^*M)^{0,1} $ \clb the spinor bundle. By the
  Atiyah-Singer Theorem and a vanishing theorem \cite{BoUr96}, \cite{MaMa02}, the dimension of $\Hilb''_k$  is given by the Riemann-Roch number of $L^k \otimes B$ when $k$ is sufficiently large. We claim
  that the projector $\Pi''_k$ of $\Ci ( M , L^k \otimes B \otimes
  S)$ onto $\Hilb_k''$ belongs to $\lag ^+ ( B \otimes S)$ and has symbol $\Uu_{00}
  \otimes p_B$ where $p_B $ is the section of $\op{End} ( B \otimes S)$ equal
  at each $x\in M$ to the projector of $B_x \otimes S_x$ onto $B_x \otimes
  \C$.  This is actually a reformulation of results by Ma and Marinescu \cite{MaMa}, as is
  explained in \cA \cite[Appendix A]{oim}. \clb  Alternatively this follows from the
  companion paper \cite{oim_copain}.

  Now the image of the symbol $\rho_{00} \otimes p_B$ is isomorphic with $B$,
  so by Theorem \ref{theo:unitary_equivalence}, when $k$ is sufficiently
  large, $\Hilb''_k$ has the same
  dimension as $\Hilb_k' = \op{Im} \Pi'_k$, where $\Pi_k'$ is any self-adjoint
  projector of $\mathcal{L}^+(B)$ with symbol $ \rho_{00} \otimes \op{id}_B$.

  To conclude, when $\Pi$ is even, by another application of theorem \ref{theo:unitary_equivalence}, for
  $B=F$, 
  $\Hilb'_k$ and $\Hilb_k$ have the same dimension when $k$ is sufficiently
  large. \cB The same proof works for $\Pi$ not being necessarily even. Actually,
  the existence of $V$ satisfying \eqref{eq:first_Step} already implies that the
  dimensions of $\Hilb_k$ and $\Hilb_k'$ are the same when $k$ is large. \clb
\end{proof}

\section{Landau Hamiltonian algebra} \label{sec:landau-levels-cn}

In this section, we come back to the algebra $\symb (\C^n)$ introduced in Section
\ref{sec:symbol-spaces}. We extend the action of the elements of $\symb (\C^n)$
on $\D ( \C^n)$ to the complete polynomial space and we compute the corresponding Schwartz
kernel. This will be used in the sequel to give an intrinsic definition of the
symbol maps $\si_q$, cf. Definition \ref{def:symbol}, and to understand the
composition properties of the class $\lag (A,B)$.

Let $\pol ( \C^n)$ be the space of polynomials map from $\C^n$ to $\C$, so any $f
\in \pol ( \C^n)$ has the form $ f = \sum a_{\al \be} z^ \al \con{z}^\be$
where the sum is finite and the $a_{\al \be}$ are complex numbers. The space
$\D ( \C^n)$ introduced in Section
\ref{sec:symbol-spaces} is the subspace of $\pol ( \C^n)$ of antiholomorphic
maps.
We endow $\pol (
\C^n)$ with the same scalar product
\begin{gather} \label{eq:scal_prod_2}
\langle f, g \rangle = (2\pi)^{-n} \int_{\C^n} e^{-|z|^2} f (z) \con{g (z)} \;
d \mu_n (z) 
\end{gather}
as in \eqref{eq:scal_product_bargm} for $\D (
\C^n)$. The family $(( (\al + \be )! )^{-\frac{1}{2}} z^{\al} \con{z}^\be, \;
\al , \be \in \N^n )$ is an orthonormal basis of $\pol ( \C^n)$.

For any $i
=1, \ldots, n  $, introduce the endomorphism $a_i = \partial_{\con{z}_i}$ and
its adjoint $a_i^{*} = \con{z}_i - \partial_{z_i}$. They satisfy the bosonic
commutation relations
$$ [ a_i ,
a_j ] = [a_i^{*}, a_j^{*} ] = 0 , \qquad [a_i , a_j^{*}] =
\delta_{ij}. $$   
So the $a_i^{*} a_i$'s are mutually commuting Hermitian endomorphisms. Their eigenspaces are the Landau levels of $\C^n$. 
In the sequel, we use the notation $a^\al := a_1^{\al(1)} \ldots a_n^{\al (n)}$
and $(a^{*})^{\al} := (a_1^*)^{\al(1)} \ldots (a_n^* )^{\al (n)}$.

\begin{prop} \label{prop:linearlandau}
  $ $ \begin{enumerate}
     \item For $i=1, \ldots, n$, $a_i^* a_i$ is diagonalisable with spectrum
  $\N$. So we have a decomposition into mutually orthogonal joint eigenspaces $ \pol ( \C^n) = \bigoplus_{\al \in
    \N^n} \mathcal{L}_\al$ with  $\mathcal{L}_{\al} = \bigcap_{i=1}^{n} \op{ker} (  a_i^* a_i - \al
  (i) ).$
\item  $\mathcal{L}_0 = \C [z_1, \ldots ,z_n]$ and for any $\al \in \N^n$, $
  \mathcal{L}_{\al} =  (a^*)^\al \mathcal{L}_0 $.
  \item For any $\al , \be \in \N^n$, let $\Vvv_{\al \be} := ( \al ! \be
    !)^{-\frac{1}{2}} (a^{*})^{\al} \Vvv_{00} a ^{\be}$ with $\Vvv_{00} $ the orthogonal projector of $\pol
    ( \C^n)$ onto $\mathcal{L}_0$. Then
    \begin{enumerate}
        \item $\Vvv_{\al\be}$ is zero on the $\mathcal{L}_{\ga}$'s with $\ga \neq
          \be$ and restricts to a unitary isomorphism from $\mathcal{L}_{\be}$
          to $\mathcal{L}_{\al}$
          \item  $\Vvv_{\al \al}$ is the orthogonal projector onto
    $\mathcal{L}_\al$. 
      \item  $ \Vvv_{\al \be} \circ \Vvv_{\tilde{\al} \tilde{\be}} =
        \delta_{\be \tilde{\al}} \Vvv_{\al \tilde \be} $  and $\Vvv_{\al \be}
        ^{*} = \Vvv_{\be \al} $ 
  \end{enumerate}
\end{enumerate}
\end{prop}

\begin{proof} The result is certainly standard in condensed matter theory. For the convenience of the reader, we explain briefly the proof
  for $n=1$. The extension in higher dimension is straightforward. We write
  $a:= a_1$, recall the commutation relation $ [a , a^{*} ]= 1 $ and set  $\mathcal{L}_m := (a^{*})^m (\C[z])$ for any $m \in \N$.  

  We check by induction that $\mathcal{L}_m = \ker ( a^* a - m)$. 
  First,  writing $\langle a^{*} a f, f \rangle = \| a
  f \|^2 $, it comes that $\ker a^* a = \ker a =
  \mathcal{L}_0$.
  Assume now that $\mathcal{L}_m = \ker ( a^* a - m)$. By the commutation relation, $f \in \mathcal{L}_m$ implies that $a^{*} a a ^{*} f =
  (m+1) a^*  f$, so $\lan_{m+1} \subset \ker ( a^* a -
  (m+1))$. Conversely, by the commutation relation again,  $a^* a f = ( m+1) f $ implies that $ (a^* a
  ) a f = m a f$ so $ af \in \lan_m$ and $f = (m+1)^{-1} a^* (af) \in
  \lan_{m+1}$. 

  To conclude that $a^* a $ is diagonalizable with eigenvalues in $\N$,
  it suffices to prove that $\pol ( \C)$ is spanned by the $\lan_m$.
  Introduce the filtration $\mathcal{F}_m := \oplus _{\ell=0}^{m} \con{z}^\ell
  \C[z]$, $m \in \N$. If $f \in \C [z]$, then $\con{z}^m f  =
  (a^{*})^m f \mod \mathcal{F}_{m-1}$. So $\mathcal{F}_m = \lan_m +
  \mathcal{F}_{m-1} = \ldots = \lan_m + \lan_{m-1} + \ldots + \lan_0$ by
  reiterating.

  So we have proved that $\pol ( \C) = \bigoplus \lan_m$ with $\lan_m = \ker (
  a^* a - m)$, which shows the first and second assertions of the proposition.  By the
  commutation relation, $a
  a^{*} = (m+1)$ on $\lan_m$. So  $a^{*} : \lan_m \rightarrow \lan_{m+1}$ is invertible with inverse
  $(m+1)^{-1} a: \lan_{m+1} \rightarrow \lan_{m}$. We conclude
  that \begin{itemize}
    \item if $p \leqslant
  m$, then $a^p$ restricts to an isomorphism from $\lan_m$  to $\lan_{m-p}$,
  whose inverse is the restriction of $\frac{(m-p)!}{m!} (a^{*})^m$ to $\lan_{m-p}$. 
 \item if $p > m$, then $a ^p ( \lan_m) = \{0\}$. 
 \end{itemize}
 With these two facts, we easily check the third assertion.
\end{proof}

By the last assertion of Proposition \ref{prop:linearlandau}, the space
$\widetilde{\symb} (\C^n)$ of endomorphisms of $\pol (
\C^n)$ generated by the $\Vvv_{\al \be}$'s is closed under composition, so it is an
algebra. By the following proposition, $\widetilde{\symb} (\C^n)$ is
isomorphic with the algebra $\symb (\C^n)$ introduced in Section
\ref{sec:symbol-spaces}, through the map sending $\Vvv_{\al\be}$ into
$\Uu_{\al\be}$. 

\begin{prop} The elements of $\widetilde{\symb} (\C^n)$ preserve the
  subspace $\D ( \C^n)$ of $\pol ( \C^{n})$. Furthermore, the restriction map
  $\op{res} : \widetilde{\symb} (\C^n) \rightarrow \op{End} ( \D ( \C^{n}) )$ is
  injective, with image $\symb ( \C^n)$ and $\op{res} ( \Vvv_{\al \be} ) = \Uu_{\al
    \be}$.
\end{prop}

Recall the decomposition \eqref{eq:symb+_-} of $\symb (\C^n)$ into the subspaces of even
and odd elements.
Since $\widetilde \symb (\C^n) \simeq \symb ( \C^n)$, this
gives us a new decomposition
$$\widetilde{\symb}
(\C^n) = \widetilde{\symb}^+
(\C^n) \oplus \widetilde{\symb}^-
(\C^n).$$

\begin{proof} Observe first that the operators $a_i$, $a_i^{*}$ and the projector $\Vvv_{00}$
preserves $\D ( \C^n)$. Furthermore,  for any $f \in \D (\C^n)$, $a_i f =
\partial_{\con{z}_i} f$, $a_i^{*} f = \con{z}_i f$ and $\Vvv_{00} f = f (0)$.
Consequently, the operators $\Vvv_{\al \be}$ preserve $\D ( \C^n)$ and an easy
computation shows that
$$ \Vvv_{\al\be} \bigl( (\be !)^{-\frac{1}{2}}  \con{z}^\be \bigr) = (\al !)^{-\frac{1}{2}}
\con{z}^{\al}, \qquad \Vvv_{\al \be} ( \con{z}^{\ga} ) = 0 , \quad \forall \ga \in \N^n
\setminus \{ \be \}. $$
This means that the restriction of $\Vvv_{\al \be}$ to $\D (\C^n)$ is exactly the
endomorphism $\Uu_{\al \be}$ introduced
in Section \ref{sec:symbol-spaces}, cf. Equation \eqref{eq:def_Ualphabeta}. So the
restriction map $\op{res}$ is well-defined, its image is $\symb ( \C^n)$, and
the $\Uu_{\al\be}$'s being linearly independent, it is injective. 
\end{proof}

Let us compute the
Schwartz kernel of each $\Vvv_{\al\be}$. 
\begin{lemme} \label{lem_kenrel_u_alpha_beta}
  For any $f \in \pol ( \C^n)$, we have
  $$ (\Vvv_{\al\be} f)(u ) = (2 \pi)^{-n} \int_{\C^{n} } e^{ u\cdot \con{v} - |v|^2 }
  \Ppp_{\al\be} (u-v) \, f(v )  \;  d\mu_n (v)  $$
  where $u \cdot \con v = \sum u_i \con{v}_i$, $\Ppp_{\al,\be} (z) =  \bigl( \al! \be ! \bigr) ^{-\frac{1}{2}} \bigl(
\partial_z - \con{z})^{\al} z^\be $.
\end{lemme}

In particular, the orthogonal projector $\sum_{|\al| = m } \tilde{\rho} _{\al
  \al}$ onto $\bigoplus_{|\al| = m} \mathcal{L}_{\al}$ has the Schwartz kernel
\begin{gather} \label{eq:noyau_niveau_m}
( 2 \pi )^{-n} e^{u \cdot \con{v} - | v |^2 } Q_m ^{(n-1)} ( | u - v |^2 ) \;
d \mu_n (v) 
\end{gather}
where $Q_m^{(n-1)}$ is the Laguerre polynomial and we have used
\eqref{eq:a_0laguerre}. 
\begin{proof}
For $\al = \be =0$, this is the well-known formula for the Schwartz kernel $K
( u,v) = (2\pi)^{-n} e^{u \cdot \con v - |v|^2} $ of the projector onto the
Bargmann space, which in our setting is the $L^2$-completion of
$\mathcal{L}_0$. So the Schwartz kernel of $(a^{*})^{\al} \Vvv_{00} a^{\be}$
is $ K_{\al,\be} (u,v) = (\con{u} -\partial_u)^{\al} (-\partial_{\con{v}})^{\be} K(u,v)$. To
compute this, we use that for a polynomial $g (z,\con{z})$
\begin{xalignat*}{2} 
(\con{u}_i - \partial_{u_i}) (K(u,v) g(u-v ) )  & = K(u,v)
(a_i^{*} g)(  u-v )  , \\
 ( \partial_{\con{v}_i}) (K(u,v) g(u-v ) )  & = K(u,v)
(b_i^{*} g)(  u-v )  
\end{xalignat*}
where $b_i^{*} := z_i - \partial_{\con{z}_i}$. So $K_{\al,\be} (u,v) = K(u,v) p(u-v) $ with
$p (z,\con z) = ( a^{*})^{\al} (b^{*})^\beta 1= \bigl(
\partial_z - \con{z})^{\al} z^\be $, which ends the proof. 
\end{proof}
So on one hand, the elements of $\widetilde{\symb} ( \C^n)$ act on $\pol ( \C^n)$, on the
other hand, the Schwartz kernel of $\Vvv_{\al\be}$ is given by a polynomial
$\Ppp_{\al\be} \in \pol ( \C^n)$. 
\begin{prop} \label{prop:tilde_symb_and_op}
$\widetilde{\symb} ( \C^n)$  consists
of the endomorphisms $V$ having the form
\begin{gather} \label{eq:sch_ker}
(V f)(u) = (2\pi)^{-n} \int_{\C^n} e^{u\cdot \con v - |v|^2} q (u-v ) f(v) \;
d\mu_n (v)  
\end{gather}
with $q \in \pol ( \C^n)$.
Furthermore, the map $\op{Op} : \pol (\C^n) \rightarrow
\widetilde{\symb} ( \C^n)$, sending $q$ into $V$, is an isomorphism which  preserves
the parity \cB and
\begin{alignat}{2} \label{eq:rel_bizare}
  \begin{split} 
    &    \op{tr} \bigl( \op{Op} (q)|_{\D ( \C^n)} \bigr) = q (0) \\
    & \op{Op} (q) \circ \op{Op} (f) = \op{Op} ( \op{Op} (q) f ) \\
    & \langle q , f \rangle = ( \op{Op}(f)^* q )(0)
\end{split}
\end{alignat}
for any $q, f \in \pol ( \C^n)$. \clb
\end{prop}
\begin{proof}
Lemma \ref{lem_kenrel_u_alpha_beta} says that  $\op{Op} ( \Ppp_{\al\be} ) =
\Vvv_{\al \be}$. The family $(\Ppp_{\al\be})$ is a basis of $\pol ( \C^n)$
because $\Ppp_{\al\be} (z )  = (\al ! \be ! ) ^{-\frac{1}{2}}(- z)^{\al} \con{z}^{\beta} + $
a linear combination of $z^{\al'} \con{z}^{\be'}$ with $\al' < \al$ and $\be'
< \be$. Since $(\Uu_{\al\be})$ is a
basis of $\symb ( \C^n)$, $(\Vvv_{\al \be})$ is a basis of $\widetilde { \symb }
( \C^n)$ and it follows that $\op{Op}$
is an isomorphism. This isomorphism preserves the parity, because $\Vvv_{\al
  \be} $ and $\Ppp_{\al\be}$ have both the same parity as $|\al| + |\be|$. 
For the first equation of \eqref{eq:rel_bizare}, it
suffices to prove it for $q = \Ppp_{\al,\be}$, and in this case, it follows from $\op{tr}
\Uu_{\al\be} =\delta_{\al\be} = \Ppp_{\al\be} (0)$. 

To prove the second equation of \eqref{eq:rel_bizare}, observe that we recover $q$ from the Schwartz kernel of
$\op{Op} (q)$  \cB by multiplying by $(2\pi)^{n}$ and setting $v =  0 $, that is $
f(u) = (2 \pi)^n \op{Op} ( f) (u,0)$. Let $g \in \pol ( \C^n)$ be the function
such that $\op{Op} (g) = \op{Op} (q ) \circ \op{Op} (f)$. Then using the
previous observation for $f$ and for $g$, we have
\begin{xalignat*}{2}
  (\op{Op} (q) f) (u) = &   ( 2\pi)^n \int_{\C^n} \op{Op}(q) (u,v) \op{Op} (f) (v,0) \;
  d\mu_n (v) \\
  = &  (2 \pi)^n \op{Op}(g) ( u,0) \\
  = & g (u) .
\end{xalignat*}
The proof of the third equation of \eqref{eq:rel_bizare} is similar  by using
 that $ \op{Op}(f)^* ( 0,u ) = \con{\op{Op} (f) (u,0)}= (2 \pi)^{-n} \con{f (u)}$. 
\end{proof} \clb

As a last remark, we can replace in the previous definitions $\C^n$ with any $n$-dimensional Hermitian
space $\EE$ as we did in section \ref{sec:symbol-spaces}. So we denote by
$\pol ( \EE)$ the space of polynomial maps $\EE \rightarrow \C$ and by
$\widetilde{\symb} ( \EE)$ the space of endomorphisms of $\pol ( \EE)$ having
the form \eqref{eq:sch_ker}, where we interpret $u \cdot \con{v}$ as the
scalar product of the vectors $u$, $v$ of $\EE$  and $|v|$ as the norm of $v$.  
Observe as well that the map $\op{Op} : \pol ( \EE ) \rightarrow
\widetilde{\symb} ( \EE )$ is well-defined. Furthermore the restriction from $\pol ( \EE)$
to its subspace $\D ( \EE)$ induces an isomorphism $\widetilde{\symb} ( \EE ) \simeq \symb (\EE)$.

\section{The Schwartz kernels of operators of \texorpdfstring{$\lag (A, B)$}{L(A,B)} } \label{sec:schw-kern-oper}

\subsection{The section \texorpdfstring{$E$}{E}} \label{sec:section-ee}

An important ingredient in the global Schwartz kernel description of operators
of $\lag (A,B)$ is a section $E$
of $L \boxtimes \con{L}$ satisfying the following conditions. For any $y\in
M$, denote by $E_y $ the section of $L \otimes \con{L}_y$ given by $E_y(x) =
E(x,y)$ for any $x \in M$. Then we will assume that for any $y\in M$
\begin{xalignat}{2} \label{eq:hypotheseE}
  \begin{split} 
 & E_y(y) = u \otimes \con{u},\qquad \forall u \in L_y \text{ with } |u| =1, \\
 & (\nabla E_y)(y) =0 , \\
 & (\nabla_{\xi} \nabla_{\eta} E_y)(y) = - \bigl( \tfrac{i}{2} \om (\xi, \eta) + \tfrac{1}{2}
  \om ( \xi, j \eta) \bigr) E_y (y) , \qquad \forall \; \xi, \eta \in T_yM
\end{split}
\end{xalignat}
Such a section appeared already in the expansion \eqref{eq:expansion_kernel} as
follows. Choose a unitary frame $t$ of $L$ and a coordinate system on the same
open set, then the section 
\begin{gather} \label{eq:exempleE}
E (
y+ \xi, y) := e^{-\varphi ( y, \xi)} t( y+ \xi) \otimes \con{t}(y),
\end{gather}
with
$\varphi$ defined as in \eqref{eq:expansion_kernel}, satisfies
\eqref{eq:hypotheseE}. From this local construction, we easily obtain a global
section $E$ by using a partition of unity.

The conditions
\eqref{eq:hypotheseE} determine the second-order Taylor expansion of $E$ at $(y,y)$ in
the directions tangent to the first factor of $M^2$. Since any tangent vector
of $M^2$ at $(y,y)$ is the sum of a vector tangent to the diagonal and a
vector tangent to the first factor, we deduce that $E$ is uniquely
determined modulo a section vanishing to third order along the diagonal. 

The function $\psi_y (x)= - 2 \ln
|E_y (x)|$ vanishes to second order at $y$ and for any $\xi, \eta \in T_yM$, $(\xi.\eta. \psi_y )(y)
=  \om (\xi, j \eta)$, so $\psi_y(x)>0$ when $x\neq y$ is
sufficiently close to
$y$. So modifying $E$ outside the diagonal, we can assume that it satisfies as
well
\begin{gather} \label{eq:normE}
|E (x,y) | <1, \qquad \forall (x,y) \in M^2 \text{ such that } x \neq y 
\end{gather}
Another important property of $E$ is the symmetry:
\begin{gather} \label{eq:symE}
  \con{E (x,y)} = E(y,x) +
  \bigo ( |x-y|^3)
\end{gather}
For a longer discussion, the reader is referred to \cite{oim}.

In the sequel we will need the following expression of $E$ in terms of complex
coordinates and a frame of $L$, both normal at a point $p_0 \in M$. \cB We say
that a function or a section on $M$ (resp. $M^2$) is in $\bigo_{p_0} (m)$
(resp. $\bigo_{p_0, p_0} (m)$ if it vanishes to order
$m$ at $p_0$ (resp. $(p_0,p_0)$). \clb
Let
$(\partial_i )_{i =1}^{n}$ be an orthonormal basis of $T^{1,0}_{p_0} M$, i.e. 
$\frac{1}{i}\om_{p_0} ( \partial_i, \con{\partial}_j) = \delta_{ij}$. Choose
complex valued functions $z_i$ on a neighborhood of $p_0$ such that
\begin{gather} \label{eq:coordonnees_z}
z_i (p_0)
=0, \qquad dz_i ( \partial_j) =\delta_{ij}, \quad dz_i ( \con \partial_j) = 0 \quad
\text{ at }p_0. 
\end{gather}
Then $(\op{Re} z_i , \op{Im}z_i )_{i=1}^n$ is a coordinate system on a
neighborhood of $p_0$ and $\om|_{p_0} = i \sum dz_i \wedge d\con{z}_i$. The
curvature of $L$ being $\frac{1}{i} \om$, there exists a unitary frame $t$ of
$L$ at $p_0$ such that
\begin{gather}  \label{eq:normal_frame}
  \nabla t  = \tfrac{1}{2}  \textstyle{\sum} ( z_i d \con{z}_i - \con{z}_i
  d z_i ) \otimes t + \bigo_{p_0} ( 2). 
\end{gather}
\cB To construct such a $t$, multiply any unitary local section by $\exp ( i
(f_1 + f_2) )$  where $f_1$ and $f_2$ are real valued functions 
respectively linear and quadratic in $z_i$, $\con{z}_i$.   With $f_1$
conveniently chosen, $\nabla t = \bigo_{p_0} (1)$ and with the right $f_2$, we
get \eqref{eq:normal_frame}. \clb

Then
\begin{gather}  \label{eq:E_linearise}
E(x,y) = e^{ z(x) \cdot \con{z} (y) -\frac{1}{2} ( |z(x) |^2 + |z (y)|^2) }
t(x) \otimes \con{t(y)} + \bigo_{(p_0, p_0)} (3).
\end{gather}
A similar expression already appeared in the description of the Schwartz kernels of the
operators of $\widetilde{\symb} (\C^n)$. Indeed, let $L_{\C^n}= \C^n \times \C$ be the trivial
holomorphic line bundle equipped with the metric such that the frame $s( z) =
(z,1)$ has a pointwise norm $|s(z)|^2 = e^{-|z|^2}$. Then it is natural to
interpret the elements of $\pol ( \C^n)$ as sections of $L_{\C^n}$, because the scalar
product \eqref{eq:scal_prod_2} is the
integral of the pointwise scalar product $\bigl( f (z) s(z), g(z) s(z) \bigr) =
f(z) \con{g} (z) e^{-|z|^2}$. Furthermore, in the integral \eqref{eq:sch_ker},
we can interpret $e^{-|v|^2} f(v)$ as the pointwise scalar product $( f(v)
s(v) , s(v))$. In other words, the Schwartz kernel of $V$  is
\begin{gather} \label{eq:2pin-e_cn-u} 
(2\pi)^{-n} E_{\C^n} (u,v) q(u-v ) \quad \text{ with }
\quad E_{\C^n}(u,v) =e^{u \cdot \con{v}} s(u) \otimes \con{s}(v).
\end{gather}
Now equip $L_{\C^n}$ with its Chern connection, that the unique connection
compatible with both the holomorphic and Hermitian structures. Then $\nabla s = - \sum \con{z}_i
dz_i \otimes s$. So the curvature is $\frac{1}{i} \om_{\C^n} $ with $\om_{\C^n} = i \sum dz_i
\wedge d \con{z}_i$. And if $t $ is the unitary frame $t(z) = e^{|z|^2/2} s
(z)$, we have  $\nabla t = \frac{1}{2} \sum ( z_i d \con{z}_i - \con{z}_i
d z_i ) \otimes t$ and
$$E_{\C^n}(u,v) = e^{ u \cdot \con v - \frac{1}{2} (
    |u|^2+ |v|^2)} t(u) \otimes \con{t(v)},$$ the same formula as \eqref{eq:E_linearise}.  

  \subsection{Schwartz kernel expansion} \label{sec:global-expansion}

We consider operator families $(P_k : \Ci( M , L^k \otimes A ) \rightarrow \Ci(M
, L^k \otimes B), \; k \in \N)$ having smooth Schwartz kernels. Recall the
notations introduced in the beginning of Section \ref{sec:operators}. In
particular, $\| P_k \|$ is the
operator norm whereas $|P_k|$ is the function of $M^2$ sending $(x,y)$ into
$|P_k(x,y)|$. 

Let $E$ be a section of $L \boxtimes \con L $ satisfying \eqref{eq:hypotheseE}
and \eqref{eq:normE} and $b \in \Ci ( M^2 , B \boxtimes \con
A)$. Then, viewing $(L^k \otimes B) \boxtimes
(\con{L}^k \otimes \con{A})$ as $(L \boxtimes \con L )^k \otimes ( B \boxtimes
\con{A})$, we introduce the operator family $(P_k)$ with Schwartz kernels 
\begin{gather} \label{eq:prel}
  P_k (x,y) = \Bigl( \frac{k}{2\pi} \Bigr)^n E^k(x,y) b (x,y)
\end{gather}
The pointwise norms of $P_k$ depend in an essential way on the vanishing order of $b$
along the diagonal. If $m \in \N$, we write $b = \bigo ( m)$ to say that all the
derivatives of $b$ of order $\leqslant m-1$ are zero at each point of the diagonal.
Recall that
$$\psi (x,y) = - 2 \ln |E(x,y)| $$ is positive outside the
diagonal, and vanishes to second order along the diagonal with a Hessian non-degenerate in the transverse direction. 

\begin{lemme} \label{lem:estimate}
  If $P_k$ is given by \eqref{eq:prel} with $b = \bigo ( m)$, then \cB $|P_k | = \bigo \bigl( k^{n-\frac{m}{2}} e^{-k
    \frac{\psi}{4}} \bigr)$  \clb and $\| P_k \| = \bigo ( k^{- \frac{m}{2}})$
\end{lemme}
\begin{proof} Since $b = \bigo (m)$, $ | b | =  \bigo ( \psi ^{\frac{m}{2}}
  )$. So \cB
  $$|P_k | \leqslant C k^n e^{-k \frac{\psi}{2}}  \psi ^{\frac{m}{2}}
  = C k^{n- \frac{m}{2}} e^{-k \frac{\psi}{4}}  e^{-k \frac{\psi}{4}}
  (k\psi) ^{\frac{m}{2}}
  \leqslant C' k^{n- \frac{m}{2}} e^{-k \frac{\psi}{4}} $$  \clb
  because  $t \rightarrow e^{-\frac{t^2}{2} } t^m$ is bounded on
  $\R_{\geqslant 0}$. This proves the first estimate and can be written
  locally in a coordinate system as
  $$|P_k (x,y) | \leqslant C k^{n - \frac{m}{2}} e^{-k|x-y|^2 /C}.$$
So $\int_M |P_k(x,y)| \, d\mu_M (y)$ and $\int_M |P_k (x,y)|\, d\mu_M (x)$ are both
$\leqslant C k^{-\frac{m}{2}}$ and the operator norm estimate follows from Schur
  test. 
\end{proof}
\cA Consider now $(P_k) \in \lag (A,B)$. Recall that by definition
\eqref{eq:exp_Lag_bis} we have for any $N \in \N$, \clb 
\begin{gather} \label{eq:expansion}
P_k (x,y) = \Bigl(\frac{k}{2\pi} \Bigr)^n E^k(x,y) \sum_{\substack{\ell \in
      \Z, \\ \ell + m ( \ell ) \leqslant N }} k^{- \frac{\ell}{2}} b_{\ell} (x,y)
  + R_{N,k} (x,y) 
\end{gather}
where
  \begin{enumerate}[i)] 
\item   $m : \Z \rightarrow \N \cup \{\infty \} $ is such that for any
  $N$, $\{ \ell/\; \ell + m(\ell) \leqslant N \}$ is finite, and $\ell + m (
  \ell) \geqslant 0$ for any $\ell$. 
  \item $(b_{\ell})_{\ell \in \Z}$ is a
  family of $\Ci ( M^2, B \boxtimes \con{A} )$ such that $b _{\ell} = \bigo ( m(
  \ell))$ for any $\ell$.
  \item $|R_{N,k} (x,y) | =  \bigo ( k^{n - \frac{N+1}{2}} )$
  uniformly on $M^2$.
\end{enumerate}
By Lemma \ref{lem:estimate}, $| E^k(x,y)
  k^{-\frac{\ell}{ 2}} b_{\ell}(x,y) | \in \bigo ( k^{ - \frac{1}{2} (\ell
    + m ( \ell))})$. So the expansion \eqref{eq:expansion} is consistent in the sense that passing
  from $N$ to $N+1$, we add new terms $k^{-\frac{\ell}{2}} b_\ell$ such that $\ell + m
  ( \ell)  = N+1$, which contribute to $P_k (x,y)$ with a $\bigo (k^{n
    -\frac{1}{2}(N+1)})$.

\begin{lemme} \label{lem:estim_expansion}
If the expansion \eqref{eq:expansion} holds, then for any $q>0$
$$  \cB  |P_k | = \bigo \bigl( k^{n}  e^{- k
  \frac{\psi }{4}} \bigr)  + \bigo ( k^{-q} ),  \clb  \qquad \| P_k \| = \bigo (1)$$
Similarly the remainders $R_{N,k}$'s satisfy for any $q>0$, 
$$  \cB |R_{N,k} | = \bigo \bigl( k^{n- \frac{N+1}{2}}  e^{- k
  \frac{\psi }{4}} \bigr) + \bigo( k^{-q}),  \clb  \qquad \| R_{N,k} \| = \bigo
(k^{-\frac{N+1}{2}}) . $$
\end{lemme}
\begin{proof} 
To prove the first estimate, we use \eqref{eq:expansion} with $N$ sufficiently
large so that $| R_{N,k} | = \bigo ( k^{-q})$, and the result follows from
Lemma   \ref{lem:estimate} because $\ell + m (\ell) \geqslant 0$.
The operator norm estimate is proved similarly by choosing $N$ so that
$|R_{N,k} | =\bigo (1)$ which implies that $\| R_{N,k} \| = \bigo (1)$.
The proof for the $R_{N,k}$ is essentially the same. \end{proof}

We next show that in the expansion \eqref{eq:expansion}, we can choose any
section $E$ satisfying the assumptions given in Section \ref{sec:section-ee}. 

\begin{lemme} \label{lem:changement_E}
  Assume \eqref{eq:expansion} holds and let $E'$ be
  a section satisfying \eqref{eq:hypotheseE} and \eqref{eq:normE}. Then there exists a family
  $(b'_{\ell})$ of $\Ci ( M^2 , B \boxtimes \con A )$ such that
  \eqref{eq:expansion} holds with $E'$ and $b'_{\ell}$ instead of $E$ and
  $b_{\ell}$. 
\end{lemme}

\begin{proof}
Observe first that \eqref{eq:expansion} holds
outside the diagonal if and only $|P_k (x,y)|$ is in $\bigo (
k^{-\infty})$ outside the diagonal, and this condition is clearly independent
of the choice of $E$ and the $b_{\ell}$'s. On a neighborhood of the diagonal, we have $E = e^{g}
E'$. Since $g \in \bigo ( 3)$, we can assume that \cB $|g| \leqslant \psi/8$. \clb Let us write
$$ P_k = \Bigl(\frac{k}{2 \pi}\Bigr)^n E^k
b_{N,k} + \bigo ( k^{n- \frac{N+1}{2}}) \quad \text{with } \quad b_{N,k} = \sum_{\ell + m ( \ell)
  \leqslant N} k^{-\frac{\ell}{2}}b_{\ell} .$$
By Lemma \ref{lem:estimate}, \cB $|E|^k b_{N,k} = \bigo (
e^{-k \frac{\psi}{4}})$. \clb 
Using that
$$ \exp z = \sum_{p = 0 }^N \frac{z^p }{p!} + r_N (z) \quad \text{ with }  \quad |r_{N} (z) | \leqslant \frac{|z|^{N+1} }{ (N+1)!} e^{|\op{Re} z |} ,$$
we deduce from $E^k = e^{kg } (E')^k$ that
\begin{gather} \label{eq:fin}
E^k b_{N,k} = (E')^k b_{N,k}    \sum_{p = 0 }^N   k^p \frac{g^p}{p!}  +
R_{N,k} 
\end{gather}
where \cB
$$| R_{N,k}  | \leqslant C_N  e^{-k  \frac{\psi}{4}   } |kg|^{N+1}  e^{ k
  |\op{Re} g |} \leqslant C_N  e^{-k  \frac{\psi}{8}   }
|kg|^{N+1} = \bigo ( k^{- \frac{N+1}{2}})$$
because $|\op{Re} g | \leqslant \psi/8$. \clb
To conclude now, it suffices to define the $b'_{\ell}$ so
that 
\begin{gather} \label{eq:last_goal}
(E')^k   \Bigl[ \sum_{\ell + m ( \ell)
  \leqslant N} k^{-\frac{\ell}{2}}b_{\ell} \Bigr] \Bigl[  \sum_{p = 0 }^N   k^p \frac{g^p}{p!}  \bigr]= (E')^k \sum_{\ell
  + m ' ( \ell) \leqslant N } k^{- \frac{\ell}{2}} b'_{\ell} + \bigo (
k^{-\frac{N+1}{2}}) 
\end{gather}
holds for any $N$. This suggests that each $b'_{\ell}$ should be equal to
the infinite sum
$$b_{\ell}
+  b_{\ell +2 } \, g + b_{\ell +4 } \frac{g^2}{2} + b_{\ell+ 6} \frac{g^3}{6}
+\ldots$$
But by Lemma \ref{lem:estimate}, the equality \eqref{eq:last_goal} depends only on the
class of $b'_{\ell}$ modulo $\bigo ( N- \ell)$, so we can interpret these
infinite sums as sums of Taylor expansions along the diagonal.
Since  $b_{\ell +2p} \, g^{p} = \bigo ( m(\ell +2 p ) + 3p ) = \bigo ( 3p)$, by
Borel lemma, there exists $b'_{\ell}$ such that for any $M$
\begin{gather} \label{eq:changement_E}
b'_{\ell}  = \sum_{p=0}^M   b_{\ell+ 2p } \frac{g^p}{p!} + \bigo (3(M+1))
\end{gather}
So $b'_{\ell} = \bigo ( m '(\ell) )$ with
$ m' (\ell ) := \op{min} \{ m( \ell + 2p ) + 3p , \; p \in \N \} .$ We easily
check that $m'$ satisfies the same condition as $m$.
We finally deduce \eqref{eq:last_goal}  by
removing with Lemma  \ref{lem:estimate} all the coefficients  leading to a $\bigo ( k^{-\frac{N+1}{2}})$.
\end{proof}

Suppose now we have an open set $U$ of $M$, and functions $u_i \in \Ci (U^2)$,
$i =1, \ldots , 2n$ vanishing along the diagonal and such that for any $y \in U$,  $(u_i(
\cdot , y))$ is a coordinate system on a neighborhood of $y$. Then we can write the Taylor expansions
along the diagonal as follows: any $ f \in \Ci (U^2)$ has a decomposition
\begin{gather} \label{eq:taylor_expansion}
f(x,y ) = \sum_{m=0}^M f_m (y, u(x,y)) + \bigo ( M+1 ) 
\end{gather}
where each $f_m (y,\xi)$ is homogeneous polynomial in $\xi$ with degree $m$. 
This can be done also for sections of $B \boxtimes \con{A}$ by introducing frames of
$A$ and $B$ on $U$, so that $\Ci (U^2 , B \boxtimes \con{A}) \simeq \Ci (U^2,  \C^r)$.
\begin{lemme} \label{lem:new_expansion}
  The expansion \eqref{eq:expansion} holds on $U^2$ if and only if there exists
  a sequence $(a_p )$ of $\Ci ( U \times \R^{2n}, \C^r)$, each $a_p
  (x,\xi)$ being polynomial in $\xi$, such that for any $N$
\begin{gather} \label{eq:new_expansion}
  P_k (x,y) = \kpi  E^k (x,y) \sum_{p=0}^{N} k^{-\frac{p}{2}} a_p ( y, k^{\frac{1}{2}} u
    (x,y)) + \bigo ( k^{n- \frac{N+1}{2}} ).  
  \end{gather}
\end{lemme}

The remainders in \eqref{eq:new_expansion} satisfy the same pointwise estimates as the
$R_{N,k}$ given in Lemma \ref{lem:estim_expansion}, the proof is identical. 

\begin{proof} If \eqref{eq:expansion} holds on $U^2$, writing the Taylor
    expansion of each $b_{\ell}$ as in \eqref{eq:taylor_expansion}, we have by
    Lemma \ref{lem:estimate}
 \begin{gather*}
    E^k (x,y) k^{-\frac{\ell}{2}} b_{\ell} (x,y)  = E^k (x,y) k^{-
      \frac{\ell}{2}} \sum_{m = m ( \ell) }^{N- \ell } b_{\ell,m} (y, u(x,y))  +
    \bigo( k^{-\frac{N+1}{2}})\\
      =E^k (x,y) k^{-\frac{\ell +m }{2}} \sum_{m = m ( \ell) }^{N- \ell } b_{\ell,m} (y,k^{\frac{1}{2}} u(x,y))  +
    \bigo( k^{-\frac{N+1}{2}})
  \end{gather*}
  So we obtain \eqref{eq:new_expansion} with $ a_p = \sum _{\ell + m ( \ell)
    \leqslant p } b_{\ell , p - \ell} $, this sum being finite
  because of the assumption satisfied by $m(\ell) $.

  Conversely, starting from the $a_{p}$'s,  for each $\ell \in \Z$, we
  construct by Borel summation a function $b_{\ell}$ such that $ b_{\ell} ( x,
  y)=
  \sum _{m =0 }^{M} a _{ \ell+m , m}(y,u (x,y))   + \bigo (M+1)$ for all $M$,
  where by convention $a_{p}=0$ for $p<0$, and $a_{m+\ell, m}$ is the degree $m$ homogeneous component of
  $a_{m+\ell}$. We readily deduce the expansion \eqref{eq:expansion} from
  \eqref{eq:new_expansion} by
  using Lemma \ref{lem:estimate} again.

  Observe that $b_{\ell} =
  \bigo(m(\ell))$ with $m(\ell)$ the smallest $m$ such that $a_{\ell + m ,
   m } \neq 0$.  Since $a_{p} = 0$ for $p<0$, we have $\ell + m(\ell) 
  \geqslant 0$. Furthermore, $\ell + m ( \ell) \leqslant N$ happens only if
  there exists $m \leqslant N - \ell$ such that $a_{\ell + m, m } \neq 0$,
  that is if there exists $p \leqslant N$ such that $ a _{p, p-\ell } \neq 0$,
  so necessarily $ p - \ell \leqslant d (p)$ where $d(p)$ is the degree of
  $a_p$. So $\ell + m (\ell) \leqslant N$ implies that $\ell \geqslant \min \{
  p - d(p) / \; p =0, \ldots, N \}$.  So $\ell + m (\ell) \leqslant N$ only
  for a finite number of $\ell$.
  \end{proof}
\cA  We have essentially proved Proposition  \ref{prop:un_de_plus}. Here are the details.
  
\begin{proof}[Proof of Proposition \ref{prop:un_de_plus}]
 Identify $U$ with an open convex set of
$\R^{2n}$, then the functions $u_i(x,y)= x_i -y_i$ satisfy the above
conditions.  And for $(x,y) = (x' + \xi' , x')$, we have $a_p ( y, k^{\frac{1}{2}}
u(x,y)) = a_p ( x', k^{\frac{1}{2}} \xi')$, so the expansions
\eqref{eq:new_expansion} and \eqref{eq:expansion_kernel} are the same when $E=
e^{-\varphi}$.
Now Proposition  \ref{prop:un_de_plus} follows from Lemma \ref{lem:new_expansion}, the local version of Lemma \ref{lem:changement_E} and the fact that $E=e^{-\varphi}$ in
\eqref{eq:expansion_kernel} satisfies the conditions \eqref{eq:hypotheseE}.
\end{proof}

It is the good place to prove Lemma \ref{lem:parite} on the characterization
of the parity in terms of local expansions. 
\begin{proof}[Proof of Lemma \ref{lem:parite}]
This follows  from the relation between the coefficients $a_{p}$ and the
coefficients $b_{\ell}$ given in the proof of Lemma \ref{lem:new_expansion}.
For instance, if $b_{\ell} =0$ for any odd integer $\ell$, then $b_{\ell,
  p-\ell} \neq 0$ only for even $\ell$ and in this case it has the same parity
as $p$, so $   a_p = \sum _{\ell + m ( \ell)
    \leqslant p } b_{\ell , p - \ell} $ has the same parity as $p$.
  Conversely, if $a_{p}$ has the same parity as $p$ for any $p$, then
  $a_{ \ell+m, m } =0$ for odd $\ell$, so $b_{\ell} $ vanishes to infinite
  order on the diagonal for odd $\ell$, so we can assume that $b_{\ell} =0$. The
  proof for odd elements is the same.
\end{proof}
\clb



\subsection{Filtration and symbol}
For an operator $(P_k ) \in \lag (A,B)$, we have two different ways of writing
the expansion of its Schwartz kernel: a global one \eqref{eq:expansion} with
coefficients $b_\ell$ and a local one \eqref{eq:new_expansion} with
coefficients $a_p$. We now discuss the uniqueness of these coefficients. 
Recall that $(P_k) \in \lag_q (A,B)$ if in all the
  local expansions \eqref{eq:new_expansion}, the coefficients $a_{p} $ are zero for
  $p < q$.
  \begin{prop} \label{prop:filtration-symbol}
    \hspace{1em}
    \begin{enumerate}
\item In the local expansions \eqref{eq:new_expansion}, the coefficients
  $a_{p}$ are uniquely determined by the section $E$, the functions $(u_i)$
  and the frames of $A$ and $B$.
  \item In the global expansion \eqref{eq:expansion}, the Taylor expansions of the
    coefficients $b_{\ell}$ along the diagonal are uniquely determined by the
    section $E$.
    \item    $  (P_k) \in \lag_q (A,B)$ iff $\bigl(  \forall \, \ell \in\Z, \;
      b_{\ell} = \bigo (  q - \ell) \bigr)$ iff $| P_k |  = \bigo ( k^{n-\frac{q}{2}}) $.
\item If $(P_k)  \in \lag_q (A,B)$, then the coefficient $a_q$ of the local
  expansion \eqref{eq:new_expansion}, viewed as a
  section of $ \pol ( TM ) \otimes B \otimes \con{A} \rightarrow U$
  does neither depend on $E$ nor on the functions $(u_i)$.  Furthermore,
  \begin{gather} \label{eq:a_n_fonction_b_ell}
    a_q = \sum_{\ell + m ( \ell) = q} b_{\ell, q - \ell}
  \end{gather}
  where the $b_{\ell}$ are
  the coefficients of the global expansion \eqref{eq:expansion} and $b_{\ell,
    q-\ell}$ is defined as in \eqref{eq:taylor_expansion}.   
\end{enumerate}
\end{prop}
\begin{proof}
 Assertions 1, 2 and 3 follow from the following facts:
Let $f_0$, \ldots, $f_q$ in $\Ci (M^2)$. Let $\psi = -2
\ln |E|$. Then 
\begin{gather} \label{eq:esti_iff}
e^{-k \psi} \sum_{\ell =0}^q k^{-\frac{\ell}{2}}
f_{\ell} = \bigo ( k^{-\frac{q}{2}}) \quad \Leftrightarrow \quad  f_0 \in \bigo
(q),\, \ldots ,\; f_q \in \bigo (0) .
\end{gather}
Indeed, recall that $\psi \geqslant 0 $, is in $\bigo ( 2 )$ and
its Hessian is non-degenerate in the direction
transverse to the diagonal. The converse of \eqref{eq:esti_iff} follows from the same proof as
Lemma \ref{lem:estimate}. The direct sense of \eqref{eq:esti_iff}  follows from  \cite[Proposition 2.4 and Remark 2.5]{oim}.

From this, we deduce that $|P_k | = \bigo ( k^{n -\frac{q}{2}})$ iff
$b_{\ell} = \bigo (q-\ell)$ for any $\ell$. Since $a_{p} = \sum_{\ell + m ( \ell) \leqslant
  p} b_{\ell , p -\ell}$ by the proof of Lemma \ref{lem:new_expansion}, $(b_\ell = \bigo (q-\ell)$ for any $\ell$) iff
$(a_p = 0,$ for any $p < q$). This last condition is the definition of
$\lag _q ( A,B)$. We have just proved Assertion 3. This implies that $|P_k |
=\bigo (k^{-\infty})$ iff ($a_{p} = 0$ for any $p$) iff $(b_{\ell} = \bigo(
\infty)$ for any $\ell$), which proves Assertions 1 and 2.

For the fourth assertion, since $P \in \lag_q(A,B)$, we have $b_{\ell} = \bigo
( q-\ell)$ for any $\ell$, so we can assume that $\ell + m ( \ell) \geqslant
q$, so
$$a_q =
\sum_{\ell + m ( \ell ) \leqslant q } b_{\ell, q- \ell} =\sum_{\ell + m ( \ell
  ) =  q } b_{\ell, q- \ell} .$$
Since $b_{\ell} = \bigo ( q-\ell)$,   the $(q-\ell)$-th order term in the Taylor expansion of
$b_{\ell}$ is intrinsically defined as a function
$$ \xi \in T_xM \rightarrow b_{\ell, q- \ell}
(x,\xi) \in B_x \otimes \con{A}_x,$$ so we can view $a_q(x,\cdot)$ as an element
of $\pol (T_xM) \otimes  B_x \otimes \con{A}_x$.

It remains to prove that for $\ell + m ( \ell) =q$, $b_{\ell, q-\ell}$ does
not depend on the choice of $E$. With the notation of the proof of Lemma
\ref{lem:changement_E}, this amounts to prove that $b'_{\ell} = b_{\ell} +
\bigo ( q-\ell +1)$. This follows from \eqref{eq:changement_E},
because $b_{\ell + 2p } g^p = \bigo ( m (\ell+ 2 p) +3 p)$ and $m(\ell +2p ) +
3p \geqslant q - (\ell+ 2p ) + 3p = q + p - \ell \geqslant q + 1 -\ell$ when $p
\geqslant 1$.
\end{proof}

We are now ready to define the symbol map
$$\si_q : \lag_q (A,B) \rightarrow \Ci (
\symb(M) \otimes \op{Hom} (A,B)).$$
First, for any $x \in M$, $T_xM$ is a
Hermitian space, so it has an associated algebra $\widetilde{\symb} ( T_xM)$
with a map $\op{Op}:\pol ( T_xM) \rightarrow \widetilde{\symb} ( T_xM)$ as in Section
    \ref{sec:landau-levels-cn}.

    For any $(P_k)$ in $\lag_q (A,B)$, by the
    fourth assertion of Proposition \ref{prop:filtration-symbol}, $a_q( x,
    \cdot) \in \pol ( T_xM) \otimes B_x \otimes \con{A}_x$. Identifying $B_x
    \otimes \con{A}_x$ with $ \op{Hom} (A_x , B_x)$, we set
\begin{gather} \label{eq:def_si_tilde}
    \widetilde{ \si}_q (P)  (x) := \op{Op} ( a_q(x, \cdot)) \in
    \widetilde{\symb} ( T_xM) \otimes \op{Hom}(A_x, B_x )
  \end{gather}
  Recall that we have an isomorphism $\widetilde{\symb} ( T_xM) \simeq \symb (T_xM)$
defined by restriction from $\pol ( T_xM)$ to $\D ( T_xM)$.
\begin{defin} \label{def:symbol}
  $\si _q (P)
(x) \in \symb ( T_xM) \otimes \op{Hom}(A_x, B_x )$ is defined as the
restriction  of $\widetilde{\si}_q (P) (x)$. 
\end{defin}
\subsection{Proofs of the results of Section \ref{sec:operators}} \label{sec:proofs-results} 
We now give the proof of Proposition  \ref{prop:lag}, Theorem
\ref{theo:lag} and Theorem \ref{theo:parity}.

\begin{proof}[Proof of Proposition \ref{prop:lag}]
The first assertion is an easy consequence of the definition of $\lag_q (A,B)$
by the local expansions. In the second assertion, the characterisation in
terms of pointwise norm is the third assertion of Proposition
\ref{prop:filtration-symbol}. By Lemma \ref{lem:estimate} or Lemma
\ref{lem:estim_expansion}, every $(P_k) \in \lag_q(A,B)$ satisfies $\| P_k \|
= \bigo ( k^{-\frac{q}{2}})$. For the converse, it suffices to show that if
$\si_0 ( P ) \neq 0$, then $\| P _k \| \geqslant c >0$.  This is a consequence
of Corollary \ref{cor:norm_estim}. 
The third assertion is straightforward.  The
fourth assertion is a variation on Borel Lemma, cf. for instance
\cite[Proposition 2.1]{oim}.
\end{proof}

\begin{proof}[Proof of Theorem \ref{theo:lag}]
In Definition \ref{def:symbol}, we have defined a map
  $$ \si_q : \lag_q(A,B) \rightarrow \Ci ( M , \symb (M) \otimes \op{Hom}
  (A,B)) .$$
  having kernel $\lag_{q+1} (A,B)$ by \cB the injectivity of $\op{Op}$, cf.
  Proposition \ref{prop:tilde_symb_and_op}. \clb
  To prove that it is surjective, we show
  that that for any $c \in \Ci ( M , \pol (TM) \otimes B \otimes \con{A})$,
there exists $P \in \lag_q (A,B)$ such that in the local expansions
  \eqref{eq:new_expansion},  $a_q =c$. To do this, let
  \cB $d  \in \N$ \clb be an upper bound of the degree of $c(x, \cdot)$ for any $x\in
   M$. For any \cB
  $m =0, \ldots , d$, \clb  let $c_m(x,\cdot)$ be the homogeneous component with
  degree $m$ of $c(x, \cdot)$. Choose a section $b_{q-m}$ of $B \boxtimes
  \con{A}$ vanishing to order $m$ along the diagonal and satisfying $b_{q-m} (x+ \xi, x) = c_m (x,\xi) + \bigo ( m+1)$. Then
  we set  \cB
   $$ P_k (x,y ) := \Bigl(\frac{k}{2\pi} \Bigr)^n E^k(x,y) \sum_{\ell = q-
    d }^q   k^{-
    \frac{\ell}{2}} b_{\ell} (x,y) .$$ \clb
  Since $b_\ell = \bigo ( q-\ell)$ for any $\ell$, $(P_k) \in \lag_q (A,B)$
  \cB by
  Assertion 3 of Proposition \ref{prop:filtration-symbol}. \clb
  By  \eqref{eq:a_n_fonction_b_ell}, we have 
$a_q = \cB \sum_{\ell = q-d}^q b_{\ell, q- \ell}   = \sum_{m=0}^{d} c_{m} \clb
=
c$, as was to be proved. 

Let us prove the remaining assertions of Theorem   \ref{theo:lag}. Let $P \in \lag_q (A,B)$.
Assertion 1, that is  $\si _q ( P) = \si_0 ( k^{q/2}
P)$, follows directly from the local expansions
\eqref{eq:new_expansion}. Let us prove Assertion 2. If $f \in \Ci ( M ,
\op{Hom} (B,C))$, then the Schwartz kernel of $ P_k'= f \circ P_k$ is $P_k' (x,y) =
f(x) (P_k (x,y))$, so $P_k'$ has the same expansion \eqref{eq:expansion} as
$P_k$ with $b'_{\ell} (x,y) = f(x)(b_{\ell}(x,y))$ instead of $b_{\ell}$,
which implies that $P'_k$ belongs to $\lag_q (A,B)$ with the same function
$\ell \mapsto m(\ell) $. Furthermore, with the notation \eqref{eq:taylor_expansion}, $b'_{\ell,
  m(\ell)} (x, \cdot ) = f(x) b_{\ell, m(\ell)} (x, \cdot)$, which implies by
\eqref{eq:a_n_fonction_b_ell} that $\si_q( P')(x) = f(x) \circ \si_q(P)(x)$.

Let us prove Assertion 3. Since $P_k^* (x,y) = \con{P_k (y,x)}$, the Schwartz
kernel of $P_k^*$ has the expansion \eqref{eq:expansion} with $E' (x,y) =
\con{E (y,x)}$ instead of $E$ and $b'_{\ell} (x,y) = \con{b_{\ell} (y,x)}$
instead of $b_{\ell}$. By \eqref{eq:symE}, we deduce that $(P_k^*) \in \lag_q
(B,A)$. Furthermore, $b'_{\ell, m(\ell)} (x, \xi) = \con{b}_{\ell,m(\ell)}(x,
- \xi)$ so $a_q' (x,\xi) = \con{a}_q (x, -\xi)$. By \eqref{eq:sch_ker},
$\op{Op} (q )^* = \op{Op} (r)$ with $r( \xi ) =   \con{q} ( - \xi)$, so $\si_q(P^*) = \si_q(P)^*$. 

Let us prove assertion 5. By \eqref{eq:new_expansion},
$$ P_k(x,x) = \frac{k^{n-q/2}}{ (2\pi)^n} \bigl( a_q(x,0 ) + \bigo (
k^{-\frac{1}{2}} ) \bigr) $$
and by the first equation of \eqref{eq:rel_bizare}, $a_q(x,0) = \op{tr}
(\si_q(P) (x))$.

\cB Let us prove half of Assertion 6. More precisely, we will deduce from
Assertion 4  that for any $P \in \lag (A,B)$, we have
\begin{gather} \label{eq:tbp}
\limsup_{k \rightarrow \infty} \| P_k \| \leqslant \sup_{x \in M} \|
  \si_0 (P)(x) \| 
\end{gather}
With the lower bound provided by Corollary \ref{cor:norm_estim}, this will
show Assertion 6. Let $f :=  \si_0(P)^* \si_0 ( P)$.  Let  $m \in \N$ be
sufficiently large so that $\pi f \pi = f $ where for any $x \in M$, $\pi (x)
\in \symb (T_xM)
\otimes \op{End} A_x $ is the selfadjoint projector onto $\D_{\leqslant m } (T_xM) \otimes
A_x$.  Then for any $C >   \sup_{x \in M} \|
  \si_0 (P)(x) \|$, there exists a symbol $g \in \Ci ( M , \symb (M) \otimes
  \op{End} A)$ such that $g ^* = g$, $\pi g \pi = g $, $g^2 = C^2 \pi - f$.
  Indeed, $g(x)$ is zero on $\D_p (T_xM) \otimes A_x$ for any $ p > m$,
  and $g(x) $ is the positive square root of $C  - f$ on $\D_{\leqslant m}
  (T_xM ) \otimes A_x$. Let $\Pi$ and $Q$ in $\lag (A)$ be self-adjoint and
  having symbol $\pi$ and $g$ respectively. Then by Assertion 4,
  $$ C^2 \Pi^2 - P^* P = Q^2 + R $$
  with $R \in \lag_1 (A,B)$. So
\begin{xalignat*}{2} 
  \| P_k \Psi \| ^2 = \langle P_k^* P_k \Psi, \Psi \rangle = & C^2 \| \Pi_k
  \Psi \|^2 - \| Q_k  \Psi \|^2 - \langle R_k \Psi , \Psi \rangle \\ \leqslant
  & (
  C^2 \| \Pi_k \|^2 + C' k^{-\frac{1}{2}} ) \| \Psi \|^2 . 
\end{xalignat*}
by Assertion 2 of Proposition \ref{prop:lag}. Since $\pi^2 = \pi$, by
Assertion 4 and Assertion 2 of Proposition  \ref{prop:lag} again, $\Pi_k^2 -
\Pi_k = \bigo ( k^{-\frac{1}{2}})$ so $\| \Pi_k \| \leqslant 1 + \bigo (
k^{-\frac{1}{2}})$. Consequently, $\| P_k \| \leqslant C ( 1 + \bigo (
k^{-\frac{1}{2}}))$ which implies \eqref{eq:tbp}. 
\clb

It remains to prove Assertion 4. Let $(P'_k) \in
\lag_{q'} (B,C)$. We will prove that $Q_k = P'_k \circ P_k$ belong to $\lag
_{q'' } ( A,C)$ with $q'' = q  + q'$ and compute its symbol.  Since the composition of operators 
with kernels in $\bigo ( k^{-p})$ and $\bigo ( k^{-\ell})$ respectively,
has a kernel in $\bigo  ( k^{-(p+\ell)})$, we can consider each summand of the
expansions \eqref{eq:expansion} for $P_k$ and $P'_k$ separately. In other
words, we can assume that $P_k = \bigl( \frac{k}{2\pi} \bigr)^n E^k f$
and $P'_k = \bigl( \frac{k}{2\pi} \bigr)^n E^k f'$ with $f = \bigo ( q)$ and $f' = \bigo (q')$. So 
\begin{gather} \label{eq:R_kernel}
  Q_k (x,z) = \Bigl( \frac{k}{2\pi} \Bigr)^{2n}  \int_M  (E(x,y) \cdot E(y,z)
  )^k g(x,y,z) d \mu_M (y)
\end{gather}
with $g(x,y,z ) = f(x,y) f' (y,z)$. Observe that $g$ vanishes to order $q''$
along $\Si = \{ (x,y,z) \in M^3 , \; x= y =
z \}$.

By \eqref{eq:normE}, $|E(x,y) \cdot
E(y,z) | < 1 $ if $(x,y,z) \notin \Si$. This implies first that the Schwartz
kernel of $(Q_k)$ is in $\bigo ( k^{-\infty})$ outside the diagonal.
Furthermore, to compute $Q_k$ on a neighborhood of $(p,p)$ up to a $\bigo (
k^{-\infty})$, we can reduce the integral \eqref{eq:R_kernel} to a
neighborhood of $p$. 
So we can work locally. 

Introduce a local
orthonormal frame $(\partial_i, i =1, \ldots , n)$ of $T^{1,0}M$ on an open
neighborhood $U$ of
$p$ in $M$. Let  $\si_i \in \Ci (U^2)$, $i =1,\ldots n$ be such that
\begin{gather} \label{eq:dsi_i--partial_j}
d\si_i ( \partial_j, 0 ) = \delta_{ij} + \bigo (1) , \qquad d\si_i ( \con{\partial}_j, 0)
= \bigo (1)
\end{gather}
for any $i$ and $j$.
Observe that if the $z_i$ are
coordinates as in \eqref{eq:coordonnees_z}, then
\begin{gather} \label{eq:si_z}
\si_i (x,y) = z_i(x) - z_i
(y) + \bigo_{(p_0, p_0)} (2).
\end{gather}
So we can use the functions $u_i = \op{Re} \si _i$ and $u_{i+n}
= \op{Im} \si_i$ when we write the Taylor expansion
\eqref{eq:taylor_expansion} and the local expansion \eqref{eq:new_expansion}.

Restricting $U$ if necessary, we can assume that for any $z$, the map $
y\in U  \rightarrow (\si_i (y,z)) \in \C^n$ is a diffeomorphism onto its
image. Let $\mu_z$ be the pull-back of the volume $\mu_n $ by this map. By \eqref{eq:si_z}, we have $\mu_M (y) = \rho (y,z) \mu_z (y)$ with
$\rho \in \Ci (U^2)$ satisfying $\rho (y,y) = 1$. 

Now using the expressions \eqref{eq:E_linearise} and \eqref{eq:si_z}, we
readily prove that  
$$  E(x,y) \cdot E(y,z) = e^{ \varphi (x,y,z) + r(x,y,z) } E
(x,z), $$
where $r(x,y,z) = \bigo_{\Si} (3)$ and $ \varphi(x,y,z) =  (\si (x,z) - \si
  (y,z))\cdot \con{\si}  (y,z) $. 
Arguing as in the proof of Lemma \ref{lem:changement_E}, it comes that
$$ ( E(x,y) \cdot E(y,z) )^k  = E^k(x,z) e^{k \varphi ( x,y,z)} \sum_{\ell =0 }^{N}
\frac{k^{\ell}}{\ell !}  (r (x,y,z))^{\ell} + \bigo \bigl( k^{-\frac{1}{2} (N+1) } \bigr)$$
so the integrand of \eqref{eq:R_kernel} is equal to
$$ E^k(x,z) e^{k \varphi ( x,y,z)} \sum_{\ell =0 }^{N}
k^{\ell} g_{\ell}  (x,y,z) \; d\mu_z (y)   + \bigo \bigl( k^{-\frac{1}{2} ( q''
  + N+1)} \bigr)
$$
with $g_{\ell} (x,y,z) = \rho(y,z) g (x,y,z)  (r (x,y,z)) ^\ell / (\ell !) =  \bigo_{\Si} ( q''+ 3\ell)$. 

For any $z \in U$, we write the Taylor expansion of $(x,y) \rightarrow
g_{\ell}(x, y,z)$ at $(z,z)$ with the
coordinates system
$$(x,y) \rightarrow \op{Re} \si_i (x,z), \op{Im} \si_i (x,z) ,
\op{Re} \si_i (y,z) , \op{Im} \si_i (y,z).$$
We obtain 
$$ g_\ell(x,y,z) = \sum_{m= q'' + 3 \ell}^{p}  h_{\ell,m} (z, \si (x,z), \si (y,z))
+ \bigo_{\Si} ( p+1) $$
with $h_{\ell,m} ( z, \xi , \eta)$ homogeneous polynomial in $\xi, \eta$ with
degree $m$.  Arguing as in Lemma \ref{lem:new_expansion}, we obtain that
\begin{gather} \label{eq:dev_loc_R}
Q_k (x,z) = \Bigl(\frac{k}{2\pi} \Bigr)^n   E^k(x,z) \sum_{\ell= 0}^{N}
\sum_{m= q''+ 3 \ell}^{q''+ 2 \ell +N}
k^{\ell - \frac{m}{2}}   I_{\ell,m} (x,z) +  \bigo \bigl( k^{-\frac{1}{2} ( q''
  + N+1)} \bigr)
\end{gather}
with
$$ I_{\ell,m} (x,z) =  \Bigl(\frac{k}{2\pi} \Bigr)^n  \int_U e^{k
  \varphi (x,y,z) } h_{\ell,m} (z, k^{\frac{1}{2}} \si (x,z) , k^{\frac{1}{2}}
\si (y,z) ) \; d\mu_z (y) $$
Set $u_i = \si_i (x,z)$ and let us use the coordinates $v_i = \si_i (y,z)$ for the
integration so that $\varphi (x,y,z) = u \cdot \con v - |v|^2$ and $d\mu_z (y) =  |dv d \con{v}|$. It comes that
$ I_{\ell, m} (x,z) = J_{\ell, m} ( z, k^{\frac{1}{2}} \si ( x,z))$ with 
\begin{gather} \label{eq:j_ell_m}
J_{\ell , m }  (z, u ) = \Bigl(\frac{k}{2\pi} \Bigr)^n \int e^{
  k^{\frac{1}{2}} u \cdot \con v - k |v|^2} h_{\ell,m} ( z, u , k^{\frac{1}{2}
} v )  \; d \mu_n (v) 
\end{gather}
where we integrate on a neighborhood of the origin in $\C^n$. We can actually
integrate on $\C^n$ because this will modify $E^k(x,z)  I_{\ell,m} (x,z)$
by a $\bigo ( e^{-k/C})$. Indeed, $| E(x,z) | = e^{-\frac{1}{2} |u|^2} + \bigo
(|u|^3)$ so $|E(x,z) | = \bigo ( e^{-\frac{1}{3} |u|^2})$ so
$$|E(x,z) e^{u
  \cdot \con v - |v|^2} | = \bigo ( e^{ - \frac{1}{3} |u|^2 + |uv| - |v|^2 } ) = \bigo (
e^{ -\frac{1}{4} |v| ^2} ) $$
and we conclude by using that $ \int_{|v|\geqslant \epsilon} e^{- \frac{k}{4} |v|^2} |v|^m \; |dv d
\con{v}| = \bigo ( e^{-k/C})$ for any $\ep >0$ and $m \in \N$. 

Taking the integral
\eqref{eq:j_ell_m} over $\C^n$, it comes that
\begin{gather} \label{eq:j_ell-m}
  J_{\ell,m}  ( z,u) = (2 \pi)^{-n} \int_{\C^n} e^{u \cdot \con v - | v| ^2 }
h_{\ell, m }  ( z, u, v)  \; d\mu_n (v)
\end{gather}
So $J_{\ell,m}$ does not depend on $k$. Furthermore it is polynomial in $u$.
To see this, it suffices to view $h_{\ell, m } ( z,u, v)$ as a polynomial in the
variables $u-v$, $v$ and to compare with the formula \eqref{eq:sch_ker}. So
$Q_k (x,z)$ has the local expansion \eqref{eq:new_expansion}, so $(Q_k)$  belongs to
$\lag _{q''} ( A, C)$. Its symbol  is given by the
leading order term in \eqref{eq:dev_loc_R} which corresponds to $\ell =0$ and
$m = q''$, that is
$$\widetilde{\si}_{q''} ( Q ) (x) =  \op{Op} ( J_{0, q''} (x,
\cdot)).$$
We can compute it in terms of the symbols of $P$ and $P'$ as follows:
 by \eqref{eq:a_n_fonction_b_ell}, $\widetilde \si_{q} ( P)(x) = \op{Op} ( a_q(x, \cdot))$ where $\xi \rightarrow
a_q(x,\xi)$ is the homogeneous polynomial of degree $q$ such that $f(x,y) = a_q (y, \si
( x,y) ) + \bigo ( q+1)$. Similarly, $\widetilde \si_{q'} ( P')(x) = \op{Op} (
a'_{q'}(x, \cdot))$ with $f'(x,y) = a'_{q'} (y, \si
( x,y) ) + \bigo ( q'+1)$. Now by  \eqref{eq:si_z}, $\si (x,y) = \si (x,z) - \si ( y,z)
+ \bigo _{\Si} (2)$ and  it comes that
\begin{xalignat*}{2}
g(x,y,z) & =  a_q ( y, \si (x,z) - \si ( y,z)) a'_{q'} ( z, \si (y,z)) +
\bigo_{\Si} (q'') \\
& =   a_q ( z, \si (x,z) - \si ( y,z)) a'_{q'} ( z, \si (y,z)) +
\bigo_{\Si} (q'') 
\end{xalignat*}
leading to  $ h_{0, q''} (z, u,v) = a_q ( z, u - v) a'_{q'} (z,v) $ and using
first \eqref{eq:sch_ker} and then \eqref{eq:rel_bizare}, we have that  
\begin{xalignat*}{2} 
  \widetilde{\si}_{q''} ( Q )(x) & =  \op{Op} ( J_{0, q''} (x,
\cdot)) = \op{Op} \bigl( \op{Op} (a_q ( x,\cdot))  a'_{q'} (x,\cdot) \bigr) \\
& =  \op{Op}
(a_q ( x,\cdot)) \circ \op{Op} (a'_{q'} (x,\cdot))   = \widetilde \si_{q} ( P)
(x)
\circ \widetilde \si_{q'} ( P' )(x)
\end{xalignat*}
as was to be proved. 
\end{proof}

\begin{proof}[Proof of theorem \ref{theo:parity}]
The first assertion follows from the definition of the parity and Proposition
\ref{prop:filtration-symbol}.  For the composition, it suffices to consider
the case treated in the previous proof: we start from $(P_k)$ and $(P'_k)$
both even. Since $h_{\ell,m}(x,\cdot) $ has degree $m$, $J_{\ell,m}(x, \cdot)$ given by
\eqref{eq:j_ell-m} has the same parity as $m$, so by \eqref{eq:dev_loc_R},
$(P'_k \circ P_k)$ is even. Last assertion is simply the fact that $\op{Op}
(a_q(x,\cdot) )$ has the same parity as $a_q (x,\cdot)$ by Proposition  \ref{prop:tilde_symb_and_op}.
\end{proof}

\subsection{Peaked sections}
In this section, we state and prove a generalisation of Proposition
\ref{prop:peaked-sections}.  Consider an auxiliary bundle $A$. Let us choose a base point $x \in M$, with a coordinate chart $U$ centered at
$x$, and a trivialisation $A|_U \simeq U \times A_x$.  To any $f \in \pol (T_xM
) \otimes A_x$, we associate the section of $L^k \otimes A$
\begin{gather} \label{eq:phi_peaked_section_+}
\Phi_k^f (x+\xi   ) = \Bigl( \frac{k}{2\pi} \Bigr)^{\frac{n}{2}} E^k (x+ \xi, x
) \, f(k^{\frac{1}{2}} \xi) \, \psi (x+ \xi) 
\end{gather}
where $E$ is chosen as in Section \ref{sec:section-ee}, and $\psi \in \Ci_0(U)$ is equal
to $1$ on a neighborhood of $x$.

\cB The space $\pol (T_xM
) \otimes A_x$ has a natural scalar product obtained by tensoring  the scalar
product \eqref{eq:scal_prod_2} of $\pol (T_xM
)$  with the Hermitian metric of $A$.\clb

\begin{prop} \label{prop:peaked-sections_++}
  $ $
  \begin{enumerate}
    \item   For any $f$, $g \in \pol ( T_x M ) \otimes A_x$, $\bigl\langle
      \Phi_k^f, \Phi_k^g \bigr\rangle = \langle f , g \rangle + \bigo
      (k^{-\frac{1}{2}} )$.
      \item For any $f  \in \pol ( T_x M ) \otimes A_x$ and $Q \in \lag (
        A,B)$, $ Q_k \Phi_k^f = \Phi_k^h + \bigo ( k^{-\frac{1}{2}})$ where
        $h = \tilde{\si}_0 ( Q)(x) \cdot f \in \pol ( T_xM) \otimes B_x$. 
      \end{enumerate}
      \end{prop}
In the second part, we used the symbol $\tilde{\si}_0 (P)$ defined in
\eqref{eq:def_si_tilde}, and $\Phi_k^{h}$ is defined as $\Phi_k^f$ with a
trivialisation of $B$.

\begin{proof} Consider the operator $P^f \in \lag (\C, A)$ with Schwartz
  kernel \cB
  $$P^f_k (y+ \xi, y) = \Bigl ( \frac{k}{2\pi} \Bigr)^n E^k (y+ \xi, y) \, f (
  k^{\frac{1}{2}} \xi)\,  \psi (y+ \xi) \, \psi (y).$$\clb 
On one hand, $( \frac{k}{2\pi})^{\frac{n}{2}} \Phi^f_k = P_k^f (\cdot ,x)$. On
the other hand, $\tilde{\si}_0(P^f)(x) = \op{Op} ( f)$. So we can compute
the scalar product of $\Phi^f_k$ and $\Phi^g_k$ as a composition of Schwartz
kernels 
\begin{xalignat*}{2}
\Bigl ( \frac{k}{2\pi} \Bigr)^n \bigl\langle \Phi_k^f, \Phi_k^g \bigr\rangle & =
((P_k^g)^* P_k^f)(x,x) \\ & = \Bigl(\frac{k}{2\pi} \Bigr)^n \op{tr}_{\D ( T_xM)} ( \op{Op} (g)^*
\op{Op} (f))  + \bigo ( k^{n-\frac{1}{2}})
\end{xalignat*}
by the last part of Theorem \ref{theo:lag}. \cB To conclude, we have by
the second equation of \eqref{eq:rel_bizare} with $\op{Op} (q) = \op{Op} (g)^*$ that  $ \op{Op} (g)^* \op{Op} (f)  = \op{Op} (  \op{Op}
(g)^* f) $ and then by the first and third equations of \eqref{eq:rel_bizare} 
$$  \op{tr}_{\D ( T_xM)} (\op{Op} (  \op{Op} (g)^* f)) = (\op{Op} (g)^* f )(0)
= \langle f , g \rangle .$$  \clb
The proof of the second part is similar, we have
$$ \Bigl ( \frac{k}{2\pi} \Bigr)^{\frac{n}{2}} Q_k \Phi^f_k = (Q_k
P_k^f)(\cdot, x) $$
By Theorem \ref{theo:lag},  $(Q_k P_k^f) \in \lag ( \C,B)$ with symbol at $x$
equal to $$\tilde{\si}_0 (Q) (x) \circ \op{Op} ( f)
= \op{Op} (\tilde{\si}_0 (Q) (x)  f  ) = \op{Op} (h)$$ by
\eqref{eq:rel_bizare}. \cB So by the local expansion \eqref{eq:expansion_kernel},
$$ ( Q_k P_k )( \cdot ,x ) = \Bigl(\frac{k}{2\pi} \Bigr)^{\frac{n}{2}} \Phi_k^h
 + r_k $$
with $r_k =R_{k} ( \cdot, x)$ where $(R_{k}) \in \lag_ 1 (\C,B)$. Finally, $\|
r_k \|^2 =  ( R_k^*  R_k ) (x,x) = \bigo ( k^{n-1})$ because $(R_k^*R_k ) \in
\lag_2 ( \C)$ by Theorem \ref{theo:lag}.      \clb
\end{proof}

We deduce the following lower bound for the operator norm of operators of
$\lag (A,B)$. If $\rho \in \symb_x (M) \otimes \op{Hom} (A_x, B_x)$, then we
denote by $\| \rho \|$ the norm
$$ \| \rho \| = \sup \bigl\{  \| \rho f \| / \| f \| ,\; f \in \D (T_xM)
\otimes A_x, \; f \neq 0  \bigr\}.$$

\begin{cor} \label{cor:norm_estim}
  For any $P \in \lag (A,B)$, we have
  $$ \liminf_{k \rightarrow \infty} \| P_k \| \geqslant \sup_{x \in M} \|
  \si_0 (P)(x) \| $$
\end{cor}

\begin{proof}
By Proposition \ref{prop:peaked-sections_++}, for any  $f \in \D (
T_xM) \otimes A_x$ non zero,
$$  \frac{\| P_k \Phi_k^f \|}{\| \Phi_k^f \| } = \frac{ \| \si_0 (P) (x) f
  \|}{\| f \| } + \bigo ( k^{-\frac{1}{2}}) .$$
So $ \liminf_{k \rightarrow \infty} \| P_k \| \geqslant  \| \si_0 (P) (x) f
  \| / \| f \|$ .
\end{proof}

\section{Derivatives} \label{sec:derivatives}

The class $\lag (A, B)$ has been defined without any control on the
derivatives of the Schwartz kernels. The reason was merely to simplify the
exposition but in the applications it is natural and necessary to understand
the composition of operators of $\lag(A, B)$ with covariant derivatives. 
 We start from general considerations, then we define a subclass $\lag
 ^{\infty}(A,B)$ where the asymptotic expansion of the Schwartz kernels hold with respect to a convenient $C^{\infty}$ topology. Finaly we apply
 this to complete the proofs of the theorems stated in the introduction. 
\subsection{The class \texorpdfstring{$\biginf (k^{-N})$}{O-infinity}} 

Consider as before a Hermitian line bundle $L \rightarrow M$ and an auxiliary
Hermitian vector bundle  $A
\rightarrow M$.
 Let $\mathcal{F}$ be the space of families 
$$s= (s_k \in \Ci (M , L^k \otimes A)
,\,  k \in \N ).$$ 
Recall that $s \in \bigo ( k^{-N})$ if for any $x \in M$, $|s_k (x)| = \bigo ( k^{-N})$ with a
$\bigo $ uniform on any compact subsets of $M$. Here we do not assume that $M$
is compact.

The definition of $\biginf$ involves the derivatives. If $L$ and $A$ are trivial
bundles so that $(s_k)$ is a sequence of $\Ci ( M, \C^r)$ with $r$ the rank of
$A$, then we say that $(s_k) \in \biginf ( k^{-N})$ if for any $m \in
\N$, the derivatives of order $m$ of $(s_k)$ are in $\bigo ( k^{-N+m})$. More
precisely, for any vector fields $X_1$, \ldots, $X_m$ of $M$, we require that
\begin{gather} \label{eq:derk}
X_1 \ldots X_m s_k = \bigo ( k^{-N+m}) .
\end{gather} 
So we loose one power of $k$ for each derivative. Because of this, the class
$\biginf ( k^{-N})$ is invariant by multiplication by $e^{ik h}$, where $h$ is
any real-valued function of $M$.

For actual vector
bundles $L$ and $A$, we introduce unitary frames $u$ and $(v_j)_{j=1}^{r}$ of $L$
and $A$ over the same open set $U$ of $M$ and write $s_k = \sum f_{k,j} u^k
\otimes v_j$ with $f_{k} \in \Ci ( U, \C^r)$. Then we say that $(s_k)$ belongs
to $\biginf ( k^{-N})$ if for all choices of unitary frames of $L$ and $A$,
the corresponding local representative sequence $(f_k)$ is in $\biginf ( k^{-N})$.
Observe that changing the frame $u$ of $L$ amounts to multiply $f_k$ by
$e^{ikh}$, so the condition that $f_k \in \biginf ( k^{-N})$ does not depend
on the frame choice when these frames are
defined on the same open set. 

The typical example of a family in $\biginf ( k^{-N})$  is an oscillating
sequence
$$s_k (x)  = k^{-N}
e^{-k \varphi (x) } a (x) $$
with $ \varphi \in \Ci ( M)$ having a non negative real part
and $a \in \Ci ( M, \C^r)$. More generally, for actual bundles, we can set
$$ s_k (x) 
= k^{-N} E^k(x)  a(x) $$ where $E \in \Ci ( M, L)$ is such that $|E|\leqslant 1$ and $a
\in \Ci ( M, A)$.

Obviously, if $N' \geqslant N$,  $\biginf ( k^{-N'} ) \subset \biginf ( k^{-N})$. Define  $\biginf (k^{-\infty}
) := \cap_N \biginf ( k^{-N})$. We will need the following result.

\begin{lemme} \label{lem:dev_der_bor}
Let $(s_\ell)$ be a sequence of $\mathcal{F}$ such that for any
$\ell$, $s_{\ell} \in
\biginf ( k^{-p(\ell)} ) $ where $(p(\ell))$ is an increasing real sequence, and $p(\ell)
\rightarrow \infty$ as $\ell \rightarrow \infty$. Then  
\begin{enumerate}
  \item 
There exists $s \in \Fam \cap \biginf ( k^{-p(0)})$ unique modulo $\biginf ( k^{-\infty}
)$ such that
\begin{gather} \label{eq:dev_Der}
s_k = \sum_{\ell =0}^{N-1} s_{\ell,k} + \biginf ( k^{-p(N)}) , \qquad
\forall N 
\end{gather}
\item Let $s \in \mathcal{F}$ such that $s \in \biginf (k^{p})$ for some
  $p$ and $s_k =
  \sum_{\ell =0}^{N-1} s_{\ell,k} + \bigo ( k^{-p(N)})$ for any $N$. Then $s
  \in \biginf (k^{-p(0)})$ and \eqref{eq:dev_Der} holds.
\end{enumerate}
\end{lemme}
 
The first part is a variation of Borel Lemma, the second part follows from
interpolation inequalities, cf. as instance \cite[Lemma 32.]{Shubin}

\cB In the sequel, we will apply this material to Schwartz kernels. So
instead of $M$, $L$, $A$, we will have $M^2$, $L \boxtimes \con{L}$ and $B
\boxtimes \con{A}$.  \clb
\subsection{Application to \texorpdfstring{$\lag (A,B)$}{L(A,B)}}

Choose a section $E$ as in section \ref{sec:section-ee} and let $b \in \Ci ( M^2
, B \boxtimes \con{A})$ vanishing to order $m$ along the diagonal. Then by the
same proof as Lemma \ref{lem:estimate}, the family $ (E^k b)$ is in
$\biginf (k^{-\frac{m}{2}})$. Actually, we even have a better result if instead
of using any derivatives, we only consider covariant derivatives for the
connection of $(L \boxtimes \con{L})^k \otimes (B \boxtimes \con{A})$ induced
by the connection of $L$ and any connections of $A$ and $B$.  
\begin{lemme} \label{lem:comp_der_lag}
 For any $\ell \in \N$, any vector fields $X_1$,  \ldots, $X_{\ell}$ of $M^2$, we have
  $ \nabla _{X_1} \ldots \nabla_{X_\ell} (E^k b) $ is in $\bigo (
  k^{-\frac{1}{2} ( m - \ell)})$.
\end{lemme}
The improvement is that we only loose a half power of $k$ for each derivative.

\begin{proof} The main observation is that $\nabla E $ vanishes on the
  diagonal. Indeed, $\nabla_{X} E =0$ on the diagonal when $X$ is tangent to
  the first factor because of the second equation in \eqref{eq:hypotheseE},
  but also when $X$ is tangent to the diagonal by the first equation in
  \eqref{eq:hypotheseE}. So on a neighborhood of the diagonal, we have \clr $\nabla_X
  E = f E$ \clb with $f \in \bigo (1)$. By Leibniz rule,
  $ \nabla_X ( E^k b) = E^k ( k f b + \nabla_X b ) $. Using this repeatedly,
  we obtain
  $$ \nabla _{X_1} \ldots \nabla_{X_\ell} (E^k b) = E^k ( k^{\ell} b_{\ell} +
  k^{\ell -1} b_{\ell-1} + \ldots + b_0 )$$ where
  $ b_{\ell} = \bigo ( m+ \ell )$, $b_{\ell-1 } \in \bigo ( m+ \ell - 2)$,
  \ldots, $b_0  \in \bigo ( m - \ell)$.  And we conclude as in the proof of
  Lemma \ref{lem:estimate}.
\end{proof}

Recall that the Schwartz kernel family of an operator $P \in \lag (A,B)$ has by definition
an expansion of the form 
\begin{gather} \label{eq:exp_der_part}
P_k (x,y) = \Bigl(\frac{k}{2\pi} \Bigr)^n E^k(x,y) \sum_{ \ell + m ( \ell ) \leqslant N } k^{- \frac{\ell}{2}} b_{\ell} (x,y)
  + R_{N,k} (x,y)  
\end{gather}
with $R_{N,k} \in \bigo ( k^{n- \frac{N+1}{2}})$. \cB Let $\lag^{\infty} (A,B)$
(resp.  $\lag_q^{\infty} (A,B)$) be the subspace of $\lag (A,B)$ (resp. $
\lag_q (A,B)$) consisting of the operator families having a Schwartz kernel in $\biginf
(k^n)$. \clb
Identifying operators and their kernels, 
$$\lag^{\infty} (A,B) = \lag (A,B) \cap \biginf (k^n), \quad \lag_q^{\infty} (A,B) = \lag_q (A,B) \cap \biginf (k^n)$$
By the following proposition, these new classes have the same properties than the $\lag (A,B)$ and this
follows directly from Lemma \ref{lem:dev_der_bor}.

\begin{prop} \label{prop:versionlinfty}$ $
  \begin{enumerate}
\item  \label{item:1}    If $P \in
\lag^{\infty} (A,B)$ and the expansion \eqref{eq:exp_der_part} holds with
$R_{N,k} \in \bigo ( k^{n- \frac{N+1}{2}})$ for any $N$, then $R_{N,k} \in
\biginf ( k^{n- \frac{N+1}{2}})$ for any $N$.
\item  \label{item:2}  For any $P \in \lag (A,B)$ there exist $Q \in
\lag^{\infty} (A,B)$ unique modulo $\biginf (k^{-\infty})$ such that $Q = P +
\bigo ( k^{-\infty})$.
\item  \label{item:3} \clr For any $ P \in \lag^{\infty}(A,B
  )$,
  \begin{enumerate}
    \item \label{item:a} the adjoint of $P$ belongs to $\lag^{\infty} (B,A)$, 
  \item \label{item:b} $(P_kQ_k) \in \lag ^{\infty} ( A,C)$ for any $Q \in \lag ^{\infty} ( B,C)$,
  \item \label{item:c} $(f \circ P_k)$ belongs to $\lag^{\infty} ( A,C)$ for any $f \in \Ci (M ,
    \op{Hom} (B,C))$; $(P_k \circ g) $ belongs to  $\lag ^{\infty} ( C,B)$ for any
    $g \in  \Ci ( M , \op{Hom} (C,A))$
    \item  \label{item:d} for any vector field $X$ of $M$ and connections on $A$ and $B$,
      $(k^{-\frac{1}{2}} P_k \circ \nabla^{L^k \otimes A}_X) $ and
      $(k^{-\frac{1}{2}} \nabla^{L^k \otimes B}_X \circ P_k)$  belong to
      $\lag^{\infty} ( A,B)$.

      Furthermore, if $(P_k)$ is even (resp. odd),
      these two operators are odd (resp. even). 
  \end{enumerate} \clb  
\item  \label{item:4} $\lag_q^{\infty} (A,B) = \lag (A,B) \cap \biginf ( k^{n -
    \frac{q}{2}})$, the restriction of $\si_q$ to $\lag_q^{\infty} ( A,B)$ is
  onto and has kernel $\lag_{q+1} ^{\infty} ( A,B)$. 
\end{enumerate}
\end{prop}
\begin{proof} Assertion \ref{item:1} follows the preliminary
observation on $E^kb$ and the second part of Lemma \ref{lem:dev_der_bor}. 
Assertion \ref{item:2} follows from the first part of Lemma 
\ref{lem:dev_der_bor}.  Claim \ref{item:b} follows from the fact that the
composition of two kernels in $\biginf ( k^{n})$ is in $\biginf (k^{2n})$, and
$\biginf ( k^{2n}) \cap \lag (A,C) = \lag^{\infty}  ( A,C)$ by the second part
of Lemma \ref{lem:dev_der_bor}. Claims \ref{item:a} and \ref{item:c}  are
straightforward. Ones proves claim \ref{item:d}  by arguing as in the proof of
Lemma \ref{lem:comp_der_lag}. Part \ref{item:4}
follows from the second part of Lemma \ref{lem:dev_der_bor}.
\end{proof}

\begin{rem} \label{sec:rem_linfty}
  We can adapt Theorems  \ref{theo:constr-proj} and
  \ref{theo:unitary_equivalence} to the spaces $\mathcal{L}^{\infty}$:  
  \begin{enumerate}
    \item  in
Theorem \ref{theo:constr-proj}, if we start with $P \in \lag^{\infty} ( A)$,
then $\chi ( P) \in \lag ^{\infty} (A)$.
\item in Theorem
\ref{theo:unitary_equivalence}, if $\Pi$ and $\Pi'$ are in
$\lag^{\infty} (A)$ and $\lag^{\infty} (B)$ respectively, then we can choose
$U \in \lag^{\infty} (A,B)$.
\end{enumerate}
In both cases, the only change in the proof is the
fact that for any families of operators $Q_k, Q'_k : \Ci
(M,L^k \otimes A) \rightarrow \Ci (M, L^k\otimes A)$ and $Q''_k : L^2 (M, L^k
\otimes A) \rightarrow L^2 (M, L^k \otimes A)$, by \cite[Section
4.3]{oim}, if the Schwartz kernel
families of $(Q_k)$ and $(Q'_k)$ are respectively  in $\biginf (k^{-N})$ and
$\biginf (k^{-N'})$, and the operator norms of $Q_k^{''}$ are in $\bigo (1)$,
then the Schwartz kernel family of $Q_k Q''_k Q_k'$ is in $\biginf (k^{-(N
  +N')})$. 
\qed \end{rem}

By Theorem  \ref{theo:lag}, we already know how to compute the symbols of
$P^*$, $PQ$, $f P$ or $gP$ in terms of the symbols of $P$ and $Q$. To complete
this, we compute the symbol of the compositions of $P$ with the covariant
derivatives $\nabla_X^{L^k \otimes A}$ and $\nabla_X^{L^k \otimes B}$. Recall
that for any $Y \in T_xM$, we defined in the introduction some  endomorphisms
$\rho (Y) \in \op{End} ( \D ( T_xM))$  in \eqref{eq:rho}. If $Y = U + \con V$ with $U, V \in T_x ^{1,0}
M$,  then $\rho ( Y)=
\rho ( U) + \rho ( \con V)$ where $\rho(U)$ is the multiplication by $i \om
(U, \cdot)$ and $\rho (\con{V})$ is the derivation with respect to $\con{V}$.

\begin{lemme} \label{lem:symb_der}
  For any $P \in \mathcal{L}^{\infty}( A,B)$ and vector field $X$
  of $M$, we have
  \begin{gather*}
    \si_0 (k^{-\frac{1}{2}} P_k \circ \nabla^{L^k
      \otimes A}_X) (x) = \si_0 ( P_k) (x) \circ \rho ( X(x) ) \\
  \cB   \si_0 (k^{-\frac{1}{2}} \nabla^{L^k
      \otimes B}_X \circ P_k) (x) = \rho ( X(x) ) \circ \si_0 ( P_k) (x) .  \clb
  \end{gather*}
  \end{lemme}
  \begin{proof} We deduce one formula from the other by taking adjoint. To
    prove the first one, it suffices by Proposition
    \ref{prop:peaked-sections_++} to show that if $(\Phi_k^f )$ is the peaked
    section associated to $ f\in \D(T_xM) \otimes A_x$, then
    $$k^{-\frac{1}{2}}
    \nabla_X \Phi_k^f = \Phi_k^g + \bigo ( k^{-\frac{1}{2}}) $$
    where $g = \rho
    ( X(x)) f $. This is easily checked if we use the normal coordinates as in
    \eqref{eq:coordonnees_z} centered at $p_0 = x$. We have to derivate \eqref{eq:phi_peaked_section_+}.
    We have first that $k^{-\frac{1}{2}} \nabla_X ( E^k) = k^{\frac{1}{2}} E^k (\nabla_X
    E)E^{-1}$ and by \eqref{eq:E_linearise}, $$(\nabla_{\partial_i} E)E^{-1} =
    - \con{z}_i + \bigo (2), \qquad (\nabla_{\con{\partial}_i} E)E^{-1} =
    \bigo (2).$$
    Second we have $ k^{-\frac{1}{2}} \partial_i f ( k^{\frac{1}{2}} \xi ) = 0$
    since $f \in \C [ \con{z}_1, \ldots, \con{z}_n]$ and $ k^{-\frac{1}{2}} \con{\partial}_i f ( k^{\frac{1}{2}} \xi ) =  
( \partial f / \partial \con{z}_i) (k^{\frac{1}{2}} \xi )$.  To conclude
recall that $\rho ( \partial_i) $ is the multiplication by $-\con{z}_i$ whereas
$\rho ( \con{\partial}_i)$ is the derivation with respect to $\con{z}_i$.  
  \end{proof}

\subsection{Kostant-Souriau operators and subprincipal estimates} 

In our context, the Kostant-Souriau operators are the operators of the form
$$ f +  \tfrac{i}{k} \nabla_X^{L^k \otimes A} : \Ci ( M , L^k \otimes A)
\rightarrow \Ci  (M , L^k \otimes A)$$ where $ f \in \Ci ( M)$ and $X$ is its
Hamiltonian vector field, that is  $\om (X, \cdot ) + df =0$.
\begin{lemme} \label{lem:ledernier} 
For any $f \in \Ci ( M, \R)$ with Hamiltonian vector field $X$, for any $P \in \lag^{\infty}_0 (
  A)$ we have
  \begin{itemize}
  \item  $ [f, P]$ belongs to $\lag_{1}^{\infty} (A)$
   \item  $[f,P] \equiv (ik)^{-1} [\nabla_X^{L^k
       \otimes A} , P]$ modulo $\lag_{2}^{\infty} (A)$.
   \end{itemize}
   So the commutator $[f +
  \frac{i}{k} \nabla_X^{L^k \otimes A} , P]$ belongs to $\lag_{2}^{\infty}(A)$. 
\end{lemme}
\begin{proof} The Schwartz kernel of $[f, P]$ is the product of $g(x,y)= f(x) -
  f(y)$ by the Schwartz kernel of $P$. Since $g$ vanishes along the diagonal,
  this implies that $[f, P] \in \lag_1 ( A)$. Furthermore, it follows from the
  definition \eqref{eq:def_si_tilde} of the symbol  that if $\wt{\si}_0 ( P)(x) = \op{Op} ( b)$
  with $b \in \pol (T_xM)$, then $\wt{\si}_1 ( [f, P] )(x) = \op{Op} ( \ell b )$ where
  $\ell = d_x f \in T_x^*M$. Now a computation from  \eqref{eq:sch_ker} shows
  that
  $$ [ \op{Op} ( b ) , a_i ] = \op{Op} ( z_i b) , \qquad [ \op{Op} (b) , a_i^*
  ] = \op{Op} ( \con{z}_i b ) $$
  where as in Section \ref{sec:landau-levels-cn} we use the annihilition and
  creation operators $a_i = \partial_{\con{z}_i}$, $a_i^* = \con{z}_i -
  \partial_{z_i}$. \cB  So working with normal coordinates at $p_0 =x$ as in
  \eqref{eq:coordonnees_z},
$$ \wt{\si}_1 ( [f, P] )(x) = \sum_i (\partial_{z_i} f)(x) [ \op{Op} (b), a_i
] +  (\partial_{\con{z}_i} f)(x) [ \op{Op} (b), a_i^* ]$$ 
Moreover $ i X = ( \partial_i f) \con{\partial}_i -   (\con{\partial}_if)
\partial_i $ at $x$.
Since $ \rho ( \con{\partial}_i)$ and $\rho (
  \partial_i)$ are the restrictions of $a_i$ and $-a_i^*$ to $\C[\con{z}_1,
  \ldots, \con{z}_n]$ respectively, we deduce that
  $$ \si_1 ( [f, P] ) = i [ \si_0 (P) , \rho(X)] .$$
On the other hand, by Lemma  \ref{lem:symb_der},   $\si_0 ( \bigl[ k^{-\frac{1}{2}}
\nabla_X^{L^k \otimes A}, P] ) = [ \rho(X), \si_0 (P) ]$. So 
$$ \si_1 ( [f +   \tfrac{i}{k} \nabla_X^{L^k \otimes A} , P] ) =0 $$
and the result follows. \clb 
\end{proof}

Let us apply this result to the Toeplitz algebra associated to a self-adjoint
projector $\Pi \in \lag^{\infty}(A) \cap \lag^{+}(A)$, whose symbol $\pi$ is
the projector on a subbundle $F$ of $\D (TM ) \otimes A$ having a definite
parity. Introduce the operators associated to $f \in \Ci ( M )$, $X , Y \in
\Ci ( M , TM)$, 
$$ T_k (f) = \Pi_k f \Pi_k, \qquad T_k ( X, Y) = k^{-1} \Pi_k \nabla^{L^k
  \otimes A}_X \nabla^{L^k
  \otimes A}_Y \Pi_k .$$
By Proposition \ref{prop:versionlinfty}, $T_k (X, Y)$ belongs to $\lag ^+(A)$
so it belongs to the Toeplitz algebra \eqref{eq:Toeplitz_Pi} associated to
$\Pi$. By Lemma \ref{lem:symb_der}, its symbol is $\pi \rho (X) \rho (Y)
\pi$.  
\begin{prop} \label{prop:subsymbolic}
  For any $f,g \in \Ci (M)$,
\begin{gather} \label{eq:lastproofeq} 
  T_{k} (f) T_k (g) = T_k (fg) + k^{-1} T_k (  X,Y) + \bigo (  k^{-2})
\end{gather}
where $X$ and $Y$ are the Hamiltonian vector fields of $f$ and $g$ respectively.
Consequently
$ i [T_k (f) , T_k (g) ] = k^{-1} T_k ( \om (X,Y) ) + \bigo ( k^{-2})$.
\end{prop}
\begin{proof}
 By a straightforward computation,
  we have
$$ \Pi f \Pi g \Pi = \Pi fg \Pi + \Pi [f, \Pi ] [g, \Pi ] \Pi $$
By lemma \ref{lem:ledernier}, $\Pi [f, \Pi] [g, \Pi] \Pi$ belongs to $\lag_2
(A)$ and its symbol is
\begin{gather*} 
\si_2 ( \Pi [f, \Pi] [g, \Pi] \Pi )   = - \pi [\rho (X) , \pi] [ \rho (Y),
\pi] \pi
\end{gather*}
\cB Observe that $\rho (X) |_x$ is an odd operator of $\D (T_xM)$. Since $F_x$ has
a definite parity,  every endomorphism of
$F_x$ is even. So  $\pi \rho (X)
\pi |_x$ is at the same time odd and even, so $\pi \rho (X) \pi =0$. \clb Similarly $\pi \rho (Y) \pi =0$, so
\begin{gather*} 
\si_2 ( \Pi [f, \Pi] [g, \Pi] \Pi )   =    \pi \rho (X) \rho (Y) \pi =  \si_2 ( k^{-2} \Pi \nabla_X^{L^k \otimes A} \nabla_Y^{L^k \otimes A} \Pi )
\end{gather*}
which proves \eqref{eq:lastproofeq}. Consequently, the rescaled commutator of $T_k(f)$,
$T_k(g)$ satisfies
$$ k [T_k (f), T_k (g) ] = T_k (X, Y) - T_k (Y,X) + \bigo (k^{-1}), $$
so its symbol is $\pi ( [\rho (X), \rho (Y) ] \pi$. To conclude the proof, we
simply use that 
\begin{gather} \label{eq:commutator_rho}
[\rho (X), \rho (Y) ] = \tfrac{1}{i} \om (X,Y),
\end{gather}
as follows easily by using a basis $U_i$ of $T_x^{1,0}M$  such that
$\frac{1}{i} \om (U_i, \con{U}_j) = \delta_{ij}$,   $\rho ( U_i) = - a_i^*$ and $\rho ( \con{U}_i ) = a_i$.
\end{proof}

\subsection{Proofs of Theorems \ref{theo:dim_landau},
  \ref{theo:Toeplitz_Landau} and \ref{theo:ladder}} \label{sec:ledernier} 

In this last section, we complete the proof of the theorems stated in the introduction, a
generalization actually since we will consider more general projectors.

Let $\Pi \in \mathcal{L}^{\infty}(A) \cap \mathcal{L}^{+} (A)$ be a self-adjoint projector with
symbol $\pi= \pi_m \otimes \op{id}_A$, where $\pi_m$ is the projector of $\D (TM)$
onto $\D_m (TM)$. Such an operator exists by Theorem
\ref{theo:constr-proj} and Remark \ref{sec:rem_linfty}.
Alternatively, the projector $\Pi = ( \Pi_{m,k})$ onto the $m$-th Landau level defined in
\eqref{eq:projecteur_fibre_auxiliaire}  has the expected properties  \cite[Theorems 5.2,
5.3]{oim_copain}. 

By Theorem \ref{theo:dim_general}, the dimension of $\Hilb_k= \op{Im}( \Pi_k)$ is
$$  \op{dim} \Hilb_k = \int_M \op{ch} ( L^k \otimes A \otimes \D_m(TM) ) \; \op{Td}
(M)$$
when $k$ is sufficiently large, which implies Theorem  \ref{theo:dim_landau}. 

Define the Toeplitz algebra
$$ \mathcal{T}^{\infty} = \{ P \in \lag^{\infty}(A) \cap \lag^{+}
(A)/\; \Pi P \Pi = P \}. $$
Clearly $\mathcal{T}^{\infty}$ is contained in the
Toeplitz algebra $\mathcal{T}$ defined in \eqref{eq:Toeplitz_Pi}, and by
assertion 2 of Proposition \ref{prop:versionlinfty}, the difference is rather
small: for every $ P \in
\mathcal{T}$, there exists $P' \in \mathcal{T}^{\infty}$ unique modulo $\biginf
( k^{-\infty})$ such that $P' = P  + \bigo ( k^{-\infty})$. Recall the symbol
map $\tau_0$ introduced in Theorem \ref{theo:toeplitz} and denote by $\tau$
its restriction to $\mathcal{T}^{\infty}$
$$\tau : \mathcal{T}^{\infty} \rightarrow \Ci ( M , \op{End} ( \D_m (TM) \otimes A))
.$$
We could as well consider the maps $\tau_q$ with $q \geqslant 1$ but we will limit ourselves to
$\tau_0$. By Theorem \ref{theo:toeplitz} and Assertions 2, 4 of Proposition
\ref{prop:versionlinfty}, $\tau$ is onto and its kernel is $k^{-1}
\mathcal{T}^{\infty}$. It follows as well from Theorem \ref{theo:toeplitz}
that for any $P,Q \in \mathcal{T}^{\infty}$
\begin{gather*} 
\tau (PQ) = \tau ( P) \tau (Q), \quad \| P_k \| = \sup_{x \in M} \|
\tau(P)_x\| + \op{o} (1)  \\
P_k (x,x) = \Bigl( \frac{k}{2\pi} \Bigr)^n \op{tr} ( \tau ( P)_x) + \bigo (k^{-1})
\end{gather*}
Choose any connection on $A$. By Proposition \ref{sec:rem_linfty}, for any $ f \in \Ci ( M , \op{End} A)$,
$p \in \N$ and vector fields $X_1$, \ldots, $X_{2p}$ of $M$, the operator 
\begin{gather} \label{eq:toeplitz_mult_der}
T_k ( f, X_1, \ldots, X_{2p}) = k^{-p} \Pi_k f \nabla_{X_1}^{L^k \otimes A}
\ldots \nabla_{X_{2p}}^{L^k \otimes A}
\Pi_k 
\end{gather}
belong to $\mathcal{T}^{\infty}$ and \cB by Lemma \ref{lem:symb_der} \clb its $\tau$-symbol is
$ ( \pi_m \rho(X_1) \ldots
\rho(X_{2p})\pi_m) \otimes  f $. By Lemma \ref{lem:algebraic},  these symbols
span  $\Ci ( M ,
\op{End} ( \D_m (TM) \otimes A))$ as a vector space, we deduce that any  $P \in \mathcal{T}^{\infty}$ is of the form
$$ P_k = \sum_{\ell =0 }^N  k^{-\ell} P_{\ell,k} + \bigo ( k^{-(N+1)}) , \qquad
\forall N \in \N$$
where for any $\ell$, $(P_{\ell,k})_k$ is a finite sum of operators such as
\eqref{eq:toeplitz_mult_der}. So in the case where the auxiliary bundle $A$ is trivial, $\mathcal{T}^{\infty}$ is the space $\mathcal{T}^{\op{sc}}_m$ defined in
the introduction. Last assertion of Theorem \ref{theo:Toeplitz_Landau} follows
from Proposition \ref{prop:subsymbolic}.

\begin{lemme}  \label{lem:algebraic} $ $
  \begin{enumerate}
  \item $\op{End} \D_m(T_xM)$ is spanned as a vector space by the
    $$\pi_m   (x) \rho (X_1) \ldots \rho (X_{2p}) \pi_m (x) $$ where $p \in \N$ and $X_1$,
    \ldots, $X_{2p} \in T_xM$.  
\item  The linear map $\Psi_x : T^{(1,0)} _xM \otimes T^{(0,1)} _xM \rightarrow \op{End} \D_m
  (T_xM)$  such that $\Psi (U \otimes
  \con{V})= \pi_m (x) \rho (U)
  \rho (\con{V}) \pi_m (x)$, is injective when $m \geqslant 1$.  
  \end{enumerate}
 \end{lemme}
 \begin{proof} Introduce a
    basis  $(U_i)$  of $T^{1,0}_xM$ such that $\frac{1}{i} \om ( U_i ,
    \con {U}_j) = \delta_{ij}$ and let $z_i = \frac{1}{i} \om ( \cdot,
    \con{U}_i)$. So $ \D
(T_x M) = \C [\con{z}_1 , \ldots, \con{z}_n]$, $\rho (
U_i) P = -  \con{z}_i P = - a_i^* P$ and $ \rho ( \con{U}_i)  P = \partial P / \partial
\con{z}_i = a_i P$. For any $\al $, $\be$ in $\N^n$, $f_{\al, \be}=
(\al!)^{-1} (a^*)^{\be}   a^\al$ satisfies $f_{\al,\be} (\con{z}^\al ) =
\con{z}^\be$ and for any $\ga \in \N^n$ such that $\ga \neq \al$ and $| \ga | = |
\al |$, $f_{\al,\be} (\con{z}^\ga) = 0$.  So the family $\pi_m(x) f_{\al, \be}
\pi_m (x)$, $| \al | = | \be | =m$ is a basis of $\op{End} F_x$.  This proves
the first assertion. The second one is the fact that the restrictions to $\D_m ( T_xM)$ of the
  $ \con{z}_i \con{\partial}_j$, $i, j =1, \ldots , n$ are linearly
  independent.  
 \end{proof}

 Equation \eqref{eq:bm_1-f-g} and Part 2 of Remark \ref{rem:Toeplitz} are consequences on the following
 Proposition.  
 \begin{prop} \label{prop:fonction}
  For any $f \in \Ci (M)$, with Hamiltonian vector field $X$, we have 
  \begin{itemize}
  \item if $n =1$ or $m=0$,
    $ T_k (f)^2 - T_k (f^2)  =  -k^{-1} ( \frac{1}{2} + m)T_k ( |X|^2) + \bigo ( k^{-2}).$ 
        \item if $n \geqslant 2$ and $m \geqslant 1$, $ X \neq 0$ implies
          that there exists no function $h \in \Ci (M)$ such that $T_k (f)^2 -
          T_k ( f^2) =
          k^{-1} T_k (h) + \bigo ( k^{-2})$.
        \end{itemize}
\end{prop}
\begin{proof}  By Proposition
  \ref{prop:subsymbolic},
  $ T_k (f)^2 = T_k (f^2) + k^{-1} S_k $ where $S_k$ is a Toeplitz operator
  with symbol $\pi \rho(X)^2 \pi$.  
  Writing $X = U + \con {U}$ with $U \in T^{1,0}_xM$, we have 
  $$ \pi_m  \rho (X)^2 \pi_m = \pi_m \bigl(  \rho (U) \rho (\con{U}) + \rho
  ( \con {U} ) \rho (U) \bigr) \pi_m  = 2 \Psi (U \otimes \con{U}) - \tfrac{1}{2} g
  (X,X)$$
  where $\Psi$ is the bundle map introduced in Lemma \ref{lem:algebraic}, and we have used first that $\pi_m \rho ( U) ^2 \pi_m =0 =\pi_m
  \rho(\con{U})^2 \pi_m $ and then that $[ \rho(\con{U}) , \rho (U)] = i \om (
  U, \con {U}) = -\frac{1}{2} g (X,X)$ by \eqref{eq:commutator_rho}.  Now if $m=0$, then $\Psi =0$. If
  $n=1$, then $2 \Psi ( U \otimes \con{U}) = - m g (X,X)$. This
proves the first assertion. For the second one, it
  suffices to prove that if $n \geqslant 2$, $m \geqslant 1$ and $X(x) \neq
  0$, then $\Psi ( U \otimes \con{U})_x \in \op{End}( \D_m ( T_xM))$ is not
  scalar. This follows from the fact that $\Psi_x$ is injective, so that $\Psi_x
  ( \al ) $ is scalar only when $ \al$ is a multiple of $\sum_{i=1}^{n} U_i \otimes
  \con{U}_i$, which never happens for $\al = U \otimes \con{U}$ when $n
  \geqslant 2$ and $U \neq 0$. 
\end{proof}
\clb

Let us prove now Theorem \ref{theo:ladder}. Introduce a quantization $( \Hilb_{F,k})$  of $(M,L)$
twisted by $F = \D_m (TM) \otimes A$. We can adapt the definition
\eqref{eq:def_w_k} of $W_k$ with the auxiliary bundle $A$ by setting
\begin{gather} \label{eq:def_w_k_A}
\begin{split}
  W_k :  \Ci ( M , L^k \otimes A ) \rightarrow \Ci (
  M, L^k \otimes F ), \qquad k \in \N\\
  W_k = 
  R_m D_{G^{\otimes (m-1)} \otimes A ,k} \circ  D_{G^{\otimes
    (m-2)}\otimes A,k} \circ \ldots \circ D_{G \otimes A, k} \circ D_{A,k}
\end{split}
\end{gather}

\begin{lemme}
  The operator $(V_k =  \frac{1}{m!} k^{-\frac{m}{2}} \Pi_{F,k} W_k \Pi_{k}$,
  $k \in \N$) belongs to $ \lag ^{\infty} ( A, F)$, has the same parity as $m$ and its symbol $\si_0(V)$ viewed as a
  morphism from $\D(TM) \otimes  A$ to $\D(TM) \otimes F$ is given by
  $$ \forall \, f \in \D_p(TM), \; \forall \, a \in A, \; \si_0(V) ( f \otimes a)
  = \begin{cases}  1 \otimes f \otimes a  \text{ if $p = m $} \\ 0 \text{
    otherwise} 
\end{cases}$$
\end{lemme}

\begin{proof} 
We claim that for any even (resp. odd) operator $P \in \lag^{\infty} (B, A) $,
$(k^{-\frac{1}{2}}D_{A,k} \circ P )$ belongs to $\lag^{\infty}  ( B, A
\otimes G) $, is odd (resp. even) and its symbol is $\varphi_A \circ \si_0 ( P)$ where
$$\varphi_A = \sum_ i a_i \otimes \con{z}_i \otimes
\op{id}_A \in \symb (TM) \otimes G \otimes \op{End} A  .$$
Here we have introduced an orthonormal frame $(
\partial_i)$ of $T^{1,0}M$, $(z_i)$ is the dual frame of $(T^{1,0}M)^*$ and
$a_i =\partial_{\con{z}_i}$ is the annihilation operator. \cB This follows from
Lemma \ref{lem:symb_der} by writing $D_{A,k} s = \sum_i \con{z}_i \otimes
\nabla_{\con{\partial}_i} s$. \clb

Consequently, 
$k^{-\frac{m}{2}}R_m D_{A \otimes G^{m-1}} \circ \ldots \circ D_{A,k} \circ \Pi_k
$ belong to $\lag ( A, F)$ with symbol $\varphi_{A}^m \circ \pi_m$ where
$\varphi_A^m $ is the morphism from $\D ( TM) \otimes A$ to $\D ( TM) \otimes \D_m (TM) \otimes A$ given by 
$$ \varphi_A^m = \sum_{i_1, \ldots, i_m =1}^{n} a_{i_1} \ldots a_{i_m} \otimes
(\con{z}_{i_1} \ldots  \con{z}_{i_m})  \otimes \op{id}_A $$
Computing $\varphi_A^m ( \con{z}^{\be} )$ for any multi-index $\be$, we show
that for any $f \in \D_p(TM)$ and $ a\in A$,  $ \varphi_A^m (  f
\otimes a) = m ! (1 \otimes f \otimes a )$ when $p=m$ and $0$ otherwise.   
\end{proof}

The symbol $\sigma_0 (V)$ is exactly the symbol $\rho$ introduced in Lemma
\ref{lem:unitary-symbole}. Now Theorem \ref{theo:ladder} and the remark on the
unitarization of $V_k$ follows by the same proof as Theorem
\ref{theo:unitary_equivalence}.


\section{Appendix} \label{sec:appendix}

In this appendix we discuss the known results for surfaces with constant
curvature. Two important features appear for the negatively curved surfaces: only the lower part of the spectrum consists of
Landau levels, and moreover, there is  an isomorphism between the $m$-th Landau level and
the first level of a Laplacian twisted by the $m$-th power of the complex determinant
bundle. 

These results appeared in the physics literature, cf. in
 particular  \cite{IeLi94} for the case with surface with genus $\geqslant 2$.
 A more recent mathematical reference is \cite{Te06}.

\subsubsection*{The plane \cite{La30}}  Consider a quantum particle confined in a two
dimensional plane $(x,y)$ and subject to a constant magnetic field
perpendicular to this plane. Its Hamiltonian is the operator
\begin{gather} \label{eq:laplacian_landau}
H = -\tfrac{1}{2}
 (\nabla_x^2 + \nabla_y^2) \quad \text{ with } \quad \nabla_x = \tfrac{\partial}{ \partial x} +
 \tfrac{i}{2} B y,\;  \nabla_y = \tfrac{\partial}{\partial y} - \tfrac{i}{2}
 Bx .
\end{gather}
$B$ is a  positive constant
representing the strength of the magnetic field.
The spectrum of $H$ is $B(\frac{1}{2} + \N )$ and the Landau levels $\Hilb_m =
\op{ker} ( H - B( \frac{1}{2} + m))$ are given in terms of the ladder
operators $\nabla_z = \nabla_x - i \nabla_y$, $\nabla_{\con{z}} = \nabla_x + i
\nabla_y$  by 
\begin{gather} \label{eq:landau_level}
\Hilb_0 = \op{ker} ( \nabla_{\con z}), \qquad   \Hilb_m = ( \nabla_z)^m
\Hilb_0 , \quad m \geqslant 1 .
\end{gather}

\subsubsection*{Surface with constant curvature, \cite{IeLi94}, \cite{Te06}} 
Let $M$ be a compact orientable surface with a Riemannian metric having a constant Gauss curvature $S$. Introduce a
Hermitian line bundle $ L \rightarrow M$ with a connection $\nabla : \Ci (M, L )
\rightarrow \Om^1 (M,L)$. Assume that the curvature satisfies
\begin{gather} \label{eq:constant_magnetic}
i \op{curv} ( \nabla)  = B
\op{vol}_g 
\end{gather}
where $B$ is a non-zero constant and $\op{vol}_g$ is the
Riemannian volume. Choosing the convenient orientation for $M$, we can assume
that $B$ is positive. The quantum Hamiltonian is the Laplacian $ \Delta := \frac{1}{2} \nabla^* \nabla $ acting
on sections of $L$. Then denoting its eigenvalue by $0 \leqslant \la_0 < \la_1 <
\ldots $, it is known that 
\begin{gather} \label{eq:spec_magentic_surface}
  \la_m = B (\tfrac{1}{2} + m ) + S \tfrac{ m (m+1)}{2} \qquad \text{ if }
  B + m S >0
\end{gather}
For a sphere or a torus, $S \geqslant 0$, and these formulas describe the
whole spectrum. If the genus of $M$ is larger than $2$, then $S<0$ and the
condition $B + m S >0$ is satisfied only for a finite number of $m$. In this
case, it is not reasonable to expect an explicit formula for the other
eigenvalues. Indeed, if $S=-1$ and $L = K^r$ with  $K$ the canonical
bundle of $M$ and $r$ a positive integer, then
\eqref{eq:spec_magentic_surface} gives the first $(r+1)$ eigenvalues, and for
any $n \in \N$, $\la_{r+n} = \la_r + \tfrac{1}{2} \mu_n$, where $\{\mu_n, \; n
\in \N\}$ is the spectrum of the Laplace-Beltrami operator of $M$. This
latter spectrum depends in an essential way on the metric. Indeed, 
 by Huber's theorem \cite[Theorem 9.2.9]{Buser}, $\{ \mu_n\}$ determines the
 length spectrum of $M$.

 The
multiplicity of the first eigenvalues is
equal to: 
\begin{gather} \label{eq:mult} 
  \op{mult} ( \la_m) = B \frac{ \op{vol}_g (M)}{2\pi} +  (\tfrac{1}{2} + m)
  \chi (M) \qquad \text{ if } B +  (m+1)S >0.
\end{gather}
Here $\chi (M)$ is the Euler characteristic of $M$ and observe that $B
\op{vol}_g (M) / (2\pi)$ is the degree of $L$, so it is an integer.
\eqref{eq:mult} follows from a description of the corresponding eigenspace $\Hilb_m = \ker ( \Delta -
\la_m)$ similar to \eqref{eq:landau_level}, that we explain below. 

\subsubsection*{Proof of formulas \eqref{eq:spec_magentic_surface} and \eqref{eq:mult}}

To start with, we do not assume that the Gauss curvature $S$ and the function $B$ defined in
\eqref{eq:constant_magnetic} are constant. We choose any orientation on $M$.
Let $j$ be the complex
structure of $M$ compatible with $g$, i.e. $g(jX, jY) = g(X,Y)$ for any
tangent vectors $X$, $Y \in T_pM$  and $ X
\wedge jX >0$ if $X \neq 0$. Since we are in real dimension 2, $j$ is
integrable. Furthermore the associated volume form $\op{vol}_g$ is the
symplectic form  $\om
(X, Y) = g(jX, Y)$.

$L$ has a natural holomorphic structure such that
its $\con{\partial}$-operator is $\nabla^{0,1}$.  We denote it by
$\con{\partial}_L: \Ci ( L) \rightarrow \Ci ( L \otimes \con{K})$ with $K =
(T^*M)^{1,0}$ the canonical bundle.  The curvature of $\nabla$ has the form $\frac{1}{i} B \om$
with $B \in \Ci ( M , \R)$. The canonical bundle has a natural metric induced
by $g$. Its Chern connection, that is its connection compatible with both its
metric and holomorphic structure, has curvature $i S \om$ where $S$
is the Gauss curvature.
\begin{theo} \label{theo:appendix}
  The following identities holds:
 \begin{enumerate}
  \item Weitzenb\"ock formula: $\Delta_L = \con{\partial}^*_L \con{\partial}_L + \frac{1}{2} B$.
    \item Bosonic commutation relation: $  \con{\partial}_L \con{\partial}_L^* = \con{\partial}_{L\otimes
        K^{-1}} ^*  \con{\partial}_{L\otimes
        K^{-1}} + (B + S)$.
    \end{enumerate}
  \end{theo}
The Weitzenb\"ock formula is a classical relation, it holds more generally on
K\"ahler manifolds. We call the second formula the bosonic commutation
relation because it replaces the canonical commutation relation  satisfied by the creation/annihilation operators $[a, a^{*} ]
=1$. In this formula,  we identify $\con{K}$ with $K^{-1}$ through
      the metric so that the operators $ \con{\partial}_L \con{\partial}_L^*$
      and $\con{\partial}_{L\otimes
        K^{-1}} ^*  \con{\partial}_{L\otimes
        K^{-1}}$ act on the same space $\Ci ( L\otimes \con{K}) = \Ci ( L
      \otimes K^{-1})$. A similar formula were obtained in \cite[Proposition 9]{Te06} for the
      same purpose of computing the spectrum of $\Delta_L$. 
\begin{proof}[Proof of the bosonic identity]
  Introduce a local holomorphic frame $s$ of
  $L$ and a complex coordinate $z$ on $M$. We have first $\con{\partial}_L ( f s) =
  f_{\con{z}} \; s \otimes d \con{z}$. To compute the adjoint, recall that the
  scalar products of $\Ci (L)$ and $\Ci( L \otimes \con{K})$ are defined by
  integrating the pointwise scalar products against the volume form. Write $\om
  = i h dz \wedge d \con{z}$ and $|s|^2 = e^{-
    \varphi}$ with $h$ and $\varphi$ real valued functions. Then $|d \con z
  |^2  = h^{-1}$ and a direct computation leads to
  $$\con{\partial}_L^* ( f s \otimes d\con{z} ) = h^{-1} ( -
  f_z + f \varphi_z ) s.$$
  With the identification $\con{K} \simeq K^{-1}$, we have $
  d\con{z}  = h^{-1} (dz)^{-1}$. We deduce that 
  $$ \con{\partial}_L \con {\partial}_L^* ( f s \otimes (dz)^{-1}) = h^{-1} \bigl( -
  f_{z \con{z} } + f_{\con{z}} ( \varphi_z - \tfrac{h_z}{h}) + f ( \varphi_{z
    \con{z}} - \partial_{\con{z}} ( \tfrac{h_z}{h})) \bigr) s \otimes (dz )^{-1}$$
  Similar computations by using $|s \otimes dz^{-1} |^2 = h e^{-\varphi}$ leads to
  $$ \con{\partial}_{L\otimes
        K^{-1}} ^*  \con{\partial}_{L\otimes
        K^{-1}} ( f s \otimes (dz)^{-1}) = h^{-1} \bigl( - f_{z \con{z}} +
      f_{\con{z}} ( \varphi_z - \tfrac{h_z}{h} ) \bigr) s \otimes (dz)^{-1} .$$
To conclude, observe that $B = h^{-1} \varphi_{z \con z}$ and $S = - h^{-1}
\partial_z \partial_{\con z} \, \ln h $.
\end{proof}

From now on, we will assume that $B$ and $S$ are constant. With the
Weitzenb\"ock formula, we pass directly from the spectrum of $\Delta _L$ to the one
of $\con{\partial}^*_L \con{\partial}_L$. We can use the Bosonic relation
exactly as it is usually done with the Landau Hamiltonian, cf. proof of
Proposition \ref{prop:linearlandau}.  We deduce
that for any $\la \neq 0$, $\la $ is an eigenvalue of $\con{\partial}^*_L
\con{\partial}_L$ if and only if $\la - (B+S)$ is an eigenvalue of
$\con{\partial}^*_{L\otimes K^{-1}} \con{\partial}_{L\otimes K^{-1}}$.
Moreover the eigenspaces have the same dimension. Indeed $\con{\partial}_L$
restricts to an isomorphism
$$ \op{Ker} ( \la - \con{\partial}^*_L
\con{\partial}_L ) \rightarrow \op{ker} ( \la - (B+S) -
\con{\partial}^*_{L\otimes K^{-1}} \con{\partial}_{L\otimes K^{-1}} ) $$
with inverse the restriction of $\la^{-1} \con {\partial}_L^*$.
Besides this, $\ker (\con{\partial}^*_L
\con{\partial}_L)$ is the space $H^0 ( L)$ of holomorphic sections of $L$. By
Riemann-Roch theorem, $H^0 (L)$ has dimension $d + \frac{1}{2} \chi (M)$ if
the degree $d = B \op{Vol}(M)/(2\pi)$ of $L$ is larger than $- \chi (M)$.

To summarize,  when $B$ is sufficiently large, $0 $ is an eigenvalue of
$\con{\partial}^*_L \con{\partial}_L$ with multiplicity equal to $ B
\op{Vol}(M) / (2 \pi) + \frac{1}{2} \chi(M)$, and the remainder of the
spectrum is identical with the spectrum of $(B+S) + \con{\partial}^*_{L\otimes K^{-1}}
\con{\partial}_{L\otimes K^{-1}}$, multiplicities included. We can iterate this argument and deduce by
induction  the formulas \eqref{eq:spec_magentic_surface},
\eqref{eq:mult} giving the first eigenvalues of $\Delta_{L}$ with their multiplicity.
Since $\op{deg} ( L \otimes K^{-1})  = \op{deg} (L) + \chi (M)$,  we can repeat ad infinitum this
argument when $\chi (M) \geqslant 0$ and obtain the whole spectrum of $\Delta_L$; whereas for $\chi (M)
<0$, only a finite number of iterations is possible.  

\bibliographystyle{plain}
\bibliography{biblio}

\bigskip

\noindent
{\bf Laurent Charles} \\
Sorbonne Universit\'e, Universit\'e de Paris, CNRS \\
Institut de Math\'ematiques de Jussieu-Paris Rive Gauche \\
F-75005 Paris, France \\
{\em E-mail:} \texttt{laurent.charles@imj-prg.fr}\\

\end{document}